\documentclass[a4paper]{article}
\usepackage[english]{babel}
\usepackage[T1]{fontenc}
\usepackage[latin1]{inputenc}
\usepackage{amssymb}
\usepackage{amsmath}
\usepackage{amsthm}
\usepackage{amscd}
\usepackage{mathrsfs}
\usepackage{bm}
\usepackage{tikz-cd}
\usepackage{pgf,tikz}
\usepackage{mathrsfs}
\usetikzlibrary{arrows}
\usepackage{pgfplots}
\pgfplotsset{compat=1.15}

\DeclareMathOperator{\Conf}{Conf}

\DeclareMathOperator{\id}{id}

\DeclareMathOperator{\height}{ht}
\DeclareMathOperator{\effsc}{effsc}

\DeclareMathOperator{\Alt}{Alt}
\DeclareMathOperator{\Hom}{Hom}
\DeclareMathOperator{\Sh}{Sh}
\DeclareMathOperator{\sgn}{sgn}
\DeclareMathOperator{\Sq}{Sq}
\DeclareMathOperator{\tr}{tr}
\DeclareMathOperator{\Gl}{Gl}

\DeclareMathOperator{\PBl}{PBl}
\DeclareMathOperator{\w}{w}
\DeclareMathOperator{\gr}{gr}

\newtheorem{theorem}{Theorem}[section]
\newtheorem{proposition}[theorem]{Proposition}
\newtheorem{corollary}[theorem]{Corollary}
\newtheorem{lemma}[theorem]{Lemma}

\theoremstyle{definition}
\newtheorem{definition}[theorem]{Definition}

\theoremstyle{remark}
\newtheorem{remark}[theorem]{Remark}

\newcommand{\Ftwo}{\ensuremath{\mathbb{F}_2}}
\newcommand{\Fp}{\ensuremath{\mathbb{F}_p}}
\newcommand{\WB}[1]{\ensuremath{W_{ B_{#1} }}}
\newcommand{\WD}[1]{\ensuremath{W_{ D_{#1} }}}

\title{The mod $ 2 $ cohomology of the infinite families of Coxeter groups of Type $ B $ and $ D $ as almost Hopf rings}
\author{Lorenzo Guerra\\
\scriptsize Univ. Lille, CNRS, UMR 8524 -- Laboratoire Paul Painlev\'e, F-59000 Lille, France \\
\scriptsize lorenzo.guerra@univ-lille.fr}
\date{}

\begin{document}

\maketitle

\begin{abstract}
We describe a Hopf ring structure on the direct sum of the cohomology groups $ \bigoplus_{n \geq 0} H^* \left( \WB{n}; \Ftwo \right) $ of the Coxeter groups of Type $ \WB{n} $, and an almost-Hopf ring structure on the direct sum of the cohomology groups $ \bigoplus_{n \geq 0} H^* \left( \WD{n}; \Ftwo \right) $ of the Coxeter groups of Type $ \WD{n} $, with coefficients in the field with two elements $ \Ftwo $.
We give presentations with generators and relations, determine additive bases and compute the Steenrod algebra action. The generators are described both in terms of a geometric construction by De Concini and Salvetti and their restriction to elementary abelian $ 2 $-subgroups.
\end{abstract}

\section{Introduction}

The Coxeter groups of Type $ \WB{n} $ and $ \WD{n} $ are two infinite families of finite reflection groups.
Coxeter groups are traditionally described via Coxeter diagrams, i.e., graphs in which each edge $ e $ has a weight $ m_e \geq 3 $.
Given such an object, the associated Coxeter group has a generator $ s_v $ for every vertex $ v $, with relations of the form $ s_v^2 = 1 $, $ \left( s_v s_w \right)^{m_e} = 1 $ for every edge $ e = \left( v,w \right) $, and $ \left( s_v s_w \right)^2 = 1 $ if $ v $ and $ w $ are not connected by an edge. For an exhaustive introduction to the geometry and topology of these gropus we refer to Davis' book \cite{Davis}.
The reflection groups of Type $ \WB{n} $ and $ \WD{n} $ are the finite Coxeter groups associated with the Coxeter diagrams in Figure \ref{fig:diagrammi Coxeter} below.
\begin{figure}[h] \label{fig:diagrammi Coxeter}
\centering
\definecolor{ffffff}{rgb}{1.,1.,1.}
\definecolor{xdxdff}{rgb}{0.49019607843137253,0.49019607843137253,1.}
\begin{tikzpicture}[line cap=round,line join=round,>=triangle 45,x=1.0cm,y=1.0cm]
\clip(-0.5,-1.5) rectangle (10.5,1.5);
\draw (0.,1.)-- (1.,0.);
\draw (1.,0.)-- (0.,-1.);
\draw (1.,0.)-- (2.,0.);
\draw [dash pattern=on 1pt off 1pt] (2.,0.)-- (3.,0.);
\draw (3.,0.)-- (4.,0.);
\draw (6.,0.)-- (7.,0.);
\draw (7.,0.)-- (8.,0.);
\draw [dash pattern=on 1pt off 1pt] (8.,0.)-- (9.,0.);
\draw (9.,0.)-- (10.,0.);
\draw (6.373175468227362,0.6355038535886468) node[anchor=north west] {$ 4 $};
\begin{scriptsize}
\draw [fill=black] (0.,1.) circle (2.5pt);
\draw[color=black] (0.31764643520257363,1.1670236433212295) node {$t_0$};
\draw [fill=black] (0.,-1.) circle (2.5pt);
\draw[color=black] (0.31764643520257363,-1.0641412674425685) node {$t_1$};
\draw [fill=black] (1.,0.) circle (2.5pt);
\draw[color=black] (1.171874668701368,0.4077096579889685) node {$t_2$};
\draw [fill=black] (2.,0.) circle (2.5pt);
\draw[color=black] (2.1779656992666148,0.4077096579889685) node {$t_3$};
\draw [fill=black] (3.,0.) circle (2.5pt);
\draw[color=black] (3.241005278731781,0.4077096579889685) node {$t_{n-2}$};
\draw [fill=black] (4.,0.) circle (2.5pt);
\draw[color=black] (4.247096309297028,0.4077096579889685) node {$t_{n-1}$};
\draw [fill=black] (6.,0.) circle (2.5pt);
\draw[color=black] (6.16436412226099,0.4077096579889685) node {$s_0$};
\draw [fill=black] (7.,0.) circle (2.5pt);
\draw[color=black] (7.170455152826236,0.4077096579889685) node {$s_1$};
\draw [fill=black] (8.,0.) circle (2.5pt);
\draw[color=black] (8.176546183391482,0.4077096579889685) node {$s_2$};
\draw [fill=black] (9.,0.) circle (2.5pt);
\draw[color=black] (9.239585762856649,0.4077096579889685) node {$s_{n-2}$};
\draw [fill=black] (10.,0.) circle (2.5pt);
\draw[color=black] (10.245676793421895,0.4077096579889685) node {$s_{n-1}$};
\end{scriptsize}
\end{tikzpicture}
\caption{Diagrams of Type $ D_n $ (left) and $ B_n $ (right)}
\end{figure}

The goal of this paper is to provide an effective description of the mod $ 2 $ cohomology of these groups. Other authors have previously computed these cohomology groups. Most notably, Swenson, in his thesis \cite{Swenson}, adapted techniques used by Hu'ng \cite{Hung:87} and Feshbach \cite{Feshbach:02}, stemming from the analysis of the restriction maps to elementary abelian $ 2 $-subgroups, to compute generators and relations for the mod $ 2 $ cohomology algebra of a finite reflection group. However, his presentation is involved and intrinsically recursive. Borrowing ideas from \cite{Sinha:12} and \cite{Sinha:17}, we exploit additional structures to provide a simpler description of the cup product. Our approach also has the advantage of being more easily readable from the well-known chain-level geometric and combinatoric description of a resolution for Coxeter groups by De Concini and Salvetti \cite{Salvetti:00}.

The sequences of Coxeter groups of Type $ B $ and $ D $ have standard embeddings $ \WB{n} \times \WB{m} \rightarrow \WB{n+m} $ and $ \WD{n} \times \WD{m} \rightarrow \WD{n+m} $ that are, in a certain sense, compatible. The homomorphisms induced by these maps on mod $ 2 $ cohomology define a coproduct $ \Delta $. The cohomology transfer maps associated with them determine a product $ \odot $.

In the $ B $ case, the resulting structure is modeled on that of the symmetric groups, the Coxeter groups of Type $ A $, as described by Giusti, Salvatore, and Sinha \cite{Sinha:12} (mod $ 2 $) and by the author \cite{Guerra:17} (modulo odd primes). Together with the usual cup product $ \cdot $, these maps form a ring in the category of $ \Ftwo $-coalgebras, i.e., a Hopf ring over $ \Ftwo $.
More explicitly, given a ring $ R $, a (graded) Hopf ring over $ R $ is a graded $ R $-module with a coproduct $ \Delta $ and two products $ \odot $ and $ \cdot $ such that:
\begin{itemize}
\item $ \left( A, \Delta, \odot \right) $ is a Hopf algebra, with an antipode $ S $.
\item $ \left( A, \Delta, \cdot \right) $ is a bialgebras over $ R $
\item if $ x,y,z \in A $ and $ \Delta \left( x \right) = \sum_i x_i' \otimes x_i'' $, then the following distributivity formula hold:
\[
	x \cdot \left( y \odot z \right) = \sum_i \left( -1 \right)^{\deg(y) \deg(x_i'')}   \left( x_i' \cdot y \right) \odot \left( x_i'' \cdot z \right)
\]
\end{itemize}

In the $ D $ case, $ \Delta $, $ \odot $, and $ \cdot $ satisfy the last two axioms in the definition of a Hopf ring, and $ \Delta $ and $ \cdot $ form a bialgebra. However, as we will explain later, $ \Delta $ and $ \odot $ do not form a bialgebra.
We call this weaker structure an \emph{almost-Hopf ring} over $ \Ftwo $.
Due to this fact, the study of the cohomology of $ \WD{n} $, with the cup product, the transfer product, and the coproduct, is more complicated. The reader will find similarities between the cohomology of $ \WD{n} $ and that of the alternating groups, as described by Giusti and Sinha \cite{Sinha:17}.
Such structures stem from the seminal work of Strickland and Turner \cite{Strickland-Turner}, in which the authors discovered an Hopf ring structure on the cohomology of symmetric groups, even with generalized cohomology theories.

The main results of this paper are Theorem \ref{teo:Bn} and Theorem \ref{teo:Dn}, stated in Section \ref{sec:RelAdd}, consisting of a presentation in terms of generators and relations of the mod $ 2 $ cohomology of the Coxeter groups of Type $ B_n $ as a Hopf ring and of Type $ D_n $ as an almost-Hopf ring respectively. For clarity reasons, the relations are spread out in a few lemmas that allow us to state and prove the relations concerning coproduct, transfer product, and cup product separately. Building on these core theorems, we also describe convenient additive bases for the cohomology of these groups, with a graphical description via skyline diagrams similar to that obtained for the symmetric group in \cite{Sinha:12}, and compute the Steenrod algebra action. Our formulation of the cohomology of $ \WB{n} $ and $ \WD{n} $ yields without additional effort many features of these cohomology algebras. For instance, Hepworth's homological stability results \cite{Hepworth:16} in these particular cases follow directly.

We obtain our presentation via three technical tools.
First, we exploit De Concini and Salvetti's geometric combinatorial model to realize such (almost) Hopf rings structures at the cochain level. Specializing their construction to the families of groups of our interest, we observe that a resolution for $ \WB{n} $ is obtained from the symmetrized version of the planar level trees used by Giusti and Sinha \cite{Sinha:17} for the symmetric groups. The cohomology of $ \WD{n} $ is governed by an oriented version of these objects.
We describe cochain representatives of the structural maps in detail. Our treatment follows the paper cited above closely. However, we note that while the transfer product is realized very similarly to the $ \Sigma_n $ case, coproducts are more complicated and require the combinatorial operation of ``pruning'' symmetric planar level trees.
This cochain-level description allows us to quickly retrieve some of our relations and give a more geometric flavor to our generating classes. For instance, they can be interpreted as Thom classes in a suitable sense.

Second, we use the existence of well-behaved maps between $ \WB{n} $, $ \WD{n} $, and $ \Sigma_n $. These homomorphisms preserve parts of our structures. Therefore, we exploit them to build our presentations on the known result for the cohomology of the symmetric groups. We provide a cochain-level description of these morphisms, and we determine both their action on generators and their relations to the coproduct and transfer product.

Third, we reconcile with Swenson's approach, and we investigate restrictions to elementary abelian $ 2 $-subgroups. The mod $ 2 $ cohomology of finite reflection groups is known to be detected by this family of subgroups. We effectively compute the action of these restriction maps on our additive bases. The multiplicative structure on the cohomology of (the invariant subalgebras of) such subgroups is known. Thus, these calculations allow us to deduce cup product relations that would be otherwise difficult to obtain.

We organize the paper as follows.
After describing the structures on the cohomology of $ \WB{n} $ and $ \WD{n} $ in Section 2, we devote the following two sections to developing our geometric tools. In Section 3, we review De Concini and Salvetti's construction, and we specialize it to $ \WB{n} $ and $ \WD{n} $. In Section 4, we investigate the combinatorics of pruning operations, and we retrieve cochain-level representatives of our structural and connecting homomorphisms. Section 5 is devoted to our main theorems. We define generators and we discuss relations between them. In this context, we also deduce from our presentation additive bases, and we discuss the relations between the cohomology of Coxeter groups of Type $ A $, $ B $, and $ C $. We postpone the proofs of the presentation theorem and some cup product relations. In Section 6, we turn our attention to the restriction to elementary abelian $ 2 $-subgroups. We review relevant results from Swenson's thesis, compute restriction maps, and use them to complete the proof of our cup-product relations. Section 7 is devoted to completing the proof of our main theorems. In Section $ 8 $, we calculate the Steenrod algebra action. The paper ends with an appendix providing a self-contained proof of the additive basis in the $ D $ case, not used in any other parts of the paper.

\subsection*{Acknowledgements}

Most of the contents of this paper are part of the author's Ph.D. thesis, written at Scuola Normale Superiore in Pisa. The author acknowledges full support from this institution. The author indebited with his Ph.D. advisor, prof. Mario Salvetti, for his guidance, and also thanks prof. Dev Sinha for helpful comments.

\newpage
\tableofcontents
\newpage

\section{(Almost) Hopf ring structures for the cohomology of $ \WB{n} $ and $ \WD{n} $}\label{sec:Hopf rings}

We begin this paper by describing in detail how the desired algebraic structures on the cohomology of Coxeter groups of Type $ B $ and $ D $ are obtained.
Throughout this paper, we use several combinatorial descriptions of the groups $ \WB{n} $ and $ \WD{n} $. We refer to Chapter 8 in \cite{Bjorner-Brenti} for a thorough treatment, and we recall below what we need for our purposes.

With reference to Figure \ref{fig:diagrammi Coxeter}, we recall that there is an inclusion $ j_n \colon \WD{n} \hookrightarrow \WB{n} $ defined by $ t_0 \mapsto s_0s_1s_0 $ and $ t_i \mapsto s_i $ if $ i > 0 $.
$ \WB{n} $ can be seen as the group of signed permutation on $ n $ numbers, that is, the group of bijective functions $ f $ from the set $ \left\{ -n, \dots, -1, 1, \dots, n \right\} $ into itself that satisfy $ f \left( -i \right) = - f \left( i \right) $ for every $ 1 \leq i \leq n $. Hence $ \WB{n} $ is naturally a subgroup of $ \Sigma_{2n} $, the symmetric group on $ 2n $ objects.
The image of $ j_n $ is $ \WB{n} \cap \Alt \left( 2n \right) $, the intersection of $ \WB{n} $ with the alternating group $ \Alt \left( 2n \right) $, the subgroup of even permutations in $ \Sigma_{2n} $.
Note that $ \Sigma_n $ can be identified with the parabolic subgroup of $ \WB{n} $ generated by $ s_1,\dots, s_{n-1} $, corresponding to the signed permutations on $ \{-n,\dots,n\} $ that preserve signs. There is also a standard projection $ \WB{n} \to \Sigma_n $, of which the previous inclusion is a section, whose kernel is the normal subgroup generated by $ s_0 $.
We observe that this provides an isomorphism between $ \WB{n} $ and the wreath product $ \Ftwo \wr \Sigma_n $, a semi-direct product of $ \Ftwo^n $ and $ \Sigma_n $. Therefore, the inclusions $ \Sigma_n \times \Sigma_m \rightarrow \Sigma_{n+m} $ extend naturally to monomorphisms $ \WB{n} \times \WB{m} \rightarrow \WB{n+m} $. These inclusions are associative and commutative up to conjugation.

Let $ A_B = \bigoplus_{n \geq 0} H^* \left( \WB{n}; \Ftwo \right) $. We define a coproduct $ \Delta $ and two product $ \cdot $ and $ \odot $ on $ A_B $ in the following way:
\begin{itemize}
\item $ \Delta $ is induced by the obvious monomorphisms $ \WB{n} \times \WB{m} \rightarrow \WB{n+m} $
\item $ \odot $ is induced by the cohomology transfer maps associated with these inclusions
\item $ \cdot $ is the usual cup product
\end{itemize}

Due to the associativity and the commutativity of the natural inclusions, these morphisms define an almost-Hopf ring structure. This is a general fact, as noticed in \cite{Sinha:17}.
In this case, however, $ A_B $ is a full Hopf ring.

\begin{proposition}
$ A_B $, with these structural morphisms, is a Hopf ring.
\end{proposition}
\begin{proof}
The almost-Hopf ring axioms hold by Theorem 2.3 in \cite{Sinha:17}. It remains only to prove that $ \left( A_B, \Delta, \odot \right) $ form a bialgebra. This claim follows from the fact (compare with \cite{Sinha:12}, Section 3) that this diagram is a pullback of finite coverings for all $ n,m \in \mathbb{N} $, where $ \pi $ indicates the projections.

\begin{center}
\begin{tikzcd}[column sep=2.5cm]
\bigsqcup_{\substack{p+q=n\\r+s=m}} \frac{ E \left( \WB{n+m} \right) }{ \WB{p} \times \WB{q} \times \WB{r} \times \WB{s}} \arrow{r}{\bigsqcup \pi_{p+q,r+s}} \ar[overlay]{d}{\bigsqcup \pi_{p+r,q+s}} & \frac{ E \left( \WB{n+m} \right) }{ \WB{n} \times \WB{m} } \ar[overlay]{d}{\pi_{n,m}} \\
\bigsqcup_{k+l=n+m} \frac{ E \left( \WB{n+m} \right) }{ \WB{k} \times \WB{l} } \ar[overlay]{r}{ \bigsqcup \pi_{k,l} } & \frac{ E \left( \WB{n+m} \right) }{ \WB{n+m} } \\
\end{tikzcd}
\end{center}
\end{proof}
We remark that, since $ A_B $ with $ \Delta $ and $ \odot $ is a conilpotent bialgebra, the existence of the antipode comes for free. This antipodal morphism does not play a role in our treatment; thus, we will not discuss it further.

Similarly, we can construct an additional almost-Hopf ring structure on the cohomology of the Coxeter groups of Type $ D_n $. Indeed, on the direct sum $ A_D = \bigoplus_{n \geq 0} H^* \left( \WD{n}; \Ftwo \right) $, we can define a coproduct $ \Delta $ and two products $ \odot $ and $ \cdot $ as done for $ A_B $.
However, these do not make $ A_D $ a full Hopf ring because, as we will see later, $ \left( A_D, \Delta, \odot \right) $ fails to be a bialgebra.

With essentially the same proof used for $ A_B $, we can prove the following easy proposition, which follows again from \cite[Theorem 2.3]{Sinha:17}.
\begin{proposition}
$ A_D $, with the coproduct and the two products defined before, is an almost-Hopf ring over $ \Ftwo $.
\end{proposition}

As we remarked in the introduction, there is a similar result for the mod $ 2 $ cohomology of the symmetric groups, obtained by Giusti, Salvatore, and Sinha in \cite{Sinha:12}. We recall their statement here because we will build our computations upon it.
\begin{theorem}[\cite{Sinha:12}, Theorems 1.2 and 3.2] \label{teo:gruppi simmetrici}
$ A_\Sigma = \bigoplus_{n \geq 0} H^* \left( \Sigma_n; \Ftwo \right) $, together with a coproduct $ \Delta \colon A_\Sigma \rightarrow A_\Sigma \otimes A_\Sigma $ induced by the obvious inclusions $ \Sigma_n \times \Sigma_m \rightarrow \Sigma_{n+m} $, a product $ \odot \colon A_\Sigma \otimes A_\Sigma \rightarrow A_\Sigma $ given by the transfer maps associated with these inclusions, and a second product $ \cdot \colon A_\Sigma \otimes A_\Sigma \rightarrow A_\Sigma $ defined as the usual cup product, is a Hopf ring over $ \Ftwo $.

$ A_\Sigma $ is generated, as a Hopf ring, by classes $ \gamma_{k,n} \in H^{n \left( 2^k-1 \right)} \left( \Sigma_{n2^k}; \Ftwo \right) $ for $ k \geq 0, n \geq 1 $. The coproduct of these classes is given by the formula
\[
	\Delta \left( \gamma_{k,n} \right) = \sum_{l=0}^n \gamma_{k,l} \otimes \gamma_{k,n-l},
\]
the cup product of generators belonging to different components is $ 0 $ and
\[
	\gamma_{k,n} \odot \gamma_{k,m} = \left( \begin{array}{c} n+m \\ n \end{array} \right) \gamma_{k,n+m}.
\]
There are no more relations between these classes.
\end{theorem}
$ \gamma_{0,n} \in H^0(\Sigma_n;\Ftwo) $ is the unit of the algebra $ H^*(\Sigma_n; \Ftwo) $ under the cup product. For this reason, we will often denote it with the symbol $ 1_n $ throughout the paper.

\section{Review of a geometric construction of De Con\-ci\-ni and Salvetti and Fox--Neuwirth-type cell structures} \label{sec:geometria}

\subsection{De Concini and Salvetti resolution}

In this section, we recall a geometric construction introduced by De Concini and Salvetti in \cite{Salvetti:00}, which we will require to describe the generators of the Hopf ring under consideration.

Given a finite reflection group $ G \leq \Gl_n \left( \mathbb{R} \right) $, there is a natural hyperplane arrangement $ \mathcal{A}_G $ in $ \mathbb{R}^n $ associated with $ G $, whose hyperplanes are the fixed points sets of reflections in $ G $.
The choice of a fundamental chamber $ C_0 $ of $ \mathcal{A}_G $ gives rise to a Coxeter presentation $ \left( G, S \right) $ for $ G $, whose set of generators $ S $ is composed by reflections with respect to hyperplanes that are supports of a face of $ C_0 $.
Every finite Coxeter group arises this way.

For any $ F \subseteq \mathbb{R}^n $, we can define:
\[
	\mathcal{A}_F = \left\{H \in \mathcal{A}_G: F \subseteq H \right\}
\]
$ \mathcal{A}_F $ gives rise to a stratification $ \Phi \left( \mathcal{A}_F \right) $ of $ \mathbb{R}^n $, in which the strata are the connected components of sets of the form $ L \setminus \bigcup_{H \in \mathcal{A}_F, H \not\supseteq L} H $, where $ L $ is the intersection of some of the hyperplanes of $ \mathcal{A}_F $.
Let $ \mathbb{R}^\infty $ be the direct limit of $ \mathbb{R}^m $ under the inclusions $ \mathbb{R}^m \hookrightarrow \mathbb{R}^m \times \{0\} \subseteq \mathbb{R}^{m+1} $. For all $ m \in \mathbb{N} \cup \left\{ \infty \right\} $, there is a stratification $ \Phi_m $ (different from the product stratification) of the topological space:
\[
	Y^{(m)}_G = \mathbb{R}^n \otimes \mathbb{R}^m \setminus \bigcup_{H \in \mathcal{A}_G} \left( H \otimes \mathbb{R}^m \right) = \left( \mathbb{R}^n \right)^m \setminus \bigcup_{H \in \mathcal{A}_G} H^m
\]
The strata in $ \Phi_m $ are defined as sets of the form $ F_1 \times \dots \times F_k \times \dots $, with $ F_k \in \Phi \left( \mathcal{A}_{F_{k-1}} \right) $ for $ k \geq 1 $. Here we put, by convention, $ F_0 = \left\{ 0 \right\} $.
In what follows, if there is no ambiguity, we will use the simpler notations $ Y^{(m)} $ and $ Y $ to indicate $ Y^{(m)}_G $ and $ Y^{(\infty)}_G $ respectively.

De Concini and Salvetti construct a regular $ G $-equivariant CW-complex $ X \subseteq Y $ that is ``dual'' to the stratification $ \Phi_\infty $, in the sense that for every stratum $ F \in \Phi_\infty $ of codimension $ d $, there exist a unique $ d $-dimensional cell in $ X $ that intersects $ F $, and they intersect transversally in a single point. For $ m < \infty $, the intersection $ X^{(m)} $ of $ X $ with $ Y^{(m)} $ is a subcomplex of $ X $, whose cells are dual to strata in $ \Phi_m $.
This construction is done equivariantly, in the sense that for every stratum $ F \in \Phi_\infty $ and every $ g \in G $, if $ \varphi \colon D^d \rightarrow X $ is the cell dual to $ F $ in $ X $, then $ (g. \textunderscore ) \circ \varphi \colon D^d \rightarrow X $ is the cell dual to $ g.F $.
The authors then show that $ X $ is a $ G $-equivariant strong deformation retract of $ Y $. Since $ Y $ is contractible and $ G $-free, the quotient $ X/G $ is a cellular model for the classifying space $ B \left( G \right) $ and the cellular chain complex $ C_*^G = C_*^{CW} \left( X \right) $ is a $ \mathbb{Z}\left[G \right] $-free resolution of $ \mathbb{Z} $.

The strata of $ \Phi_\infty $ have a more compact combinatorial description in terms of the Coxeter presentation.
For every $ s \in S $ generating reflection for $ G $, we let $ H_s $ be the hyperplane fixed by $ s $. $ H_s $ divides the space $ \mathbb{R}^n $ into two semi-spaces $ H_s^+ $ and $ H_s^- $. We let $ H_s^+ $ be the semi-space that contains the chosen fundamental chamber $ C_0 $.
To a flag $ \underline{\Gamma} = (S \supseteq \Gamma_1 \supseteq \Gamma_2 \supseteq \dots \supseteq \Gamma_k = \varnothing) $ of subsets of $ S $  we can associate a stratum $ F $ of $ \Phi_\infty $ such that $ x = \left( x_1, \dots, x_n \right)  \in \left( \mathbb{R}^\infty \right)^n $ belongs to $ F $ if  and only if the following condition is satisfied for every $ s \in S $ and every $ r \geq 1 $:
\[
	\left( (x_1)_r, \dots, (x_n)_r \right) \in H_s \mbox{ if } s \in \Gamma_r, \mbox{ } \left( (x_1)_r, \dots, (x_n)_r \right) \in H_s^+ \mbox{ if } s \in \Gamma_{r-1} \setminus \Gamma_r
\]
Thus, to a couple $ \left( \underline{\Gamma}, g \right) $, where $ \underline{\Gamma} $ is a flag as before and $ g \in G $, we can associate the stratum $ g.F $ obtained from the above $ F $ by applying $ g $.
This construction yields an algebraic-combinatorial description of the cellular chain complex of $ X $. The main theorem of De Concini and Salvetti's paper is the following.

\begin{theorem}[\cite{Salvetti:00}, Section 3] \label{teo:Salvetti}
Let $ \left( G, S \right) $ be a finite Coxeter group, and consider the set:
\[
	\left\{ \left( \underline{\Gamma}, \gamma \right): \gamma \in G, \underline{\Gamma} = \left( \Gamma_1 \supseteq \Gamma_2 \supseteq \dots \supseteq \Gamma_k \supseteq \dots \right), \Gamma_1 \subseteq S, \exists k: \Gamma_{k} = \varnothing \right\}
\]
The function described above is a bijection between this set and the set of strata in $ \Phi_\infty $ (and thus, by duality, with the set of cells in $ X $).
The codimension of the stratum (and the dimension of the corresponding dual cell) associated with $ \left( \underline{\Gamma}, \gamma \right) $ is equal to $ \sum_{r=1}^\infty | \Gamma_r | $, and the action of an element $ g \in G $ on strata and cells corresponds to the function $ \left( \underline{\Gamma}, \gamma \right) \mapsto \left( \underline{\Gamma}, g\gamma \right) $.

Let $ c \left( \underline{\Gamma}, \gamma \right) $ be the cell dual to the stratum corresponding to $ \left( \underline{\Gamma}, \gamma \right) $. The boundary homomorphism in $ C_*^{CW} \left( X \right) $ is given by the formula
\[
	\partial c \left( \underline{\Gamma}, \gamma \right) = \sum_{i \geq 1} \sum_{\tau \in \Gamma_i} \sum_{\substack{\beta \in W^{\Gamma_i \setminus \left\{ \tau \right\}}_{ \Gamma_i } \\ \beta^{-1} \Gamma_{i+1} \beta \subseteq \Gamma_i \setminus \left\{ \tau \right\}}} \left( -1 \right)^{\alpha \left( \underline{\Gamma},i,\tau,\beta \right)} c \left( \underline{\Gamma}', \gamma \beta \right)
\]
where $ \alpha $ is an integer number easily computed in terms of $ \underline{\Gamma} $, $ i $, $ \tau $, $ \beta $, $ \Gamma'_k = \Gamma_k $ for $ k < i $, $ \Gamma'_i = \Gamma_i \setminus \left\{ \tau \right\} $ and $ \Gamma'_k = \beta^{-1} \Gamma_k \beta $ if $ k > i $, and $ W^{T}_{T'} $, for $ T' \subseteq T \subseteq S $ is the set of minimal length coset representatives for the parabolic subgroup $ W_{T'} $ in $ W_{T} $.
\end{theorem}

We remark that in the case of Coxeter groups of Type $ B $ or $ D $, minimal coset representatives are explicitly known. For a complete description, we refer, for instance, to \cite{Papi}.

\subsection{Alexander duality and Fox-Neuwirth complexes}

We recall an alternative description of $ C_*^G $. This description has been exposed in \cite{Sinha:13}, where it is investigated in much detail in the $ A_n $ case. As observed in that paper, for every $ 1 \leq m \leq \infty $, the strata of $ \Phi_m $
are the interiors of cells in a $ G $-equivariant cell structure on the Alexandroff compactification $ \left( Y^{(m)} \right)^+ = Y^{(m)} \cup \left\{ \infty \right\} $.

Denote its augmented ($ G $-equivariant) cellular chain complex with the symbol $ \widetilde{FN}^m_G $. Its cells are the closures $ e \left( F \right) $ of strata $ F \in \Phi_m $ (together with the base point $ \left\{ * \right\} $) and, from the construction of $ X^{(m)} $ as a CW-complex dual to $ \Phi_m $ (details in \cite{Salvetti:00}), $ e \left( F \right) $ is contained in the boundary of $ e \left( F' \right) $ if and only if the cell of $ X $ dual to $ F $ contains the cell dual to $ F' $ in its boundary. This fact implies that the complex $ \widetilde{FN}^m_G $ is, up to a shift of degrees, the dual of $ C^{CW}_* \left( X^{(m)} \right) $, at least modulo $ 2 $ (in general, there are differences in some signs due to orientations).
Explicitly, the closure in $ \widetilde{FN}^m_G $ of a stratum of dimension $ d $ correspond to the dual of a chain in $ C^{CW}_* \left( X^{(m)} \right) $ of dimension $ nm-d $. In the remaining section of this paper, we will always implicitly assume this shift, and we will grade $ \widetilde{FN}^m_G $ to match the corresponding dimension of the dual cell.

In particular, $ \widetilde{FN}^m_G $ calculates the cohomology of $ Y^{(m)} $ and is therefore acyclic up to dimension $ nm-2 $. Alternatively, we can see this, as explained in \cite{Sinha:13}, by observing that the Atiyah duality theorem implies that the Spanier dual of $ Y^{(m)} $ is $ \left( Y^{(m)} \right)^+ $.

Passing to the limit for $ m \rightarrow \infty $, we obtain an acyclic $ \Ftwo $-complex $ \widetilde{FN}_G \otimes \Ftwo $, dual to $ C_*^{CW} \left( X \right) \otimes \Ftwo $, for which a basis $ \left\{ e \left( S \right) \right\}_{S \in \Phi_{\infty}} $ is parametrized by strata in $ \Phi_{\infty} $. The degree of $ e \left( S \right) $ as a cochain of $ X $ is equal to the codimension of $ F $. This is an equivariant cochain model for $ E \left( G \right) $. In particular, the quotient $ FN_G \otimes \Ftwo = \widetilde{FN}_G/G \otimes \Ftwo $ calculates $ \tilde{H}^* \left( G; \Ftwo \right) $.
In the following, when we need to stress the Coxeter group $ G $ involved, we will use the heavier notation $ \Phi_{\infty,G} $ instead of $ \Phi_{\infty} $.

This description of the cochain complex $ FN_G $ calculating the cohomology of $ G $ fits particularly well with a chain-level interpretation of duality via intersection theory that we will occasionally use in proofs and that we briefly recall here.
Given a manifold $ X $ and an immersion $ i \colon W \to X $ of a codimension $ d $ manifold in $ X $, we say that a smooth singular chain in $ X $ is transverse to $ i $ if, for every simplex $ \sigma \colon \Delta^k \to X $ of the chain, $ \sigma $ is transverse on every face of $ \Delta^k $ and subface, in the sense of manifolds with corners. It can be proved that the subcomplex consisting of chains that are transverse to $ i $ is chain equivalent to the full one.
To every $ d $-dimensional singular simplex $ \sigma \colon \Delta^d \rightarrow X $ transverse to $ i $ we can associate the element $ \tau_W \left( \sigma \right) \in \Ftwo $ given by the mod $ 2 $ cardinality of $ \sigma^{-1} \left( W \right) $.
This procedure defines a cochain dual in the complex dual to the chain complex of singular chains transverse to $ i $. If $ i $ is a proper embedding, $ \tau_W $ is a cocycle and defines a cohomology class.
The most important constructions in cohomology can be understood geometrically using this model. In particular, if $ f \colon Y \rightarrow X $ is transverse to $ i $, then $ f^\# ( \tau_W ) = \tau_{f^{-1}(W)} $. The reader will find a complete reference of this approach to cohomology in \cite{Friedmann-Medina-Sinha}.

In our particular context, each stratum $ S \in \Phi_{\infty} $ defines such a cochain $ \tau_S $. We understand the coboundary of $ \tau_S $ as $ \tau_{\partial(S)} $, so we can identify $ \widetilde{FN}^*_G $, at least modulo $ 2 $, with the cochain complex spanned by $ \tau_S $ for strata $ S \in \Phi_\infty $. Suppose $ W \subseteq Y^{(\infty)}_G $ is a proper submanifold of codimension $ d $ obtained as a union of strata. In that case, its associated cochain $ \tau_W $ is the sum of $ \tau_S $ for strata $ S \subseteq W $ of minimal codimension, and $ \delta(\tau_W) = 0 $. If, in addition, the action of $ G $ preserves $ W $, then, passing to the quotient, its image $ \overline{W} \subseteq \frac{Y^{(\infty)}_G}{G} $ defines a Thom class represented in $ FN^*_G $ by the sum of strata contained in $ \overline{W} $. This construction is made precise in \cite[Definition 4.6]{Sinha:12}.

\subsection{The special case of Coxeter groups of Type B}

We conclude this section by further investigating the cases of our interest $ G = \WB{n} $ and, in the following subsection, $ G = \WD{n} $.
The strata of $ \Phi_m $ for the symmetric group $ \Sigma_n $ can be described in terms of leveled trees, as shown in \cite{Sinha:13} using ideas dating back to Vassiliev \cite{Vassiliev:92}. A straightforward adaptation of these ideas shows that, in the case of the Coxeter groups of Type $ B_n $, we can describe them in terms of symmetric leveled trees. This interpretation encodes geometrically and combinatorially the structure of $ \WB{n} $ as a wreath product of $ \Sigma_n $ with a cyclic group of order $ 2 $. Below we provide the precise definitions.

First, we observe that, since $ \WB{n} $ is generated by a set $ S = \left\{ s_0, \dots, s_{n-1}\right\} $ of $ n $ fundamental reflections as described in Figure \ref{fig:diagrammi Coxeter}, the Fox-Neuwirth complex $ \widetilde{FN}^*_{\WB{n}} $ has a $ \mathbb{Z}\left[ \WB{n} \right] $-basis $ \{e(\underline{a}\} $ indexed by $ n $-tuples of non-negative integer numbers $ \left( a_0, \dots, a_{n-1} \right) $.

The reflection arrangement associated with $ \WB{n} $ can be described as $ \mathcal{A}_{B_n} = \left\{ H_{i,j}^\pm \right\}_{1 \leq i < j \leq n } \cup \left\{ H_i^0 \right\} $ where:
\begin{align*}
	H_{i,j}^\pm &= \left\{ \underline{x} \in \mathbb{R}^n: x_i = \pm x_j \right\} \\
	H_i^0 &= \left\{ \underline{x} \in \mathbb{R}^n: x_i = 0 \right\}
\end{align*}
Moreover, $ s_0 $ can be identified with the reflection with respect to $ H_1^0 $ and, for every $ i > 0 $, $ s_i $ with the reflection with respect to $ H_{i,i+1}^+ $.
Thus the basis element corresponding to $ \underline{a} = \left( a_0, \dots, a_{n-1} \right) $ is described as the stratum
\begin{align*}
	e \left( \underline{a} \right) = \Big\{ &\left( \underline{x_1}, \dots, \underline{x}_n \right) \in \left( \mathbb{R}^\infty \right)^n: \forall 1 \leq i \leq n-1, \forall 1 \leq j \leq a_i: \left( x_i \right)_j = \left( x_{i+1} \right)_j,\\ &\left( x_i \right)_{a_i+1} < \left( x_{i+1} \right)_{a_i+1}, \forall 1 \leq k \leq a_0: \left( x_1 \right)_k = 0, \left( x_1 \right)_{a_0+1} > 0 \Big\}.
\end{align*}

Passing to the quotient by the action of $ \WB{n} $, we see that $ FN^*_{\WB{n}} $ has a $ \mathbb{Z} $-basis constituted by $ \left[ a_0: \dots: a_{n-1} \right] = \left[ e \left( a_0, \dots, a_{n-1} \right) \right] $.

The differential on $ FN^*_{\WB{n}} $ is complicated, but it is combinatorially accessible via a description of its basis in terms of trees.
\begin{definition}\label{def:symmetric trees}
A \emph{signed depth-ordering} is a sequence of labeled inequalities of the form $ \Gamma = (0 <_{a_0} i_1 <_{a_1} \dots <_{a_{n-1}} i_n) $, where $ i_k \in \{-n,\dots,-1,1,\dots,n\} $ for all $ 1 \leq k \leq n $, and these indices have pairwise different absolute values. By convention, we let $ i_0 = 0 $.

A \emph{planar level tree} is a planarly embedded tree $ T $ satisfying the following conditions:
\begin{itemize}
\item it has a root vertex embedded in $ (0,0) $ and all the other vertices having their second coordinate (the ``height'') equal to a positive integer
\item two edges connected by an edge have heights whose difference is 1
\item the height along the unique minimal path from the root to every leaf is always increasing
\end{itemize}

A \emph{planar level tree with labels in $ I $} is a couple $ (T,\lambda) $ as follows. $ T $ is a planar level tree, and $ \lambda $ is a bijective labeling of the leaves of $ T $ with elements of $ I $.

A \emph{symmetric planar level tree} is a planar level tree invariant under the reflection $ r $ along the $ \bm{y} $ axis and having an odd number of leaves.

An \emph{antisymmetric planar level tree with labels in $ \{-n,\dots,n\} $} is a labeled planar level tree $ (T,\lambda) $ with labels in $ \{-n,\dots, n\} $ such that $ T $ is symmetric, and two leaves that correspond to each other under the application of $ r $ have labels opposite to each other.

The antisymmetric planar level tree associated with a depth ordering $ \Gamma $ is the antisymmetric planar level tree $ T_{\Gamma} $, unique up to isotopy, defined by the following properties:
\begin{itemize}
\item the labels of the leaves, from left to right, are $ -i_n,\dots,-i_1,0,i_1,\dots,i_n $
\item the leaves labeled $ i_{k-1},i_k $, for $ 1 \leq k \leq n $, are separated by a vertex of height $ a_k $ but not by vertices of height less than $ a_k $
\end{itemize}

Let $ k \geq 0 $. The \emph{$ k $-symmetrization} $ S_k(T) $ (respectively $ \tilde{S}_k(T) $) of a planar level tree $ T $ (with labels in $ \{1,\dots,n\} $) is a symmetric planar level tree $ S $ (respectively, antisymmetric planar level tree with labels in $ \{-n,\dots,n\} $) obtained by the following procedure. Glue $ T $ from the right to a vertical linear planar level tree lying into the $ \bm{y} $ axis up to height $ k $. Then, add the mirror image of such tree under $ r $ to obtain a symmetric planar level tree (choosing the unique antisymmetric labeling that extends the labeling of $ T $ in the labeled case).
\end{definition}

There is a free action of $ \WB{n} $ on the set antisymmetric planar level trees with labels in $ \{-n,\dots,n\} $ given by interpreting elements of $ \WB{n} $ as signed permutations and permuting labels accordingly.
We always assume that the edges of a level tree are oriented so that there is a unique oriented path from the root vertex to each leaf.

Similarly to the symmetric group case, we have the following immediate proposition.
\begin{proposition}\label{prop:FN basis}
The function $ \Gamma \mapsto T_{\Gamma} $ is a bijection between the set of signed depth-orderings with $ n $ labels and the set of isotopy classes of antisymmetric planar level trees with labels in $ \{-n,\dots,n\} $.
Furthermore, to $ \Gamma =(0 <_{a_0} i_1 <_{a_1} \dots <_{a_{n-1}} i_n) $ is associated a stratum $ \sigma e(\underline{a}) \in \Phi_{\infty,\WB{n}}$, where $ \sigma(k) = i_k $, $ \underline{a} = (a_0,\dots,a_{n-1}) $, and this provides a $ \WB{n} $-equivariant additive basis of $ \widetilde{FN}^*_{\WB{n}} $ labeled by signed depth-orderings or, equivalently, by isotopy classes of antisymmetric planar level trees with labels in $ \{-n,\dots,n\} $.
$ \WB{n} $ acts on this basis by permuting labels. Consequently, an additive basis for $ FN^*_{\WB{n}} $ is given by symmetric planar level trees with $ 2n+1 $ leaves.
\end{proposition}

An example of an antisymmetric planar level tree $ (T,\lambda) $, with labels in $ [-3,3] $, is given in Figure \ref{fig:antisymmetric tree} below. The signed depth-ordering associated with that is $ \Gamma = (0 <_1 -2 <_0 -3 <_1 1) $ and the corresponding stratum is $ \sigma e([1,0,1]) $, where $ \sigma(1) = -2 $, $ \sigma(2) = -3 $ and $ \sigma(3) = 1 $.

\begin{figure}[h]
\caption{An example of antisymmetric planar level tree with labels in $ [-3,3] $.}
\label{fig:antisymmetric tree}
\begin{center}
\begin{tikzpicture}[line cap=round,line join=round,>=triangle 45,x=1cm,y=1cm]
\begin{axis}[
x=1cm,y=1cm,
axis lines=middle,
grid style=dashed,
ymajorgrids=true,
xmajorgrids=false,
xmin=-3.5,
xmax=3.5,
ymin=-0.5,
ymax=4.9,
xtick=\empty,
yticklabels={,,}
]
\clip(-5.826666666666665,-1.58) rectangle (7.213333333333337,6.1);
\draw [line width=1pt] (0,0)-- (0,1);
\draw [line width=1pt] (-2,1)-- (0,0);
\draw [line width=1pt] (0,0)-- (2,1);
\draw [line width=1pt] (2,1)-- (2,2);
\draw [line width=1pt] (0,1)-- (0,2);
\draw [line width=1pt] (0,2)-- (0,3);
\draw [line width=1pt] (0,3)-- (0,4);
\draw [line width=1pt] (-2,1)-- (-2,2);
\draw [line width=1pt] (0,1)-- (-1,2);
\draw [line width=1pt] (-1,2)-- (-1,3);
\draw [line width=1pt] (-1,3)-- (-1,4);
\draw [line width=1pt] (0,1)-- (1,2);
\draw [line width=1pt] (1,2)-- (1,3);
\draw [line width=1pt] (1,3)-- (1,4);
\draw [line width=1pt] (2,2)-- (2,3);
\draw [line width=1pt] (2,3)-- (2,4);
\draw [line width=1pt] (-2,2)-- (-2,3);
\draw [line width=1pt] (-2,3)-- (-2,4);
\draw [line width=1pt] (-2,1)-- (-3,2);
\draw [line width=1pt] (-3,2)-- (-3,3);
\draw [line width=1pt] (-3,3)-- (-3,4);
\draw [line width=1pt] (2,1)-- (3,2);
\draw [line width=1pt] (3,2)-- (3,3);
\draw [line width=1pt] (3,3)-- (3,4);
\draw (-3.3333333333333317,4.646666666666671) node[anchor=north west] {$-1$};
\draw (-3.3333333333333317,4.646666666666671) node[anchor=north west] {$-1$};
\draw (2.8266666666666693,4.673333333333338) node[anchor=north west] {$1$};
\draw (-0.1733333333333311,4.673333333333338) node[anchor=north west] {$0$};
\draw (1.6266666666666691,4.646666666666671) node[anchor=north west] {$-3$};
\draw (-2.16,4.673333333333338) node[anchor=north west] {$3$};
\draw (-1.1733333333333313,4.673333333333338) node[anchor=north west] {$2$};
\draw (0.626666666666669,4.646666666666671) node[anchor=north west] {$-2$};
\begin{scriptsize}
\draw [fill=black] (0,0) circle (1.5pt);
\draw [fill=black] (0,1) circle (1.5pt);
\draw [fill=black] (-2,1) circle (1.5pt);
\draw [fill=black] (2,1) circle (1.5pt);
\end{scriptsize}
\end{axis}
\end{tikzpicture}
\end{center}
\end{figure}

We observe that we can use Proposition \ref{prop:FN basis} to reinterpret operations on (symmetric) level trees in terms of $ n $-tuples or (signed) depth-orderings. For instance, the $ k $-symmetrization of trees provides a linear map $ S_k \colon FN^*_{\Sigma_n} \to FN^*_{\WB{n}} $ that we can interpret as $ [a_1:\dots:a_{n-1}] \mapsto [k:a_1:\dots:a_{n-1}] $.

We can now describe the differential in terms of this basis.
\begin{definition} {\rm (from \cite{Sinha:13})} $ \, \, \, $\label{def:vertex permutation}
Let $ (T,\lambda) $ be a planar level tree. Let $ v $ be an internal vertex. Let $ E(v) $ be the set of edges whose source vertex is $ v $. The planar embedding of $ T $ induces an order on $ E(v) $, defined from left to right. A vertex permutation of $ (T,\lambda) $ at $ v $ is another planar level tree that is isomorphic to $ (T,\lambda) $ as a labeled tree but with a different planar embedding that differs from the original one only by the ordering on $ E(v) $.

Given a planar level tree $ (T,\lambda) $ and an internal vertex $ v $, let $ (e,f) $, $ e<f $, be a couple of adjacent edges in $ E(v) $. Let $ u_e $, $ u_f $ be the targets of $ e $ and $ f $, respectively. Let $ \sigma $ be a shuffle of the two sets $ E(u_e) $ and $ E(u_f) $. Let $ d_{e,f,\sigma}(T,\lambda) $ be the planar level tree obtained by gluing together $ e $ and $ f $, with common target $ \overline{u} $, and then applying the vertex permutation that permutes the edges in $ E(\overline{u}) $ by $ \sigma $.
\end{definition}

Recall that, in the $ A_n $ case, the differential in $ \widetilde{FN}^* $ of the basis element corresponding to a planar level tree with labels $ (T,\lambda) $ is given by the sum over $ (v,\sigma) $ as above of $ d_{v,\sigma}(T,\lambda) $. Similarly, we have the following proposition, which essentially states that a symmetrization of the previous construction gives the differential in the $ B_n $ case.
\begin{proposition}\label{prop:tree differential}
With the correspondence provided by Proposition \ref{prop:FN basis}, the differential of the cochain complex $ \widetilde{FN}^*_{\WB{n}} \otimes \Ftwo $ is given in terms of antisymmetric level trees with labels by $ d(T,\lambda) = \sum_{(e,f,\sigma)}\sum_{(e',f',\tau)} d_{e,f,\sigma}d_{e',f',\tau}(T,\lambda) $, where the sum is over sextuples $ (e,f,\sigma,e',f',\tau) $ such that $ d_{e,f,\sigma}d_{e',f',\tau}(T,\lambda) $ is again an antisymmetric planar level tree.
Equivalently, $ d(T,\lambda) $ is obtained by performing an operation $ d_{e,f,\sigma} $ starting from a couple of adjacent vertices $ (e,f) $ lying into the positive half-plane $ \{(x,y): x \geq 0\} $, and then perform the mirror operation on the mirror pair of adjacent edges $ (e',f') $ in the negative half-plane. If we call such symmetric operation $ d^S_{e,f,\sigma} $, we have that
\[
d(T,\lambda) = \sum_{(e,f)} \sum_{\sigma} d^S_{e,f,\sigma}(T,\lambda),
\]
where the sum is over couples $ (e,f) $ of adjacent edges in the positive half-plane and shuffles $ \sigma $ of the two sets of vertices incident to the target of $ e $ and $ f $, respectively.
\end{proposition}

We can equivalently express this construction using planar level trees $ T $ with $ n+1 $ leaves labeled by $ (-n,\dots,-1,0,1\dots,n) $, with labels having pairwise different absolute values, such that the leftmost leaf has label $ 0 $. We recover the corresponding antisymmetric level tree as follows. We choose a representative of the isotopy class of $ (T,\lambda) $ in which the entire oriented path from the root vertex to the label $ 0 $ lies on the $ \bm{y} $ axis. Then we merge $ T $ with its mirror image along $ \bm{y} $ with opposite labels. In this case, the differential is given by contracting a couple of adjacent edges and shuffling. When the result is a tree whose leftmost leaf is not labeled by $ 0 $, we replace the part of the tree belonging to the negative half-plane $ \{(x,y): x \leq 0\} $ with its mirror image in the positive half-plane, with opposite labels, and shuffle the corresponding edges in all possible ways.

\subsection{The special case of Coxeter groups of Type D}

We now turn to the description of the complex $ FN_{\WD{n}}^* $.
Once again, since this Coxeter group has $ n $ fundamental reflections $ t_0, \dots, t_{n-1} $, a $ \mathbb{Z} \left[ \WD{n} \right] $-basis for $ \widetilde{FN}_{\WD{n}}^* $ is indexed by $ n $-tuples of non-negative integers $ \underline{a} = \left( a_0, \dots, a_{n-1} \right) $.

The inclusion $ j_n \colon \WD{n} \rightarrow \WB{n} $ identifies the reflection arrangement associated with $ \WD{n} $ with the sub-arrangement of $ \mathcal{A}_{\WB{n}} $ composed by the hyperplanes $ H_{i,j}^\pm $, for $ 1 \leq i < j \leq n $. $ t_i = s_i $ for $ 1 \leq i \leq n $, while $ t_0 $ is the reflection along $ H_{1,2}^- $.
Thus the basis element of $ \widetilde{FN}_{\WD{n}}^* $ corresponding to $ \underline{a} $ is described as the stratum
\begin{align*}
	&e \left( \underline{a} \right) = \Big\{ \left( \underline{x}_1, \dots, \underline{x}_n \right) \in \left( \mathbb{R}^\infty \right)^n: \forall 1 \leq i \leq {n-1}, 1 \leq j \leq a_i: \left( x_i \right)_j = \left( x_{i+1} \right)_j,\\
	&\left( x_i \right)_{a_i+1} < \left( x_{i+1} \right)_{a_i+1}, \forall 1 \leq k \leq a_0: \left( x_2 \right)_k = -\left( x_1 \right)_k, \left( x_2 \right)_{a_0+1} > - \left( x_1 \right)_{a_0+1} \Big\}.
\end{align*}

Passing to the quotient by the action of $ \WD{n} $, we see that $ FN^*_{\WD{n}} $ has a $ \mathbb{Z} $-basis constituted by $ \left[ a_0: \dots a_{n-1} \right] = \left[ e \left( a_0, \dots, a_{n-1} \right) \right] $.

The complex $ \frac{ \widetilde{FN}_{\WB{n}}^* }{ j_n \left( \WD{n} \right) } $ also calculates the cohomology of $ \WD{n} $. Therefore, there is a cochain equivalence $ \varphi \colon FN^*_{\WD{n}} \rightarrow \frac{ \widetilde{FN}_{\WB{n}}^* }{ j_n \left( \WD{n} \right) } $ between the two resolutions. We conclude this section by computing an explicit formula for $ \varphi $ that we will use to perform cochain-level computation in the following sections.
For instance, we will prove the relations for coproduct of transfer products of Hopf ring generators by mapping them to $ \frac{ \widetilde{FN}_{\WB{n}}^* }{ j_n \left( \WD{n} \right) } $, where their expressions are closer to the $ B_n $ case. As a notational convention, we denote this cochain complex with $ {FN'}_{\WD{n}}^{*} $.

First, we observe that $ \left[ \WB{n}: j_n \left( \WD{n} \right) \right] = 2 $, thus $ j_n \left( \WD{n} \right) $ is a normal subgroup of $ \WB{n} $. The two cosets of $ j_n \left( \WD{n} \right) $ in $ \WB{n} $ are represented by the identity and $ s_0 $, the only fundamental reflection of $ \WB{n} $ that is not contained in $ j_n \left( \WD{n} \right) $.
Thus, given a $ \mathbb{Z} \left[ \WB{n} \right] $-basis $ \mathcal{B} $ for $ \widetilde{FN}^*_{\WB{n}} $, the classes of $ x $ and $ s_0.x $, where $ x \in \mathcal{B} $, provide a $ \mathbb{Z} $-basis for $ {FN'}^{*}_{\WD{n}} $.
Let $ \mathcal{B} $ be the basis defined above in terms of $ n $-tuples or equivalently of symmetric planar level trees, parametrized by $ n $-tuples of non-negative integers $ \underline{a} = \left( a_0, \dots, a_{n-1} \right) $. We denote by $ \left[a_0: \dots: a_{n-1} \right]^+ $ and $ \left[ a_0: \dots: a_{n-1} \right]^- $ the cochains in $ {FN'}^{*}_{\WD{n}} $ arising from the basis element corresponding to $ \underline{a} $ and $ s_0 \underline{a} $.

The complex $ {FN'}^*_{\WD{n}} $ also has a description in terms of trees.
\begin{definition}\label{def:tree orientation}
Let $ T $ be a symmetric planar level tree with $ 2n+1 $ leaves. An \emph{orientation} of $ T $ is the choice of an element of $ \frac{L}{\sim} $, where
$ L $ is the set of antisymmetric labelings of $ T $ with labels in $ \{-n,\dots,n\} $, and $ \sim $ is the equivalence relation defined by $ \lambda \sim \lambda' \Leftrightarrow \exists \sigma \in \Alt(2n+1): \lambda' = \lambda \sigma $.
An \emph{oriented symmetric planar level tree} is a couple $ (T,\mathcal{O}) $, where $ T $ is a symmetric planar level tree and $ \mathcal{O} $ is an orientation of $ T $.
\end{definition}
Note that if two antisymmetric labelings of a symmetric planar level tree $ T $ differ by a permutation $ \sigma \in \Sigma_{\{-n,\dots,n\}} $, then $ \sigma $ must fix $ 0 $ and act as a signed permutation on $ \{-n,\dots,-1,1,\dots,n\} $. Hence, an orientation of $ T $ is the choice of an antisymmetric labeling up to the action of $ j_n(\WD{n}) $. Since the index $ [\WB{n}:j_n(\WD{n})] $ is $ 2 $, there are two possible orientations for a symmetric planar level tree $ T $, determined by the parity of the number of negative labels of leaves in the positive half-plane. In particular, we can identify an orientation $ \mathcal{O} $ with a sign $ + $ or $ - $, corresponding to labelings with an even or odd number of positively labeled leaves in the positive half-plane, respectively.

Moreover, from the fact that $ \Alt(2n+1) $ is normal in $ \Sigma_{2n+1} $, it follows that if $ T $ is a symmetric planar level tree, $ \lambda $ is a labeling of $ T $ and $ \sigma(T) $ is a vertex permutation of $ T $ at a vertex $ v $, then the orientation of the permuted labeled tree $ \sigma(T,\lambda) $ only depends on the orientation determined by $ \lambda $. Therefore, the rule for the differential in $ \widetilde{FN}^*_{\WB{n}} $ induces a formula for the differential in $ {FN'}^*_{\WD{n}} $ in terms of trees. Hence, we have the following description.
\begin{proposition}\label{prop:FN D}
$ {FN'}^*_{\WD{n}} $ can be described as the cochain complex having additive basis indexed by oriented symmetric planar level trees with $ 2n+1 $ leaves, with differential induced by the symmetric tree differential in $ \widetilde{FN}^*_{\WB{n}} $ by keeping track of orientations.
\end{proposition}
The reader is encouraged to compare this description with the notion of ``charged'' configuration used in \cite{Sinha:17}.

\section{Geometry and combinatorics: chain level formulas} \label{sec:geometric description}

We devote this section to developing some formulas that will allow us to perform calculations at the (co)chain level. These computations will be needed at points, especially when retrieving relations. We first compute some connecting maps between the Fox-Neuwirth complexes of Coxeter groups of Type $ A $, $ B $, and $ D $. Then, we provide cochain representatives of the structural maps of our almost-Hopf ring structures.

\subsection{The connecting homomorphisms} \label{subsec:connecting}

As $ \widetilde{FN}^*_{\WD{n}} $ and $ \widetilde{FN'}^*_{\WD{n}} $ are both free resolutions of $ \mathbb{Z} $ as a $ \mathbb{Z}[\WD{n}] $-module, they need to be $ \WD{n} $-equivariantly cochain equivalent. 
We begin by providing a formula for an explicit equivalence $ \varphi $ relating the two models $ FN^*_{\WD{n}} $ and $ {FN'}^*_{\WD{n}} $.
\begin{lemma} \label{lem:CHE}
There is a cochain homotopy equivalence $ \varphi^* \colon FN^*_{\WD{n}} \rightarrow {FN'}^*_{\WD{n}} $ defined by the formula:
\[
\varphi^* \left[ a_0: \dots:a_{n-1} \right] = \left\{ \begin{array}{l} \left[ a_0: a_1: a_2: \dots: a_{n-1} \right]^+ \quad\quad\mbox{if } a_0 < a_1 \\
\left[a_0: a_1: a_2: \dots: a_{n-1} \right]^+ + \left[ a_1: a_0: a_2: \dots: a_{n-1} \right]^- \\
\quad\quad\mbox{if } a_0 = a_1 \\
\left[ a_1: a_0: a_2: \dots: a_{n-1} \right]^- \quad\quad\mbox{if } a_0 > a_1 \\ \end{array} \right.
\]
induced by the inclusion $ Y^{(\infty)}_{\WB{n}} \subseteq Y^{(\infty)}_{\WD{n}} $ and yielding the identity in cohomology.
\end{lemma}
\begin{proof}
We observe that the inclusion $ Y^{(\infty)}_{\WB{n}} \subseteq Y^{(\infty)}_{\WD{n}} $ is a $ \WD{n} $-equivariant homotopy equivalence. Moreover, the inverse image in $ Y^{(\infty)}_{\WB{n}} $ of each stratum of $ \Phi_{\infty,\WD{n}} $ is a union of strata in $ \Phi_{\infty,\WB{n}} $. Thus, passing to quotients, this yields a map $ \varphi \colon \frac{ Y^{(\infty)}_{\WB{n}} }{ \WD{n} } \to \frac{Y^{(\infty)}_{\WD{n}}}{ \WD{n} } $ that induces a well-defined map between the cochain complexes $ \varphi^* \colon FN^*_{\WD{n}} \to {FN'}^*_{\WD{n}} $.

We now check that $ \varphi^* $ satisfies the given formulas. It is sufficient to consider the finite approximations $ \varphi^{(d)} \colon \frac{ Y^{(d)}_{\WB{n}} }{ \WD{n} } \to \frac{Y^{(d)}_{\WD{n}}}{ \WD{n} } $. For any given stratum $ S = e \left( a_0, \dots, a_{n-1} \right) $ for $ \WD{n} $, since $ \varphi^{(d)} $, being a $ 0 $-codimensional immersion, is transverse to $ S $, we have that $ ( \varphi^{(d)} )^* \left( \tau_S \right) = \tau_{(\varphi^{(d)})^{-1} ( S )} $.
We now distinguish three cases:
\begin{itemize}
	\item if $ a_0 < a_1 $, then $ ( \varphi^{(d)} )^{-1} ( S ) = e \left( a_0, a_1, a_2, \dots, a_{n-1} \right) $
	\item if $ a_0 > a_1 $, then $ ( \varphi^{(d)} )^{-1} ( S ) = s_0.e \left( a_1, a_0, a_2, \dots, a_{n-1} \right) $
	\item if $ a_0 = a_1 $, then $ ( \varphi^{(d)} )^{-1} ( S ) $ is the union of the cells $ e \left( a_0, \dots, a_{n-1} \right) $, $ s_0.e \left( a_0, \dots, a_{n-1} \right) $ and strata of bigger codimension.
\end{itemize}
This implies that $ \varphi^* $ has the desired form.
\end{proof}

We also consider the following group homomorphisms:
\begin{itemize}
\item the standard inclusion $ j \colon \Sigma_n \to \WB{n} $ already considered in the previous section
\item the involution $ c_{s_0} \colon \WD{n} \to \WD{n} $ given by conjugation by $ s_0 $, the unique generating reflection of $ \WB{n} $ that does not belong to $ \WD{n} $, that fixes $ t_i $ for $ 2 \leq i < n $ and switches $ t_0 $ and $ t_1 $.
\item the two inclusions $ i_+, i_- \colon \Sigma_n \to \WD{n} $ given, in terms of the Coxeter generators $ t_0, \dots, t_n $ of Figure \ref{fig:diagrammi Coxeter}, by $ i_{\pm} \left( i, i+1 \right) = t_i $ if $ i \geq 2 $, $ i_{+} \left( 1,2 \right) = t_1 $ and $ i_{-} \left( 1,2 \right) = t_0 $
\end{itemize}
We denote by $ \iota \colon H^*(\WD{n}; \Ftwo) \to H^*(\WD{n}; \Ftwo) $ the morphism induced by $ c_{s_0} $ on cohomology.

We note that the two following properties hold by construction:
\begin{itemize}
\item $ \pi j = \id_{\Sigma_n} $
\item $ \pi \circ i_{+} = \pi \circ i_{-} = \id_{\Sigma_n} $, where $ \pi \colon \WD{n} \to \Sigma_n $ is the composition of the inclusion $ j \colon \WD{n} \to \WB{n} $ with the projection $ \WB{n} \to \Sigma_n $
\item $ c_{s_0} \circ i_{+} = i_{-} $
\end{itemize}

We compute cochain representatives of $ \iota $ in the following lemmas.
\begin{lemma} \label{lem:involuzione geometrica}
$ \iota $ is induced by the cochain-level map $ \iota^\# \colon FN^*_{\WD{n}} \rightarrow FN^*_{\WD{n}} $ described as follows:
\[
	\iota^\# \left[ a_0: a_1: a_2: \dots: a_{n-1} \right] = \left[ a_1: a_0: a_2: \dots: a_{n-1} \right]
\]
\end{lemma}
\begin{proof}
Since the image under $ c_{s_0} $ of a fundamental reflection for $ W = \WD{n} $ is again a fundamental reflection, for every $ \Gamma' \subseteq \Gamma \subseteq \left\{ t_0, \dots, t_{n-1} \right\} $, the set of minimal-length coset representatives satisfy $ c_{s_0} \left( W^{\Gamma'}_{\Gamma} \right) = W^{c_{s_0} \left( \Gamma' \right)}_{c_{s_0} \left( \Gamma \right)} $.
Thus the following defines a $ c_{s_0} $-equivariant chain map $ C^{\WD{n}}_* \rightarrow C^{\WD{n}}_* $:
\[
	e \left( \Gamma_1 \supseteq \dots \supseteq \Gamma_k \supseteq \dots \right) \mapsto e \left( c_{s_0} \left( \Gamma_1 \right) \supseteq \dots \supseteq c_{s_0} \left( \Gamma_k \right) \supseteq \dots \right)
\]
This yields, dually, the desired cochain map $ FN^*_{\WD{n}} \rightarrow FN^*_{\WD{n}} $.
\end{proof}

We can also describe $ \iota $ in terms of $ {FN'}^{*}_{\WD{n}} $. The proof of the following lemma is straightforward.
\begin{lemma} \label{lem:involuzione geometrica2}
$ \iota $ is induced by the cochain-level map $ {\iota'}^{\#} \colon {FN'}^{*}_{\WD{n}} \rightarrow {FN'}^{*}_{\WD{n}} $ described as follows:
\begin{align*}
	{\iota'}^{\#} \left[ a_0: \dots: a_{n-1} \right]^+ &= \left[ a_0: \dots: a_{n-1} \right]^- \\
	{\iota'}^{\#} \left[ a_0: \dots: a_{n-1} \right]^- &= \left[ a_0: \dots: a_{n-1} \right]^+ \\
\end{align*}
In terms of oriented symmetric planar level trees, the map $ {\iota'}^{\#} $ acts on $ (T,\mathcal{O}) $ by replacing $ \mathcal{O} $ with the opposite orientation.
\end{lemma}

The following identity is also proved by direct inspection.
\begin{lemma}\label{lem:CHEiota}
The following diagram commutes:
\begin{center}
\begin{tikzcd}
FN^*_{\WD{n}} \arrow{r}{\varphi^*} \arrow{d}{\iota^{\#}} & {FN'}^*_{\WD{n}} \arrow{d}{{\iota'}^{\#}}\\
FN^*_{\WD{n}} \arrow{r}{\varphi^*} & {FN'}^*_{\WD{n}} \\
\end{tikzcd}
\end{center}
\end{lemma}

The formulas for the other connecting maps follow from a general remark.
\begin{lemma}
Let $ G $ be a Coxeter group, with Coxeter generators $ S = \{s_0,\dots,s_n\} $ and $ H \leq G $ be a parabolic subgroup, generated by a subset $ T = \{s_{i_0},\dots, s_{i_m}\} \subseteq S $. The inclusion $ H \hookrightarrow G $ is represented at the chain level by the chain map $ C^H_* \to C^G_* $ given by $ c(\underline{\Gamma},\gamma) \mapsto c(\underline{\Gamma},\gamma) $, for flags $ \underline{\Gamma} = (\Gamma_0 \supseteq \Gamma_1 \supseteq \dots \supseteq \Gamma_k \supseteq \varnothing) $ with $ \Gamma_0 \subseteq T \subseteq S $ and elements $ \gamma \in H $.

Dually, it is represented at the cochain level by the cochain map $ FN^*_G \to FN^*_H $ given by
\[
[e(a_0,\dots,a_n)] \in C^H_* \mapsto \left\{ \begin{array}{ll}
[e(a_{i_0},\dots,a_{i_m})] & \mbox{if } \forall j \notin \{i_0,\dots,i_m\}: a_j = 0 \\
0 & \mbox{otherwise}
\end{array} \right. .
\]
\end{lemma}
\begin{proof}
Since the inclusion of parabolic subgroups preserves minimal coset representatives, the De Concini-Salvetti boundary formula of Theorem \ref{teo:Salvetti} implies that the given linear morphism $ C^H_* \to C^G_* $ is an $ H $-equivariant chain map.
Dualizing this yields the cochain formula between Fox-Neuwirth complexes.
\end{proof}

As particular cases of this lemma, we retrieve cochain formulas for our connecting homomorphisms:
\begin{corollary} \label{cor:connecting hom cochain}
The following statements are true.
\begin{enumerate}
\item The linear morphism $ j^{\#} \colon FN^*_{\WB{n}} \to FN^*_{\Sigma_n} $ given by
\[
[a_0,\dots,a_{n-1}] \mapsto \left\{ \begin{array}{ll}
[a_1,\dots,a_{n-1}] & \mbox{if } a_0 = 0 \\
0 & \mbox{if } a_0 > 0
\end{array} \right.
\]
represents $ j $ at the cochain level.
\item The linear morphism $ i_+^{\#} \colon FN^*_{\WD{n}} \to FN^*_{\Sigma_n} $ given by
\[
[a_0,\dots,a_{n-1}] \mapsto \left\{ \begin{array}{ll}
[a_1,\dots,a_{n-1}] & \mbox{if } a_0 = 0 \\
0 & \mbox{if } a_0 > 0
\end{array} \right.
\]
represents $ i_+ $ at the cochain level.
\item The linear morphism $ i_-^{\#} \colon FN^*_{\WD{n}} \to FN^*_{\Sigma_n} $ given by
\[
[a_0,\dots,a_{n-1}] \mapsto \left\{ \begin{array}{ll}
[a_0,a_2,\dots,a_{n-1}] & \mbox{if } a_1 = 0 \\
0 & \mbox{if } a_1 > 0
\end{array} \right.
\]
represents $ i_+ $ at the cochain level.
\end{enumerate}
\end{corollary}

\subsection{Structural morphisms: $ A_B $}

We want to describe the almost-Hopf ring structures presented in Section \ref{sec:Hopf rings} in our geometric context.
We begin with the coproduct map i $ A_B $. In contrast with the symmetric group case, the cochain-level map inducing the coproduct is relatively complicated. Its underlying combinatorics is built upon elementary steps that we, mindful of the botanic analogy, suggestively call ``prunings.''

\begin{definition} \label{def:pruning}
Let $ T $ be a planar level tree. An \emph{elementary $ k $-pruning} of $ T $ is a planar level tree $ T' $ obtained by the following procedure. Choose an internal vertex $ v $ of $ T $ of height $ k $, and consider on $ E(v) $ the order induced by the planar embedding. Let $ 1 \leq l < |E(v)| $, consider the $ l $ biggest elements $ e_1,\dots,e_l $ of $ E(v) $ with respect to this order, and let $ v'_i $ be the target of $ e_i $. $ T'' $ is the subtree of $ T $ spanned by $ v $ and all vertices that can be reached from one of the $ v'_i $ through an oriented path. $ T' $ is the complementary subtree of $ T'' $ in $ T $. We call the planarly embedded subtree $ T'' $ the \emph{scrap} of the elementary $ k $-pruning. An elementary $ k $-pruning is said to be \emph{minimal} if $ l = 1 $.
A \emph{$ k $-pruning} is a couple $ (T',T'') $ constructed as follows:
\begin{itemize}
\item $ T' $ is obtained from a sequence of elementary $ k $-prunings $ T \leadsto T'_1 \leadsto T'_2 \leadsto \dots \leadsto T'_j = T' $ performed on pairwise different vertices $ v_1,\dots,v_j $ of $ T $, with scraps $ T''_1, \dots, T''_j $
\item $ T'' $ is a planar level tree obtained by joining these scrap subtrees along a vertex $ w $ of height $ k $ and performing a vertex permutation at $ w $ that shuffles the edges of the scraps
\end{itemize}

Let $ T $ be a symmetric planar level tree. An \emph{elementary symmetric $ k $-pruning} of $ T $ is a tree $ T' $ obtained as follows. Apply to $ T $ an elementary (non-symmetric) $ k $-pruning whose scrap $ T'' $ does not contain the central leaf belonging to the $ \bm{y} $ axis. Then, remove the image of the subtree of $ T'' $ under the reflection $ r $ along the vertical axis. $ T'' $ is called the \emph{scrap} of the elementary symmetric pruning. An elementary symmetric $ k $-pruning is said to be \emph{minimal} if it is obtained from a minimal elementary $ k $-pruning.
A \emph{symmetric $ k $-pruning} is a couple $ (T',T'') $, where:
\begin{itemize}
\item $ T' $ is obtained from a sequence of elementary $ k $-prunings $ T \leadsto T'_1 \leadsto T'_2 \leadsto \dots \leadsto T'_j = T' $ performed on pairwise different vertices of $ T $, with scraps $ T''_1, \dots, T''_j $
\item $ T'' $ is a non-symmetric planar level tree obtained by joining the scrap subtrees to a vertex $ w $ of height $ k $ and performing a vertex permutation at $ w $ that shuffles the edges of the scraps
\end{itemize}
\end{definition}

We note that elementary $ k $-prunings at different vertices commute, both in the symmetric and non-symmetric cases. Hence, a $ k $-pruning or symmetric $ k $-pruning is uniquely determined by the set of elementary $ k $-prunings that compose it, independently of the order in which they are performed.

There is also an alternative way to define (symmetric) $ k $-prunings in terms of minimal $ k $-prunings instead of elementary ones. A (symmetric) $ k $-pruning is obtained by performing a sequence of minimal elementary (symmetric) $ k $-prunings, not necessarily at pairwise different vertices, and then joining the scraps at a vertex of height $ k $ without shuffling the edges.

We now consider three linear morphisms that we will need to define the cochain-level coproduct map:
\begin{itemize}
\item the $ k $-pruning map $ P_k \colon FN^*_{\WB{n}} \otimes \mathbb{F}_2 \to \bigoplus_{a+b=n} FN^*_{\WB{a}} \otimes FN^*_{\WB{b}} \otimes \mathbb{F}_2$ that maps a symmetric planar level tree $ T $ to the sum $ \sum T' \otimes S_k(T'') $ over all the possible symmetric $ k $-prunings $ (T',T'') $ of $ T $
\item the minimal $ k $-pruning map $ P_k^{\min} \colon FN^*_{\WB{n}} \otimes \mathbb{F}_2 \to \bigoplus_{a+b=n} FN^*_{\WB{a}} \otimes FN^*_{\WB{b}} \otimes \mathbb{F}_2$ that maps a symmetric planar level tree $ T $ to the sum $ \sum T' \otimes S_k(T'') $ over all the possible minimal elementary symmetric $ k $-prunings $ (T',T'') $ of $ T $
\item the concatenation map $ C \colon FN^*_{\WB{n}} \otimes FN^*_{\WB{m}} \otimes \mathbb{F}_2 \to FN^*_{\WB{n+m}} $ such that $ C([a_0:\dots:a_{n-1}] \otimes [b_0:\dots:b_{m-1}]) = [a_0:\dots:a_{n-1}:b_0:\dots:b_{m-1}] $
\end{itemize}
The map $ P_k $ is exemplified in Figure \ref{fig:pruning}. We understand $ C $ at the level of symmetric planar level trees as the function given by the following procedure. Take a couple of such objects $ (T,S) $. Cut $ S $ along its central vertical axis. Finally, glue the right piece of $ S $ onto the right side of $ T $ and the left part onto its left side to obtain a new symmetric planar level tree.
We remark that these linear morphisms are degree-preserving, but they are not chain maps.

\begin{figure}[hbtp]
\caption{The map $P_1 $, defined as the sum of all possible symmetric $ 1 $-pruning, on a given symmetric planar level tree.}
\label{fig:pruning}
\begin{center}
\def \globalscale {2.000000}
\begin{tikzpicture}[y=0.80pt, x=0.80pt, yscale=-\globalscale, xscale=\globalscale, inner sep=0pt, outer sep=0pt]

  \draw[color=black] (80,15) node {$\mapsto$};
  \draw[color=black] (165,15) node {$\otimes$};
  \draw[color=black] (185,15) node {$+$};
  \draw[color=black] (55,45) node {$\otimes$};
  \draw[color=black] (95,45) node {$+$};
  \draw[color=black] (145,45) node {$\otimes$};
  \draw[color=black] (195,45) node {$+$};
  \draw[color=black] (65,75) node {$\otimes$};
  \draw[color=black] (95,75) node {$+$};
  \draw[color=black] (125,75) node {$\otimes$};
  \draw[color=black] (195,75) node {$+$};
  \draw[color=black] (25,105) node {$\otimes$};
  \draw[color=black] (95,105) node {$+$};
  \draw[color=black] (145,105) node {$\otimes$};
  \draw[color=black] (195,105) node {$+$};
  \draw[color=black] (45,135) node {$\otimes $};

  \path[draw=black,line cap=butt,line join=miter,line width=0.212pt] (5.0000,5.0000) -- (5.0000,15.0000) -- (20.0000,20.0000) -- (35.0000,25.0000) -- (50.0000,20.0000) -- (65.0000,15.0000) -- (65.0000,5.0000) -- (65.0000,5.0000) -- (65.0000,5.0000);

  \path[draw=black,line cap=butt,line join=miter,line width=0.212pt] (10.0000,5.0000) -- (10.0000,10.0000) -- (20.0000,15.0000) -- (20.0000,20.0000) -- (20.0000,20.0000) -- (20.0000,20.0000);

  \path[draw=black,line cap=butt,line join=miter,line width=0.212pt] (60.0000,5.0000) -- (60.0000,10.0000) -- (50.0000,15.0000) -- (50.0000,20.0000);

  \path[draw=black,line cap=butt,line join=miter,line width=0.212pt] (15.0000,5.0000) -- (20.0000,10.0000) -- (20.0000,15.0000) -- (20.0000,15.0000);

  \path[draw=black,line cap=butt,line join=miter,line width=0.212pt] (55.0000,5.0000) -- (50.0000,10.0000) -- (50.0000,15.0000) -- (50.0000,15.0000);

  \path[draw=black,line cap=butt,line join=miter,line width=0.212pt] (20.0000,5.0000) -- (20.0000,10.0000) -- (20.0000,10.0000);

  \path[draw=black,line cap=butt,line join=miter,line width=0.212pt] (50.0000,5.0000) -- (50.0000,10.0000) -- (50.0000,10.0000);

  \path[draw=black,line cap=butt,line join=miter,line width=0.212pt] (35.0000,5.0000) -- (35.0000,25.0000);

  \path[draw=black,line cap=butt,line join=miter,line width=0.212pt] (25.0000,5.0000) -- (25.0000,10.0000) -- (35.0000,20.0000) -- (45.0000,10.0000) -- (45.0000,5.0000) -- (45.0000,5.0000);

  \path[draw=black,line cap=butt,line join=miter,line width=0.212pt] (40.0000,5.0000) -- (40.0000,15.0000) -- (40.0000,15.0000);

  \path[draw=black,line cap=butt,line join=miter,line width=0.212pt] (30.0000,5.0000) -- (30.0000,15.0000);

  \path[draw=black,line cap=butt,line join=miter,line width=0.212pt] (95.0000,5.0000) -- (95.0000,15.0000) -- (110.0000,20.0000) -- (125.0000,25.0000) -- (140.0000,20.0000) -- (155.0000,15.0000) -- (155.0000,5.0000) -- (155.0000,5.0000) -- (155.0000,5.0000);

  \path[draw=black,line cap=butt,line join=miter,line width=0.212pt] (100.0000,5.0000) -- (100.0000,10.0000) -- (110.0000,15.0000) -- (110.0000,20.0000) -- (110.0000,20.0000) -- (110.0000,20.0000);

  \path[draw=black,line cap=butt,line join=miter,line width=0.212pt] (150.0000,5.0000) -- (150.0000,10.0000) -- (140.0000,15.0000) -- (140.0000,20.0000);

  \path[draw=black,line cap=butt,line join=miter,line width=0.212pt] (105.0000,5.0000) -- (110.0000,10.0000) -- (110.0000,15.0000) -- (110.0000,15.0000);

  \path[draw=black,line cap=butt,line join=miter,line width=0.212pt] (145.0000,5.0000) -- (140.0000,10.0000) -- (140.0000,15.0000) -- (140.0000,15.0000);

  \path[draw=black,line cap=butt,line join=miter,line width=0.212pt] (110.0000,5.0000) -- (110.0000,10.0000) -- (110.0000,10.0000);

  \path[draw=black,line cap=butt,line join=miter,line width=0.212pt] (140.0000,5.0000) -- (140.0000,10.0000) -- (140.0000,10.0000);

  \path[draw=black,line cap=butt,line join=miter,line width=0.212pt] (125.0000,5.0000) -- (125.0000,25.0000);

  \path[draw=black,line cap=butt,line join=miter,line width=0.212pt] (115.0000,5.0000) -- (115.0000,10.0000) -- (125.0000,20.0000) -- (135.0000,10.0000) -- (135.0000,5.0000) -- (135.0000,5.0000);

  \path[draw=black,line cap=butt,line join=miter,line width=0.212pt] (130.0000,5.0000) -- (130.0000,15.0000) -- (130.0000,15.0000);

  \path[draw=black,line cap=butt,line join=miter,line width=0.212pt] (120.0000,5.0000) -- (120.0000,15.0000);

  \path[draw=black,line cap=butt,line join=miter,line width=0.212pt] (175.0000,5.0000) -- (175.0000,25.0000) -- (175.0000,25.0000);

  \path[draw=black,dash pattern=on 0.21pt off 1.69pt,line cap=butt,line join=miter,line width=0.212pt,miter limit=4.00] (5.0000,20.0000) -- (200.0000,20.0000) -- (200.0000,20.0000);

  \path[draw=black,dash pattern=on 0.21pt off 1.70pt,line cap=butt,line join=miter,line width=0.212pt,miter limit=4.00] (200.0000,25.0000) -- (5.0000,25.0000);

  \path[draw=black,dash pattern=on 0.21pt off 1.69pt,line cap=butt,line join=miter,line width=0.212pt,miter limit=4.00] (200.0000,10.0000) -- (5.0000,10.0000);

  \path[draw=black,dash pattern=on 0.21pt off 1.69pt,line cap=butt,line join=miter,line width=0.212pt,miter limit=4.00] (5.0000,5.0000) -- (200.0000,5.0000);

  \path[draw=black,line cap=butt,line join=miter,line width=0.212pt] (4.9458,35.0000) -- (4.9458,45.0000) -- (19.9458,50.0000) -- (24.9458,55.0000) -- (29.9458,50.0000) -- (44.9458,45.0000) -- (44.9458,35.0000) -- (44.9458,35.0000) -- (44.9458,35.0000);

  \path[draw=black,line cap=butt,line join=miter,line width=0.212pt] (9.9458,35.0000) -- (9.9458,40.0000) -- (19.9458,45.0000) -- (19.9458,50.0000) -- (19.9458,50.0000) -- (19.9458,50.0000);

  \path[draw=black,line cap=butt,line join=miter,line width=0.212pt] (39.9458,35.0000) -- (39.9458,40.0000) -- (29.9458,45.0000) -- (29.9458,50.0000);

  \path[draw=black,line cap=butt,line join=miter,line width=0.212pt] (14.9458,35.0000) -- (19.9458,40.0000) -- (19.9458,45.0000) -- (19.9458,45.0000);

  \path[draw=black,line cap=butt,line join=miter,line width=0.212pt] (34.9458,35.0000) -- (29.9458,40.0000) -- (29.9458,45.0000) -- (29.9458,45.0000);

  \path[draw=black,line cap=butt,line join=miter,line width=0.212pt] (19.9458,35.0000) -- (19.9458,40.0000) -- (19.9458,40.0000);

  \path[draw=black,line cap=butt,line join=miter,line width=0.212pt] (29.9458,35.0000) -- (29.9458,40.0000) -- (29.9458,40.0000);

  \path[draw=black,line cap=butt,line join=miter,line width=0.212pt] (24.9458,35.0000) -- (24.9458,55.0000);

  \path[draw=black,dash pattern=on 0.21pt off 1.69pt,line cap=butt,line join=miter,line width=0.212pt,miter limit=4.00] (5.0000,50.0000) -- (200.0000,50.0000) -- (200.0000,50.0000);

  \path[draw=black,dash pattern=on 0.21pt off 1.70pt,line cap=butt,line join=miter,line width=0.212pt,miter limit=4.00] (200.0000,55.0000) -- (5.0000,55.0000);

  \path[draw=black,dash pattern=on 0.21pt off 1.69pt,line cap=butt,line join=miter,line width=0.212pt,miter limit=4.00] (5.0000,45.0000) -- (200.0000,45.0000);

  \path[draw=black,dash pattern=on 0.21pt off 1.69pt,line cap=butt,line join=miter,line width=0.212pt,miter limit=4.00] (200.0000,40.0000) -- (5.0000,40.0000);

  \path[draw=black,dash pattern=on 0.21pt off 1.69pt,line cap=butt,line join=miter,line width=0.212pt,miter limit=4.00] (5.0000,35.0000) -- (200.0000,35.0000);

  \path[draw=black,line cap=butt,line join=miter,line width=0.212pt] (75.4208,35.0000) -- (75.4208,55.0000) -- (75.4208,55.0000) -- (75.4208,55.0000);

  \path[draw=black,line cap=butt,line join=miter,line width=0.212pt] (65.4208,35.0000) -- (65.4208,40.0000) -- (75.4208,50.0000) -- (85.4208,40.0000) -- (85.4208,35.0000);

  \path[draw=black,line cap=butt,line join=miter,line width=0.212pt] (80.4208,35.0000) -- (80.4208,45.0000);

  \path[draw=black,line cap=butt,line join=miter,line width=0.212pt] (70.4208,35.0000) -- (70.4208,45.0000) -- (70.4208,45.0000);

  \path[draw=black,line cap=butt,line join=miter,line width=0.212pt] (120.1832,35.0000) -- (120.1832,55.0000) -- (120.1832,55.0000) -- (120.1832,55.0000);

  \path[draw=black,line cap=butt,line join=miter,line width=0.212pt] (110.1832,35.0000) -- (110.1832,40.0000) -- (120.1832,50.0000) -- (130.1832,40.0000) -- (130.1832,35.0000);

  \path[draw=black,line cap=butt,line join=miter,line width=0.212pt] (125.1832,35.0000) -- (125.1832,45.0000);

  \path[draw=black,line cap=butt,line join=miter,line width=0.212pt] (115.1832,35.0000) -- (115.1832,45.0000) -- (115.1832,45.0000);

  \path[draw=black,line cap=butt,line join=miter,line width=0.212pt] (105.1832,35.0000) -- (105.1832,50.0000) -- (120.1832,55.0000) -- (135.1832,50.0000) -- (135.1832,35.0000) -- (135.1832,35.0000);

  \path[draw=black,line cap=butt,line join=miter,line width=0.212pt] (170.0000,55.0000) -- (170.0000,35.0000);

  \path[draw=black,line cap=butt,line join=miter,line width=0.212pt] (175.0000,35.0000) -- (175.0000,45.0000) -- (170.0000,50.0000) -- (165.0000,45.0000) -- (165.0000,35.0000);

  \path[draw=black,line cap=butt,line join=miter,line width=0.212pt] (175.0000,45.0000) -- (180.0000,40.0000) -- (180.0000,35.0000);

  \path[draw=black,line cap=butt,line join=miter,line width=0.212pt] (180.0000,40.0000) -- (185.0000,35.0000);

  \path[draw=black,line cap=butt,line join=miter,line width=0.212pt] (165.0000,45.0000) -- (160.0000,40.0000) -- (160.0000,35.0000);

  \path[draw=black,line cap=butt,line join=miter,line width=0.212pt] (160.0000,40.0000) -- (155.0000,35.0000);

  \path[draw=black,dash pattern=on 0.21pt off 1.69pt,line cap=butt,line join=miter,line width=0.212pt,miter limit=4.00] (5.0000,15.0000) -- (200.0000,15.0000) -- (200.0000,15.0000);

  \path[draw=black,dash pattern=on 0.21pt off 1.69pt,line cap=butt,line join=miter,line width=0.212pt,miter limit=4.00] (4.8168,80.0000) -- (199.8168,80.0000) -- (199.8168,80.0000);

  \path[draw=black,dash pattern=on 0.21pt off 1.70pt,line cap=butt,line join=miter,line width=0.212pt,miter limit=4.00] (199.8168,85.0000) -- (4.8168,85.0000);

  \path[draw=black,dash pattern=on 0.21pt off 1.69pt,line cap=butt,line join=miter,line width=0.212pt,miter limit=4.00] (4.8168,75.0000) -- (199.8168,75.0000);

  \path[draw=black,dash pattern=on 0.21pt off 1.69pt,line cap=butt,line join=miter,line width=0.212pt,miter limit=4.00] (199.8168,70.0000) -- (4.8168,70.0000);

  \path[draw=black,dash pattern=on 0.21pt off 1.69pt,line cap=butt,line join=miter,line width=0.212pt,miter limit=4.00] (4.8168,65.0000) -- (199.8168,65.0000);

  \path[draw=black,line cap=butt,line join=miter,line width=0.212pt] (5.0000,65.0000) -- (5.0000,70.0000) -- (15.0000,75.0000) -- (15.0000,80.0000) -- (15.0000,80.0000) -- (15.0000,80.0000);

  \path[draw=black,line cap=butt,line join=miter,line width=0.212pt] (55.0000,65.0000) -- (55.0000,70.0000) -- (45.0000,75.0000) -- (45.0000,80.0000);

  \path[draw=black,line cap=butt,line join=miter,line width=0.212pt] (10.0000,65.0000) -- (15.0000,70.0000) -- (15.0000,75.0000) -- (15.0000,75.0000);

  \path[draw=black,line cap=butt,line join=miter,line width=0.212pt] (50.0000,65.0000) -- (45.0000,70.0000) -- (45.0000,75.0000) -- (45.0000,75.0000);

  \path[draw=black,line cap=butt,line join=miter,line width=0.212pt] (15.0000,65.0000) -- (15.0000,70.0000) -- (15.0000,70.0000);

  \path[draw=black,line cap=butt,line join=miter,line width=0.212pt] (45.0000,65.0000) -- (45.0000,70.0000) -- (45.0000,70.0000);

  \path[draw=black,line cap=butt,line join=miter,line width=0.212pt] (30.0000,65.0000) -- (30.0000,85.0000);

  \path[draw=black,line cap=butt,line join=miter,line width=0.212pt] (20.0000,65.0000) -- (20.0000,70.0000) -- (30.0000,80.0000) -- (40.0000,70.0000) -- (40.0000,65.0000) -- (40.0000,65.0000);

  \path[draw=black,line cap=butt,line join=miter,line width=0.212pt] (35.0000,65.0000) -- (35.0000,75.0000) -- (35.0000,75.0000);

  \path[draw=black,line cap=butt,line join=miter,line width=0.212pt] (25.0000,65.0000) -- (25.0000,75.0000);

  \path[draw=black,line cap=butt,line join=miter,line width=0.212pt] (15.0000,80.0000) .. controls (30.0000,85.0000) and (30.0000,85.0000) .. (30.0000,85.0000) -- (45.0000,80.0000);

  \path[draw=black,line cap=butt,line join=miter,line width=0.212pt] (75.0000,65.0000) -- (75.0000,75.0000) -- (80.0000,80.0000) -- (85.0000,75.0000) -- (85.0000,65.0000);

  \path[draw=black,line cap=butt,line join=miter,line width=0.212pt] (80.0000,65.0000) -- (80.0000,85.0000);

  \path[draw=black,line cap=butt,line join=miter,line width=0.212pt] (105.0000,65.0000) -- (105.0000,80.0000) -- (110.0000,85.0000) -- (115.0000,80.0000) -- (115.0000,65.0000);

  \path[draw=black,line cap=butt,line join=miter,line width=0.212pt] (110.0000,65.0000) -- (110.0000,85.0000);

  \path[draw=black,line cap=butt,line join=miter,line width=0.212pt] (135.0000,65.0000) -- (140.0000,70.0000) -- (145.0000,75.0000);

  \path[draw=black,line cap=butt,line join=miter,line width=0.212pt] (185.0000,65.0000) -- (180.0000,70.0000) -- (175.0000,75.0000);

  \path[draw=black,line cap=butt,line join=miter,line width=0.212pt] (140.0000,65.0000) -- (140.0000,70.0000) -- (145.0000,75.0000) -- (145.0000,75.0000);

  \path[draw=black,line cap=butt,line join=miter,line width=0.212pt] (180.0000,65.0000) -- (180.0000,70.0000) -- (175.0000,75.0000) -- (175.0000,75.0000);

  \path[draw=black,line cap=butt,line join=miter,line width=0.212pt] (145.0000,65.0000) -- (145.0000,75.0000) -- (145.0000,75.0000);

  \path[draw=black,line cap=butt,line join=miter,line width=0.212pt] (175.0000,65.0000) -- (175.0000,70.0000) -- (175.0000,75.0000);

  \path[draw=black,line cap=butt,line join=miter,line width=0.212pt] (160.0000,65.0000) -- (160.0000,85.0000);

  \path[draw=black,line cap=butt,line join=miter,line width=0.212pt] (150.0000,65.0000) -- (150.0000,70.0000) -- (160.0000,80.0000) -- (170.0000,70.0000) -- (170.0000,65.0000) -- (170.0000,65.0000);

  \path[draw=black,line cap=butt,line join=miter,line width=0.212pt] (165.0000,65.0000) -- (165.0000,75.0000) -- (165.0000,75.0000);

  \path[draw=black,line cap=butt,line join=miter,line width=0.212pt] (155.0000,65.0000) -- (155.0000,75.0000);

  \path[draw=black,line cap=butt,line join=miter,line width=0.212pt] (145.0000,75.0000) -- (160.0000,80.0000) -- (175.0000,75.0000);

  \path[draw=black,dash pattern=on 0.21pt off 1.69pt,line cap=butt,line join=miter,line width=0.212pt,miter limit=4.00] (4.8168,110.0000) -- (199.8168,110.0000) -- (199.8168,110.0000);

  \path[draw=black,dash pattern=on 0.21pt off 1.70pt,line cap=butt,line join=miter,line width=0.212pt,miter limit=4.00] (199.8168,115.0000) -- (4.8168,115.0000);

  \path[draw=black,dash pattern=on 0.21pt off 1.69pt,line cap=butt,line join=miter,line width=0.212pt,miter limit=4.00] (4.8168,105.0000) -- (199.8168,105.0000);

  \path[draw=black,dash pattern=on 0.21pt off 1.69pt,line cap=butt,line join=miter,line width=0.212pt,miter limit=4.00] (199.8168,100.0000) -- (4.8168,100.0000);

  \path[draw=black,dash pattern=on 0.21pt off 1.69pt,line cap=butt,line join=miter,line width=0.212pt,miter limit=4.00] (4.8168,95.0000) -- (199.8168,95.0000);

  \path[draw=black,line cap=butt,line join=miter,line width=0.212pt] (4.9874,95.0000) -- (4.9874,110.0000) -- (9.9874,115.0000) -- (14.9874,110.0000) -- (14.9874,95.0000);

  \path[draw=black,line cap=butt,line join=miter,line width=0.212pt] (9.9874,95.0000) -- (9.9874,115.0000);

  \path[draw=black,line cap=butt,line join=miter,line width=0.212pt] (45.0000,95.0000) -- (50.0000,100.0000);

  \path[draw=black,line cap=butt,line join=miter,line width=0.212pt] (75.0000,95.0000) -- (70.0000,100.0000);

  \path[draw=black,line cap=butt,line join=miter,line width=0.212pt] (50.0000,95.0000) -- (50.0000,100.0000) -- (55.0000,105.0000) -- (55.0000,105.0000);

  \path[draw=black,line cap=butt,line join=miter,line width=0.212pt] (70.0000,95.0000) -- (70.0000,100.0000) -- (65.0000,105.0000) -- (65.0000,105.0000);

  \path[draw=black,line cap=butt,line join=miter,line width=0.212pt] (55.0000,95.0000) -- (55.0000,105.0000) -- (55.0000,105.0000);

  \path[draw=black,line cap=butt,line join=miter,line width=0.212pt] (65.0000,95.0000) -- (65.0000,100.0000) -- (65.0000,105.0000);

  \path[draw=black,line cap=butt,line join=miter,line width=0.212pt] (59.9871,95.0000) -- (59.9871,115.0000);

  \path[draw=black,line cap=butt,line join=miter,line width=0.212pt] (35.0000,100.0000) -- (40.0000,105.0000) -- (59.9871,110.0000) -- (80.0000,105.0000) -- (85.0000,100.0000) -- (85.0000,100.0000);

  \path[draw=black,line cap=butt,line join=miter,line width=0.212pt] (80.0000,95.0000) -- (80.0000,105.0000) -- (80.0000,105.0000);

  \path[draw=black,line cap=butt,line join=miter,line width=0.212pt] (40.0000,95.0000) -- (40.0000,105.0000);

  \path[draw=black,line cap=butt,line join=miter,line width=0.212pt] (55.0000,105.0000) -- (59.9871,110.0000) -- (65.0000,105.0000);

  \path[draw=black,line cap=butt,line join=miter,line width=0.212pt] (35.0000,100.0000) -- (35.0000,95.0000) -- (35.0000,95.0000) -- (35.0000,95.0000) -- (35.0000,95.0000);

  \path[draw=black,line cap=butt,line join=miter,line width=0.212pt] (85.0000,100.0000) -- (85.0000,95.0000);

  \path[draw=black,line cap=butt,line join=miter,line width=0.212pt] (120.0000,115.0000) -- (120.0000,95.0000);

  \path[draw=black,line cap=butt,line join=miter,line width=0.212pt] (115.0000,95.0000) -- (115.0000,105.0000) -- (115.0000,110.0000) -- (120.0000,115.0000) -- (125.0000,110.0000) -- (125.0000,95.0000) -- (125.0000,95.0000) -- (125.0000,95.0000);

  \path[draw=black,line cap=butt,line join=miter,line width=0.212pt] (125.0000,105.0000) -- (130.0000,100.0000) -- (130.0000,95.0000);

  \path[draw=black,line cap=butt,line join=miter,line width=0.212pt] (130.0000,100.0000) -- (135.0000,95.0000);

  \path[draw=black,line cap=butt,line join=miter,line width=0.212pt] (115.0000,105.0000) -- (110.0000,100.0000) -- (110.0000,95.0000);

  \path[draw=black,line cap=butt,line join=miter,line width=0.212pt] (110.0000,100.0000) -- (105.0000,95.0000);

  \path[draw=black,line cap=butt,line join=miter,line width=0.212pt] (155.0000,95.0000) -- (155.0000,105.0000) -- (170.0000,110.0000) -- (185.0000,105.0000) -- (185.0000,95.0000);

  \path[draw=black,line cap=butt,line join=miter,line width=0.212pt] (160.0000,95.0000) .. controls (160.0000,100.0000) and (160.0000,100.0000) .. (160.0000,100.0000) -- (170.0000,110.0000) -- (180.0000,100.0000) -- (180.0000,95.0000);

  \path[draw=black,line cap=butt,line join=miter,line width=0.212pt] (165.0000,95.0000) -- (165.0000,105.0000) -- (165.0000,105.0000);

  \path[draw=black,line cap=butt,line join=miter,line width=0.212pt] (175.0000,95.0000) -- (175.0000,105.0000);

  \path[draw=black,line cap=butt,line join=miter,line width=0.212pt] (170.0000,95.0000) -- (170.0000,115.0000);

  \path[draw=black,dash pattern=on 0.21pt off 1.69pt,line cap=butt,line join=miter,line width=0.212pt,miter limit=4.00] (4.8168,140.0000) -- (199.8168,140.0000) -- (199.8168,140.0000);

  \path[draw=black,dash pattern=on 0.21pt off 1.70pt,line cap=butt,line join=miter,line width=0.212pt,miter limit=4.00] (199.8168,145.0000) -- (4.8168,145.0000);

  \path[draw=black,dash pattern=on 0.21pt off 1.69pt,line cap=butt,line join=miter,line width=0.212pt,miter limit=4.00] (4.8168,135.0000) -- (199.8168,135.0000);

  \path[draw=black,dash pattern=on 0.21pt off 1.69pt,line cap=butt,line join=miter,line width=0.212pt,miter limit=4.00] (199.8168,130.0000) -- (4.8168,130.0000);

  \path[draw=black,dash pattern=on 0.21pt off 1.69pt,line cap=butt,line join=miter,line width=0.212pt,miter limit=4.00] (4.8168,125.0000) -- (199.8168,125.0000);

  \path[draw=black,line cap=butt,line join=miter,line width=0.212pt] (20.0000,145.0000) -- (20.0000,125.0000);

  \path[draw=black,line cap=butt,line join=miter,line width=0.212pt] (15.0000,125.0000) -- (15.0000,135.0000) -- (15.0000,140.0000) -- (20.0000,145.0000) -- (25.0000,140.0000) -- (25.0000,125.0000) -- (25.0000,125.0000) -- (25.0000,125.0000);

  \path[draw=black,line cap=butt,line join=miter,line width=0.212pt] (25.0000,135.0000) -- (30.0000,130.0000) -- (30.0000,125.0000);

  \path[draw=black,line cap=butt,line join=miter,line width=0.212pt] (30.0000,130.0000) -- (35.0000,125.0000);

  \path[draw=black,line cap=butt,line join=miter,line width=0.212pt] (15.0000,135.0000) -- (10.0000,130.0000) -- (10.0000,125.0000);

  \path[draw=black,line cap=butt,line join=miter,line width=0.212pt] (10.0000,130.0000) -- (5.0000,125.0000);

  \path[draw=black,line cap=butt,line join=miter,line width=0.212pt] (65.0000,125.0000) -- (65.0000,135.0000) -- (70.0000,140.0000) -- (75.0000,135.0000) -- (75.0000,125.0000);

  \path[draw=black,line cap=butt,line join=miter,line width=0.212pt] (60.0000,125.0000) -- (60.0000,135.0000) -- (70.0000,140.0000) -- (80.0000,135.0000) -- (80.0000,125.0000);

  \path[draw=black,line cap=butt,line join=miter,line width=0.212pt] (70.0000,125.0000) -- (70.0000,145.0000);

  \path[draw=black,line cap=butt,line join=miter,line width=0.212pt] (55.0000,125.0000) -- (55.0000,130.0000) -- (60.0000,135.0000);

  \path[draw=black,line cap=butt,line join=miter,line width=0.212pt] (80.0000,135.0000) -- (85.0000,130.0000) -- (85.0000,125.0000);

\end{tikzpicture}
\end{center}
\end{figure}

In the $ A_n $ case, we can define a similar $ k $-pruning map $ P'_k $ by summing all non-symmetric $ k $-prunings. For $ k = 0 $, $ P'_0 $ is a chain map, and it is shown in \cite{Sinha:13} to induce the coproduct in cohomology. This statement is not true in the $ B_n $ case because the differential of an antisymmetric planar level tree with labels behaves badly near the central ``trunk'' labeled $ 0 $. Nevertheless, at each level $ k $, away from this central stem, this is essentially true. For this intuitive reason, we must define our cochain-level coproduct map differently: prune a symmetric planar level tree at every level and tensor it with a symmetric planar level tree whose principal $ k $-blocks, as defined below in Definition \ref{def:principal block}, are the scraps of the performed prunings. To prove this statement, we need some preliminary calculations.

Suppose that a symmetric planar level tree $ T $ corresponds to $ [a_0: \dots :a_{n-1}] \in FN_{\WB{n}} $. In that case, consider the set of couples of adjacent edges $ (e,f) $ (with $ e<f$) in $ T $ having the same source vertex and belonging to the positive half-plane $ \{(x,y): x \geq 0\} $. This set in bijective correspondence with $ \{0,\dots,n-1\} $, and the height of the common vertex of the couple $ (e,f) $ corresponding to $ i $ is $ a_i $ . This bijection is explicitly given by counting the leaves in the positive half-plane that lie on the left of $ e $. For $ 0 \leq i \leq n-1 $, we denote with $ d_i(T) $ or equivalently with $ d_{e,f} $ the sum of the addends $ d^S_{e,f,\sigma} $ of the differential $ d $, as expressed in Proposition \ref{prop:tree differential}, in which a vertex shuffle constructed from the couple $ (e,f) $ corresponding to $ i $ appear. Thus, we have $ d(T) = \sum_{i=0}^{n-1} d_i(T) $.
\begin{lemma} \label{lem:pruning maps differential}
Let $ T $ be a symmetric planar level tree corresponding to $ [a_0:\dots:a_{n-1}] $. Let $ m_{k} $ be the smallest index such that $ a_{m_{k}} = k $. Let $ I $ be the trivial symmetric planar level tree. Then the following statements are true:
\begin{enumerate}
\item the pruning maps and the differential satisfy the equality
\[
P_k d + dP_k + (\id \otimes d_{0})(P_k - \id \otimes I) = (\id \otimes d_{m_{k-1}}) (\id \otimes C) (P_k \otimes \id) P_{k-1}^{\min}.
\]
\item $ P_k (T) = T \otimes I $ if $ a_i < k $ for all $ 0 \leq i < n $
\item for all $ \bm{a} = [a_0:\dots.a_{n-1}] $ and $ \bm{b} = [b_0:\dots:b_{m-1}] $ such that $ b_0 < \min\{a_0,\dots,a_{n-1}\} $, we have $ d_i C(\bm{a} \otimes \bm{b}) = C (d_i \otimes \id) (\bm{a} \otimes \bm{b}) $ for $ 0 \leq i < n $ and $ d_i C(\bm{a} \otimes \bm{b}) = C(\id \otimes d_{i-n}) (\bm{a} \otimes \bm{b}) $ for $ n < i <n+m $, and also for $ i = n $ if $ b_0 < \min\{a_0,\dots,a_{n-1}\}-1 $
\item $ (\id \otimes C) (P_k^{\min} \otimes \id) P_k (T) = P_k(T) - T \otimes I $
\item $ C (C \otimes \id) = C(\id \otimes C) $
\end{enumerate}
\end{lemma}
\begin{proof}
The statements from 2 to 5 are easy. Regarding 2, if $ a_i < k $ for all $ i $, $ T $ has no vertex of height $ k $ with more than one outgoing edge. Thus the only possible symmetric $ k $-pruning is the trivial one. Regarding 3, the bijection $ \varphi \colon \{0,\dots,n-1\} \sqcup \{0,\dots,m-1\} \to \{0,\dots,n+m-1\} $ that shifts elements of $ \{0,\dots,m-1\} $ by $ n $ yields a bijection between pairs $ (e,f) $ of adjacent edges of the symmetric planar level tree $ T $ corresponding to $ C(\bm{a} \otimes \bm{b}) $ and those of the symmetric planar level trees $ T' $ and $ T'' $ corresponding to $ \bm{a} $ and $ \bm{b} $ respectively. If $ b_0 < \min\{a_0,\dots,a_{n-1}\} $, then for all $ i \in \{0,\dots,n+m-1\} $, with the only possible exception of $ i=n $, this bijection preserves $ E(v_{e_i}) $ and $ E(v_{f_i}) $, where $ v_{e_i} $ and $ v_{f_i} $ are the target vertices of the corresponding pair of edges $ (e_i,f_i) $. The edges in $ E(v_{e_i}) $ and $ E(v_{f_i}) $ of the corresponding pair come either both from $ T' $ or both from $ T'' $. Hence $ d_i C(\bm{a} \otimes \bm{b}) = d_{\varphi^{-1}(i)} \bm{a} \otimes \bm{b} $. If $ b_0 < \min\{a_0,\dots,a_{n-1}\} -1 $ the same is also true for the edges $ e_n $ and $ f_n $, so the equality is satisfied also in this case. 4 is immediate from the definition of $ k $-prunings and the combinatorics of shuffles, and 5 is obvious.

On the contrary, 1 is more complicated and requires a more detailed proof.
As a notational convention, let $ d^h_l = \sum_{i: a_i = l} d_i $, the sum of the contributions to the differential coming from vertices at height $ k $. We compare $ d^h_l P_k(T) $ with $ P_k d^h_l(T) $. We consider different cases depending on the difference between $ k $ and $ l $.
\begin{itemize}
\item If $ l > k $, $ d^h_l $ is computed by gluing together a pair $ (e,f) $ of adjacent edges of height bigger than $ k $ (and its mirror pair) and performing a shuffle at the new target vertex. These operations only change a connected subtree whose vertices all have height bigger than $ k $, and, by construction, $ k $-prunings commute with such operations. Hence $ d^h_lP_k = P_k d^h_l $.
\item If $ l = k $, then we can write $ d^h_k P_k(T) = \sum_{(T',T'')} \sum_{(e,f)} d_{e,f} (T' \otimes S_k(T'')) $, where the sum is over symmetric $ k $-prunings $ (T',T'') $ of $ T $ and pairs of adjacent edges $ (e,f) $ in the positive half-plane with a common source vertex of height $ k $ in $ T' $ or $ S_k(T'') $. We also note that $ S_k(T'') $ has a unique vertex $ w $ of height $ k $.
There is an obvious bijection
\[
\bigsqcup_{v \in V(T): h(v) = k} E(v) \leftrightarrow \bigsqcup_{u \in V(T'): h(v) = k} E(u) \sqcup \left( E(w) \setminus \{e_0\} \right)
\]
that maps an edge to its image in $ T' $ (if it is not pruned away) or in $ S_k(T'') $ (if it is), and that arises from the fact that, for elementary prunings, $ T = T' \cup T'' \cup r(T'') $. $ e_0 $ is the unique edge belonging to the central vertical stem whose source vertex is $ w $. Moreover, this bijection preserves the properties of belonging to the positive and negative half-plane. Therefore, we can write the summation above expressing $ d^h_k P_k(T) $ as the sum of three pieces:
\begin{itemize}
\item The first piece is the sum of the terms corresponding to $ (e,f) $ such that $ (e,f) $ come from adjacent edges in $ T $. These terms correspond to symmetric $ k $-prunings of $ d^S_{e,f,\sigma}(T) $, for shuffles $ \sigma $ at the common vertex of $ e $ and $ f $. Hence, their sum yields $ P_k d^h_k(T) $.
\item The second piece is the sum of the terms corresponding to $ (e,f) $ in $ S_k(T'') $ such that $ e \not= e_0 $ and $ (e,f) $ do not come from adjacent vertices of $ T $. Under this condition, the symmetric vertex permutation $\sigma(S_k(T'')) $ of $ S_k(T'') $ at $ w $ that switched the positions of $ e $ and $ f $ still produces a shuffle of the scraps of the elementary prunings involved in $ (T',T'') $. Every tree in $ d_{(e,f)}(S_k(T'')) $ cancel out with a tree in $ d_{(f,e)}(\sigma(S_k(T''))) $. Hence, this second piece is $ 0 $.
\item The third piece is given by the terms corresponding to $ (e,f) $ with $ e = e_0 $. These terms yield $ (\id \otimes d_0) P_k(T) $.
\end{itemize}
Finally, we deduce that $ d^h_k P_k(T) = (\id \otimes d_0) P_k(T) + P_k d^h_k(T) $.
\item If $ l = k-1 $, $ P_k d^h_{k-1}(T) = \sum_{(e,f)} \sum_{(T',T'')} T' \otimes S_k(T'') $, where the sum is taken over couples $ (e,f) $ of adjacent edges in $ T $ whose common source $ v $ has height $ k-1 $, and symmetric $ k $-prunings $ (T',T'') $ of trees in $ d_{(e,f)}(T) $. Let $ v_e $ and $ v_f $ be the targets of $ e $ and $ f $, respectively. By construction, $ d_{(e,f)}(T) $ glues $ v_e $ and $ v_f $ to a single vertex $ \overline{v} $, such that $ E(\overline{v}) = E(v_e) \sqcup E(v_f) $, suitably shuffled. Let $ A $ be the set of edges removed by the corresponding elementary symmetric prunings at $ \overline{v} $ and at $ r(\overline{v}) $, the mirror vertex of $ \overline{v} $ (which might coincide).
We retroeve symmetric $ k $-prunings for which $ E(v_e) \not\subseteq A $ and $ E(v_f) \not\subseteq A $ from symmetric $ k $-prunings $ (T',T'') $ of $ T $ by applying $ d_{(e,f)} $ to $ T' $.
Now assume that $ v $ is not on the central stem of the tree. If $ e \not= \min(E(v)) $, it is the successor of an edge $ g \in E(v) $, and the terms of $ P_k d_{e,f}(T) $ for which $ E(v_e) \subseteq A $ cancel out with the terms of $ P_k d_{g,e}(T) $ for which $ E(v_e) \subseteq A $. Similarly, all the terms for which $ E(v_f) \subseteq A $ and $ f \not= \max(E(v)) $ cancel out. The only remaining terms are those in which we remove an entire subtree corresponding to $ \min(E(v)) $ (which is the mirror image of $ \max(E(r(v)) $). If $ v $ belongs to the central axis, we must slightly modify the argument to take into account only edges in the positive half-plane and shows that the surviving terms are those in which an entire subtree stemming from $ \max(E(v)) $ is removed. The sum of all these elements is exactly equal to the correcting term $ (\id \otimes d_{m_{k-1}}) (\id \otimes C) (P_k \otimes \id) P_{k-1}^{\min}(T) $.
We deduce that
\[
d^h_{k-1}P_k(T) = P_k d^h_{k-1}(T) + (\id \otimes d_{m_{k-1}}) (\id \otimes C) (P_k \otimes \id) P_{k-1}^{\min}(T).
\]
\item If $ l < k-1 $, since $ k $-prunings only depend on the part of the tree above height $ k $ and $ d^h_l $ does not change it, the same argument used for $ l > k $ shows that $ d^h_l P_k = P_k d^h_l $.
\end{itemize}
Combining the equalities obtained in these cases yields 1.
\end{proof}

We are now ready to construct a cochain representative of the coproduct map $ H^*(\WB{n}) \to \bigoplus_{i=0}^n H^*(\WB{i}) \otimes H^*(\WB{n-i}) $.

\begin{proposition} \label{prop:coprodotto geometrico}
Let $ \Delta_k \colon FN^*_{\WB{n}} \otimes \mathbb{F}_2 \to \bigoplus_{i=0}^n FN^*_{\WB{i}} \otimes FN^*_{\WB{n-i}} \otimes \mathbb{F}_2 $ be the linear maps defined recursively by the following formulas:
\begin{itemize}
\item $ \Delta_0 = P_0 $
\item for $ k > 0 $, $ \Delta_k = (\id \otimes C)(P_k \otimes \id) \Delta_{k-1} $
\end{itemize}
The following statements are true:
\begin{enumerate}
\item the limit $ \Delta = \varinjlim \Delta_k $ exists
\item $ \Delta $ is a cochain map
\item $ \Delta $ represents the cohomology coproduct map at the cochain level
\end{enumerate}
\end{proposition}
\begin{proof}
\begin{enumerate}
\item Let $ \bm{a} \in FN^*_{\WB{n}} $ and let $ m = \max\{a_0,\dots,a_{n-1}\} $. Statement 2 of Lemma \ref{lem:pruning maps differential} guarantees that $ \Delta_k(\bm{a}) = \Delta_m(\bm{a}) $ for all $ k>m $. Thus, the sequence $ \{\Delta_k\}_{k=0}^\infty $ stabilizes and consequently has a limit.
\item We first observe that Lemma \ref{lem:pruning maps differential} (4) and (5) imply that $ (\id \otimes C)(P^{\min}_k \otimes \id) \Delta_k = \Delta_k - \Delta_{k-1} $ for all $ k \geq 0 $, with the convention that $ \Delta_{-1}(T) = T \otimes I $. Combining this remark with Lemma \ref{lem:pruning maps differential} (3) and (5), we obtain that, for all $ k\geq1$,
\begin{align*}
&\-\-\-(\id \otimes C)(\id \otimes d_{m_{k-1}} \otimes \id)(\id \otimes C \otimes \id )(P_k \otimes \id \otimes \id )(P_{k-1}^{\min} \otimes \id) \Delta_{k-1} \\
&= (\id \otimes d_{m_{k-1}}) (\id \otimes C)(\id \otimes C \otimes \id)(P_k \otimes \id \otimes \id)(P_{k-1}^{\min} \otimes \id) \Delta_{k-1} \\
&= (\id \otimes d_{m_{k-1}}) (\id \otimes C)(P_k \otimes C)(P_{k-1}^{\min} \otimes \id) \Delta_{k-1} \\
&= (\id \otimes d_{m_{k-1}}) (\id \otimes C)(P_k \otimes \id)(\Delta_{k-1} - \Delta_{k-2}).
\end{align*}
We use this to prove by induction on $ k $ that $ \Delta_k d = d \Delta_k + ( \id \otimes d_0)(\Delta_k - \Delta_{k-1}) $. For $ k = 0 $ this identity is the content of the first statement of Lemma \ref{lem:pruning maps differential}. For $ k>0 $, we deduce from the identity above and the previous lemma that:
\begin{align*}
\Delta_k d &= (\id \otimes C)(P_k \otimes \id) \Delta_{k-1} d \\
&= (\id \otimes C) ( P_k \otimes \id) d \Delta_{k-1} + (\id \otimes C)(P_k \otimes d_0)(\Delta_{k-1} - \Delta_{k-2}) \\
&= (\id \otimes C)(P_k \otimes \id)d \Delta_{k-1} + (\id \otimes C)d(P_k \otimes \id)(\Delta_{k-1} - \Delta_{k-2}) \\
&\quad + (d - d_{m_{k-1}})(\id \otimes C)(P_k \otimes \id)(\Delta_{k-1} - \Delta_{k-2}) \\
&= (\id \otimes C)d (P_k \otimes \id)\Delta_{k-1} \\
&\quad + (\id \otimes C)(\id \otimes d_0 \otimes \id) [(P_k - \id \otimes I) \otimes \id] \Delta_{k-1} \\
&\quad + (\id \otimes C) d (P_k \otimes \id ) ( \Delta_{k-1} - \Delta_{k-2}) \\
&\quad + d(\id \otimes C)(P_k \otimes \id)(\Delta_{k-1} - \Delta_{k-2}) \\
&= (\id \otimes C)d (P_k \otimes \id) \Delta_{k-2} + (\id \otimes d_0)(\id \otimes C)[(P_k - \id \otimes I) \otimes \id ] \Delta_{k-1} \\
&\quad + d \Delta_k - d(1 \otimes C)(P_k \otimes \id)\Delta_{k-2} \\
&= d \Delta_k - (\id \otimes d_0)(\Delta_k - \Delta_{k-1}) + (\id \otimes C) d (P_k \otimes \id)\Delta_{k-2} \\
&\quad - d (\id \otimes C) (P_k \otimes \id) \Delta_{k-2} \\
&= d \Delta_k - (\id \otimes d_0)(\Delta_k - \Delta_{k-1}).
\end{align*}
To justify the last equality, we observe that $ (P_k \otimes \id)\Delta_{k-2} $ is a sum of terms of the form $ \bm{c} \otimes \bm{a} \otimes \bm{b} $ with $ b_0 < \min\{a_i\} - 1 $, and we apply the stronger clause of Lemma \ref{lem:pruning maps differential} (3).

Now the identity $ d \Delta = \Delta d $ follows by passing to the limit, and using that the sequence $ \{ \Delta_k \}_{k=0}^\infty $ stabilizes.
\item Consider the dg-module $ U $ over $ \Ftwo $ with basis given by symmetric planar level trees with antisymmetric labels in any finite subset $ I \subseteq \mathbb{N} $, not necessarily $ \{-n,\dots,n\} $, with the symmetric tree differential. Note that  $ \bigoplus_{n \geq 0} \widetilde{FN}^*_{\WB{n}} \otimes \mathbb{F}_2 $ embeds in $ U $ in the obvious way. We observe that the linear maps $ P_k $, $ P^{\min}_k $, and $ C $ lift to linear maps $ \tilde{P}_k, \tilde{P}^{\min}_k \colon U \to U \otimes U $ and $ \tilde{C} \colon U \otimes U \to U $. $ \tilde{P}_k $ and $ \tilde{P}^{\min}_k $ are still defined via prunings, but we additionally keep track of the labels of the subtrees involved. We compute $ \tilde{C} $ on $ T' \otimes T'' $ by splitting $ T'' $ symmetrically along the vertical axis, keeping labels, and symmetrically attach the two parts to $ T' $ to obtain a new basis element of $ U $. Lemma \ref{lem:pruning maps differential} still holds for this labeled version of the morphisms by the same proof. Consequently, there is a labeled version $ \tilde{\Delta} \colon U \to U \otimes U $ of $ \Delta $, constructed recursively via finite approximations $ \tilde{\Delta}_k $, that still commutes with the differential.
Note that we can also embed $ \widetilde{FN}^*_{\WB{n}} \otimes \widetilde{FN}^*_{\WB{m}} \otimes \mathbb{F}_2 $ into $ U \otimes U $ by  keeping the labels of trees in $ \widetilde{FN}^*_{\WB{n}} $ and relabeling trees in $ \widetilde{FN}^*_{\WB{m}} $ via the bijection $ \{0,\dots,m-1\} \to \{n,\dots,n+m-1\} $ that raises numbers by $ n $. There is also a projection $ U \times U \to \widetilde{FN}^*_{\WB{n}} \otimes \widetilde{FN}^*_{\WB{m}} \otimes \mathbb{F}_2 $ that maps every basis element of $ U \otimes U $ that do not belong to $ \widetilde{FN}^*_{\WB{n}} \otimes \widetilde{FN}^*_{\WB{m}} \otimes \mathbb{F}_2 $ to $ 0 $.
By induction, we easily see that restricting each $ \tilde{\Delta}_k $ for all $ k $ (and, consequently, $ \tilde{\Delta} $) to $ \widetilde{FN}^*_{\WB{n}} \otimes \widetilde{FN}^*_{\WB{m}} \otimes \mathbb{F}_2 $ and composing with this projection we obtain linear maps $ \bigoplus_{n \geq 0} \widetilde{FN}^*_{\WB{n}} \otimes \mathbb{F}_2 \to \bigoplus_{n \geq 0} \widetilde{FN}^*_{\WB{n}} \otimes \bigoplus_{n \geq 0} \widetilde{FN}^*_{\WB{n}} \otimes \mathbb{F}_2 $ that are equivariant with respect to the monomorphisms $ \WB{n} \times \WB{m} \to \WB{n+m} $ and satisfy the same formal relation with respect to the differential.
By identifying $ FN^*_{\WB{n}} $ with the invariant subspace $ (\widetilde{FN}^*_{\WB{n}})^{\WB{n}} $, the limit map $ \tilde{\Delta} $ restricts to $ \Delta $, which is thus a cochain-level realization of the coproduct map.
\end{enumerate}
\end{proof}

We now turn our attention to the transfer map. We need a preliminary definition.
\begin{definition}{\rm (partially from \cite{Sinha:13})} $ \, \, \, $ \label{def:principal block}
Let $ \bm{a} = \left[ a_0: \dots: a_{n-1} \right] \in FN^*_{\WB{n}} $ be as defined above. In what follows, we assume, by convention, that $ a_{-1} = a_{n} = 0 $. We say that the chain $ \left[ a_i: \dots: a_j \right] $ is a \emph{$ k $-block} of $ \bm{a} $ if $ \forall i \leq l \leq j: \mbox{ } a_l > k $ and $ \max\left\{a_{i-1},a_{j+1} \right\} \leq k $. We say that a $ k $-block $ \left[a_i: \dots: a_j \right] $ of $ \bm{a}  $ is \emph{principal} if, in addition, $ \min_{0 \leq r < i} a_r = k $. We denote by $ \PBl_k(\bm{a}) $ the tuple of the principal $ k $-blocks of $ \bm{a} $, ordered from left to right.
\end{definition}
For example, $ \bm{a} = \left[ 3:2:3:1:2:1:3:2:0:3 \right] $ has four $ 1 $-blocks: $ B_{1,1} = \left[ 3:2:3 \right] $, $ B_{1,2} = \left[ 2 \right] $, $ B_{1,3} = \left[3:2\right] $, and $ B_{1,4} = \left[3\right] $. $ \PBl_1(\bm{a}) = (B_{1,2},B_{1,3}) $.

Note that a basis element $ \bm{a} $ is uniquely determined by $ \{ \PBl_k(\bm{a}) \}_{k=0}^\infty $, the collection of its principal blocks. To retrieve $ \bm{a} $ from these data, we can use the following procedure. First, for all $ k  \geq 0 $, add an entry equal to $ k $ before each principal $ k $-block and concatenate all such tuples to obtain $ \bm{a}_k $. Then, we obtain $ \bm{a} $ as the concatenation of $ \dots, \bm{a}_k, \bm{a}_{k-1},\dots, \bm{a}_0 $. This sequence is necessarily finite because for $ k > \max_{i=0}^{n-1} a_i $ $ \PBl(\bm{a}) $ is the empty $ 0 $-uple.
With this method, we can construct a basis element $ \bm{a} $ from an eventually empty collection of tuples $ \{ B_k \} $, where the entries of $ B_k $ are tuples of natural numbers strictly bigger than $ k $. 

We also observe that $ k $-blocks can be retrieved from the corresponding symmetric level tree $ T $. They are given by the connected components of $ T \cap \{(x,y) \in \mathbb{R}^2: x \geq 0, y > k \} $. Interpreted this way, a $ k $-block is principal if and only if it does not intersect the central vertical axis but is contained in the $ (k-1) $-block intersecting it.

\begin{proposition} \label{prop:transfer geometrico B}
Given $ \bm{a} \in FN^*_{\WB{n}} $, $ \bm{b} \in FN^*_{\WB{m}} $ and $ k \geq 0 $, let $ n_{a,k} $ and $ n_{b,k} $ be the lengths of $ \PBl_k(\bm{a}) $ and $ \PBl_k(\bm{b}) $, respectively. Given a sequence $ \underline{\sigma} = \{ \sigma_k \}_{k=0}^{\infty} $ of permutations $ \sigma_k \in \Sigma_{n_{a,k}+n_{b,k}} $, define $ \underline{\sigma}(\bm{a},\bm{b}) $ as the unique basis elements of $ FN^*_{\WB{n+m}} $ such that, for all $ k \geq 0 $, $ \PBl_k(\underline{\sigma}(\bm{a},\bm{b})) = \sigma_k(\PBl_k(\bm{a}),\PBl_k(\bm{b})) $, where $ ( \PBl_k(\bm{a}), \PBl_k(\bm{b})) $ is the concatenated $ (n_{a,k}+n_{b,k}) $-tuple and $ \sigma_k $ acts on $ (n_{a,k}+n_{b,k}) $-tuples by permuting the entries.
Let $ \odot \colon FN^*_{\WB{n}} \otimes FN^*_{\WB{m}} \otimes \Ftwo \rightarrow FN^*_{\WB{n+m}} \otimes \Ftwo $ be the homomorphism that maps $ \bm{a} \otimes \bm{b} $ to the sum $ \sum_{\underline{\sigma}}\underline{\sigma}(\bm{a},\bm{b}) $ over sequences of permutations $ \underline{\sigma} = \{ \sigma_k \}_{k=0}^\infty $ such that $ \sigma_k $ is a $ (n_{a,k},n_{b,k}) $-shuffle for all $ k \geq 0 $. Informally, $ \bm{a} \odot \bm{b} $ is the sum of basis elements whose principal $ k $-blocks are obtained by shuffling the principal $ k $-blocks of $ \bm{a} $ and $ \bm{b} $.
This defines a morphism of complexes that induces the transfer product in cohomology.
\end{proposition}
\begin{proof}
The reflection arrangement of $ \WB{n} \times \WB{m} $, with its product reflection action on $ \mathbb{R}^n \times \mathbb{R}^m $, is $ \mathcal{A}_{B_n} \times \mathcal{A}_{B_m} = \{ H \times \mathbb{R}^m \}_{H \in \mathcal{A}_{B_n}} \cup \{ \mathbb{R}^n \times H' \}_{H' \in \mathcal{A}_{B_m}} $. Being it a sub-arrangement of $ \mathcal{A}_{B_{n+m}} $, we have a natural inclusion $ Y^{(\infty)}_{\WB{n+m}} \to Y^{(\infty)}_{\mathcal{A}_{B_n} \times \mathcal{A}_{B_m}} \cong Y^{(\infty)}_{\WB{n}} \times Y^{(\infty)}_{\WB{m}} $. We can explicitly obtain such inclusion by splitting a configuration of $ n+m $ points into the two sub-configurations consisting of its first $ n $ points and its last $ m $ points, respectively, and relabeling the indices of the second one. This map is a $ \WB{n} \times \WB{m} $-equivariant homotopy equivalence.

Therefore, passing to quotients, this yields a map $ \pi \colon \frac{ Y^{(\infty)}_{\WB{n+m}} }{ \WB{n} \times \WB{m} } \to \frac{ Y^{(\infty)}_{\WB{n}} }{ \WB{n} } \times \frac{ Y^{(\infty)}_{\WB{m}} }{ \WB{m} } $ that models the standard homotopy equivalence $ B \left( \WB{n} \times \WB{m} \right) \simeq B \left( \WB{n} \right) \times B \left( \WB{m} \right) $.
Moreover, the obvious quotient map $ \pi' \colon \frac{ Y^{(\infty)}_{\WB{n+m}} }{ \WB{n} \times \WB{m} } \to \frac{ Y^{(\infty)}_{\WB{n+m}} }{ \WB{n+m} } $ is a covering model for $ B \left( \WB{n} \times \WB{m} \right) \to B \left( \WB{n+m} \right) $.

Let $ x = \left[ a_0: \dots: a_{n-1} \right] \otimes \left[ b_0: \dots: b_{m-1} \right] $ be a basis element for the Fox-Neuwirth complex $ FN^*_{B_n} \otimes FN^*_{B_m} \otimes \Ftwo $. Let $ \sigma $ be a smooth singular simplex transverse to our strata. By construction, the evaluation of $ \left[a_0: \dots:a_{n-1} \right] \odot \left[ b_0: \dots: b_{m-1} \right] $ on $ \sigma $ is the sum of the evaluations of $ x $ on $ \pi \left( \tilde{\sigma} \right) $, as $ \tilde{\sigma} $ varies among all liftings of $ \sigma $. A direct calculation shows that some $ \pi \left( \tilde{\sigma} \right) $ intersects the stratum corresponding to $ x $ if and only if $ \sigma $ intersects some stratum $ e \left( \bm{c} \right) $, where the $ k $-principal blocks of $ \bm{c} $ are obtained by shuffling the $ k $-principal blocks of $ \bm{a} $ and $ \bm{b} $. This proves the lemma.
\end{proof}

We conclude the treatment of the structural maps on the cohomology of $ \WB{n} $ with some potentially helpful remarks. Since we will not use these facts in this paper, we will not provide complete statements nor proofs of these last claims. Nevertheless, it should be straightforward, although notationally heavy, to fill in the details.
\begin{remark}
\begin{enumerate}
\item The transfer and coproduct maps commute already at the cochain level. To see this, you can observe that, by construction, $ \Delta(\bm{a}) $ is a sum of tensors $ \bm{a}' \otimes \bm{a}'' $ where $ \PBl_k(\bm{a}'') $ is given by the leftovers of symmetric $ k $-prunings of $ \bm{a} $, suitably shuffled, and that the pruning map $ P_k $ itself commute with $ \odot $.
\item The same constructions of the coproduct map in terms of prunings and the transfer map in terms of principal block shuffles can be generalized to the cohomology with integral coefficients. In these cases, additional signs that we can compute from those appearing in \ref{teo:Salvetti} are required.
\end{enumerate}
\end{remark}

\subsection{Structural morphisms: $ A_D $}

The coproduct and the transfer product for $ \WD{n} $ are described geometrically, similarly to what we did for $ \WB{n} $. However, some complications arise. For example, we can not repeat the proof of Proposition \ref{prop:transfer geometrico B} as it is for $ FN^*_{\WD{n}} $, because, in this case, a product of strata $ S\times S'\subseteq Y^{(\infty)}_{\WD{n}} \times Y^{(\infty)}_{\WD{m}} $ is not necessarily the closure of a union of strata in $ Y^{(\infty)}_{\WD{n+m}} $.
However, these ideas adapt well to the cochain complex $ {FN'}_{\WD{n}}^* $, which we will use in the following as a cochain model. We can retrieve the identities we need in $ FN^*_{\WD{n}} $ by using the equivalence $ \varphi $ of Lemma \ref{lem:CHE}.

We can now state the formulas parallel to Proposition \ref{prop:coprodotto geometrico} and Proposition \ref{prop:transfer geometrico B} for $ \WD{n} $.
First, we consider the following oriented versions of the pruning and concatenation maps. Given a symmetric $ k $-pruning $ (T',T'') $ of a symmetric planar level tree $ T $, let $ \mathcal{O} $ and $ \mathcal{O}' $ be orientations of $ T $ and $ T' $ respectively. Fix an antisymmetric labeling $ \lambda' $ of $ T' $ inducing $ \mathcal{O}' $, and an antisymmetric labeling $ \lambda $ of $ T $ inducing $ \mathcal{O} $ such that its restriction to $ T' $, seen as a subtree of $ T $, is $ \lambda' $. By keeping track of the labels of scraps, $ \lambda $ induces an antisymmetric labeling $ \lambda'' $ on $ S_k(T'') $ and, consequently, an orientation $ \mathcal{O}'' $. Unless the $ k $-pruning is trivial, it is always possible to find such labelings $ \lambda $ and $ \lambda' $, and the resulting orientation $ \mathcal{O}'' $ only depends on $ \mathcal{O} $ and $ \mathcal{O}' $.
\begin{definition}
Let $ k \in \mathbb{N} $ and let $ (T,\mathcal{O}) $ be an oriented symmetric planar level tree. An \emph{oriented $ k $-pruning} of $ (T,\mathcal{O}) $ is a quadruple $ (T', \mathcal{O}',T'',\mathcal{O''}) $ where:
\begin{itemize}
\item $ (T',T'') $ is a $ k $-pruning of $ T $
\item $ \mathcal{O}' $ is an orientation of $ T' $
\item $ \mathcal{O}'' $ is the orientation of $ S_k(T'') $ determined from $ \mathcal{O} $ and $ \mathcal{O}' $ via the procedure above
\end{itemize}
An oriented $ k $-pruning of $ (T',\mathcal{O}',T'',\mathcal{O}'') $ of $ T $ is called \emph{positive} (respectively \emph{negative}) if $ \mathcal{O}' $ is the positive (respectively negative) orientation of $ T' $.
\end{definition}
Given a non-trivial $ k $-pruning $ (T',T'') $ of $ T $, there are precisely two ways to extend it to an oriented $ k $-pruning $ (T',\mathcal{O},T'',\mathcal{O}'') $, one positive and one negative.

We now mimic the construction we produced for $ \WB{n} $ to describe the coproduct. We thus consider the following maps:
\begin{itemize}
\item the positive and negative $ k $-pruning maps $ P^{+}_k, P^{-}_k \colon \bigoplus_{n \geq 0} {FN'}^*_{\WD{n}} \otimes \Ftwo \to \bigoplus_{n \geq 0} {FN'}^*_{\WD{n}} \otimes \bigoplus_{m \geq 0} {FN'}^*_{\WD{m}} \otimes \Ftwo $ given by the formula $ P^{\pm}_k (T) = \sum_{(T',\mathcal{O}',T'',\mathcal{O}'')} (T',\mathcal{O}') \otimes (S_k(T''),\mathcal{O}'') $, where the sum runs over all positive and negative oriented $ k $-prunings of $ T $, respectively
\item $ \hat{C} \colon {FN'}^*_{\WD{n}} \otimes {FN'}^*_{\WD{m}} \otimes \Ftwo \to {FN'}^*_{\WD{n+m}} \otimes \Ftwo $, the oriented concatenation map,  given by the formulas $ \hat{C} \left((\bm{a},+) \otimes (\bm{b},+) \right) = \left(C(\bm{a} \otimes \bm{b}),+\right) $, $ \hat{C} \left((\bm{a},+) \otimes (\bm{b},-) \right) = \left(C(\bm{a} \otimes \bm{b}),-\right) $, $ \hat{C} \left((\bm{a},-) \otimes (\bm{b},+) \right) = \left(C(\bm{a} \otimes \bm{b}),-\right) $, $ \hat{C} \left((\bm{a},-) \otimes (\bm{b},-) \right) = \left(C(\bm{a} \otimes \bm{b}),+\right) $.
\end{itemize}
We can also define $ \Delta^{+}_k, \Delta^{-}_k \colon {FN'}^*_{\WD{n}} \otimes \Ftwo \to \bigoplus_{i=0}^n {FN'}^*_{\WD{i}} \otimes {FN'}^*_{\WD{n-i}} \otimes \Ftwo $ by the recursive formulas
\begin{itemize}
\item $ \Delta^{\pm}_0 = P^{\pm}_0 $,
\item $ \Delta^{\pm}_k = (\id \otimes \hat{C}) ( P^{\pm}_k \otimes \id) \Delta^{\pm}_{k-1} $ if $ k \geq 1 $.
\end{itemize}
Let $ \Delta $ be the direct limit $ \varinjlim_k (\Delta_k^{+} + \Delta_k^{-}) $.

\begin{proposition} \label{prop:coprodotto geometrico D}
The oriented pruning coproduct $ \Delta $ is a cochain map and induces the coproduct $ \Delta \colon A_D \to A_D \otimes A_D $ in cohomology.
\end{proposition}
\begin{proof}
It is enough to observe that, looking at the proof of Proposition \ref{prop:coprodotto geometrico}, we can obtain the map $ \Delta \colon {FN'}^*_{\WD{n}} \otimes \Ftwo \to \bigoplus_{i=0}^n {FN'}^*_{\WD{i}} \otimes {FN'}^*_{\WD{n-i}} \otimes \Ftwo $ from $ \tilde{\Delta} \colon \widetilde{FN}^*_{\WB{n}} \otimes \Ftwo \to \bigoplus_{i=0}^n \widetilde{FN}^*_{\WB{i}} \otimes \widetilde{FN}^*_{\WB{n-i}} \otimes \Ftwo $ by restricting to $ \WD{n} $-invariants.
\end{proof}

\begin{proposition} \label{prop:transfer geometrico D}
Let $ \bm{a}^{\pm} $ and $ \bm{b}^{\pm} $ be generic basis elements of $ {FN'}^*_{\WD{n}} $ and $ {FN'}^*_{\WD{m}} $ respectively, where $ \bm{a} $ (respectively $ \bm{b} $) is defined by an $ n $-tuple $ \underline{a} $ (respectively an $ m $-tuple $ \underline{b} $) of non-negative integers.
Let $ \odot \colon {FN'}^*_{\WD{n}} \otimes {FN'}^*_{D_m} \otimes \Ftwo \rightarrow {FN'}^*_{\WD{n+m}} \otimes \Ftwo $ be the homomorphism that maps $ \bm{a}^{\pm} \otimes \bm{b}^{\pm} $ to the sum of all elements $ \bm{c}^{\pm} $, such that the principal $ k $-blocks of $ \bm{c} $ are obtained by shuffling the principal $ k $-blocks of $ \bm{a} $ and $ \bm{b} $ for all $ k \geq 0 $, and the sign of $ \bm{c} $ is deduced from the signs of $ \bm{a} $ and $ \bm{b} $ by applying the multiplication sign rule $ (+,+) \mapsto + $, $ (+,-) \mapsto - $, $(-,+) \mapsto - $, and $ (-,-) \mapsto + $.
This map is a morphism of complexes and induces the transfer product in cohomology.
\end{proposition}
\begin{proof}
The proof is essentially the same as that of Proposition \ref{prop:transfer geometrico B}.
\end{proof}

\section{The almost-Hopf ring presentations}

This section contains the statements of the Hopf ring presentation for $ A_B $ and the almost-Hopf ring presentation for $ A_D $. We thus state our main theorems, whose proof will be postponed until Section \ref{sec:proof} because we still need to develop some necessary algebraic machinery.
In the first subsection, we construct our generators, providing cochain representatives and a geometric interpretation. In the second one, we explain our relations and state Theorems \ref{teo:Bn} and \ref{teo:Dn}. We then apply these results to extract combinatorially accessible additive bases for $ A_B $ and $ A_D $ in Subsection \ref{sec:RelAdd}. Finally, the last subsection is devoted to the link between all these almost-Hopf ring structures.

\subsection{Generators}

We define certain cohomology classes that we will later prove to generate our (almost-)Hopf rings.
We begin with $ A_B $.
\begin{definition} \label{def:cochain generators B}
In $ FN^*_{\WB{n}} $, the following cochains are defined for $ k \geq 0 $, $ m > 0 $, and $ n > 0 $:
\begin{itemize}
\item $ \displaystyle \gamma_{k,m} = [0:\underbrace{\underbrace{1:1:\dots:1}_{2^k -1 \mbox{ times }}: 0: \underbrace{1:1:\dots:1}_{2^{k}-1 \mbox{ times}}: 0: \dots: 0: \underbrace{1:1:\dots:1}_{2^k-1 \mbox{ times}}}_{m \mbox{ times}}] $
\item $ \displaystyle \delta_n = [ \underbrace{1:1: \dots: 1}_{n \mbox{ times }} ] $
\end{itemize}
\end{definition}

A direct calculation shows that both $ \gamma_{k,m} $ and $ \delta_n $ have trivial differential, and thus define cohomology classes $ \gamma_{k,n} \in H^{m(2^k-1)}(\WB{m2^k},\Ftwo) $ and $ \delta_n \in H^n(\WB{n}; \Ftwo) $, that we still denote, with a slight abuse of notation, with the same symbols. While the proof of this fact is entirely straightforward, we provide a proof for the sake of completeness.
\begin{lemma} \label{lem:coboundary generators}
$ \gamma_{k,m} $ and $ \delta_n $ are cocycles in $ FN^*_{\WB{m2^k}} \otimes \Ftwo $ and $ FN^*_{\WB{n}} \otimes \Ftwo $, respectively.
\end{lemma}
\begin{proof}
$ \gamma_{k,m} $ is represented by the symmetric planar level tree in figure \ref{fig:gamma tree}.
We prove that $ d_i(\gamma_{k,m}) = 0 $ for all $ 0 \leq i < m2^k $ by considering different cases:
\begin{itemize}
\item If $ i \not= l2^k $, $ 0 \leq l < m $, the addend $ d_i $ of the differential identifies two edges adjacent in a vertex $ v_j $, $ 1 \leq j \leq m $, and performs a vertex shuffle at the new vertex. Exactly two possible vertex shuffles yield the same tree. Hence $ d_i(\gamma_{k,m}) = 0 $.
\item If $ i = l2^k $ for some $ 1 \leq l < m $, then $ d_i(\gamma_{k,m}) $ is obtained by gluing together $ v_l $ and $ v_{l+1} $ and shuffling the outgoing edges of these two vertices. Since all these shuffles yield the same tree, and there is an even number of them (precisely $ \left( \begin{array}{c} 2^{k+1} \\ 2^k \end{array} \right) $), we have again that $ d_i(\gamma_{k,m}) = 0 $.
\item If $ i = 0 $, $ v_1 $ and its mirror vertex are glued to the central axis of the tree, and the corresponding outgoing edges are permuted with a symmetric shuffle. Again, there is an even number of them (precisely $ 2^{2^k} $), and thus $ d_0(\gamma_{k,m}) = 0 $.
\end{itemize}

$ \delta_n $ is represented by a symmetric planar level tree with $ 2n+1 $ leaves and a single internal vertex of height $ 1 $. The same proof used in the second case of $ \gamma_{k,m} $ shows that $ d_i (\delta_n) = 0 $ for all $ 0 \leq i < n $.
\end{proof}

\begin{figure}[h]
\caption{The planar symmetric level tree representing $ \gamma_{k,m} $.}\label{fig:gamma tree}
\begin{center}
\begin{tikzpicture}[line cap=round,line join=round,>=triangle 45,x=1cm,y=1cm]
\begin{axis}[
x=1cm,y=1cm,
axis lines=middle,
grid style=dashed,
ymajorgrids=true,
xmajorgrids=false,
xmin=-4,
xmax=4,
ymin=-0.9,
ymax=3.5,
xtick=\empty,
yticklabels={,,}
]
\draw [line width=1pt] (0,2)-- (0,0);
\draw [line width=1pt] (1,1)-- (0,0);
\draw [line width=1pt] (0,0)-- (3,1);
\draw [line width=1pt] (0.5,2)-- (1,1);
\draw [line width=1pt] (1,2)-- (1,1);
\draw [line width=1pt] (1,1)-- (1.5,2);
\draw [line width=1pt] (2.5,2)-- (3,1);
\draw [line width=1pt] (3,1)-- (3,2);
\draw [line width=1pt] (3.5,2)-- (3,1);
\draw [line width=1pt] (-0.5,2)-- (-1,1);
\draw [line width=1pt] (-1,2)-- (-1,1);
\draw [line width=1pt] (-1,1)-- (-1.5,2);
\draw [line width=1pt] (-1,1)-- (0,0);
\draw [line width=1pt] (0,0)-- (-3,1);
\draw [line width=1pt] (-2.5,2)-- (-3,1);
\draw [line width=1pt] (-3,1)-- (-3,2);
\draw [line width=1pt] (-3.5,2)-- (-3,1);
\draw (1.65,2.29) node[anchor=north west] {$\dots$};
\draw (0.35,2.95) node[anchor=north west] {$\overbrace{\quad\quad\quad}^{2^k}$};
\draw (2.35,2.95) node[anchor=north west] {$\overbrace{\quad\quad\quad}^{2^k}$};
\draw (-1.65,2.95) node[anchor=north west] {$\overbrace{\quad\quad\quad}^{2^k}$};
\draw (-3.65,2.95) node[anchor=north west] {$\overbrace{\quad\quad\quad}^{2^k}$};
\draw (-2.35,2.29) node[anchor=north west] {$\dots$};
\draw (0.1,3.4) node[anchor=north west] {$\overbrace{\quad\quad\quad\quad\quad\quad\quad\quad\quad}^{m}$};
\draw (-3.9,3.4) node[anchor=north west] {$\overbrace{\quad\quad\quad\quad\quad\quad\quad\quad\quad}^{m}$};
\draw (1,1) node[anchor=north west] {$v_1$};
\draw (3,1) node[anchor=north west] {$v_m$};
\draw (0,0) node[anchor=north west] {$v_0$};
\begin{scriptsize}
\draw [fill=black] (0,0) circle (1.5pt);
\draw [fill=black] (0,1) circle (1.5pt);
\draw [fill=black] (1,1) circle (1.5pt);
\draw [fill=black] (3,1) circle (1.5pt);
\draw [fill=black] (-3,1) circle (1.5pt);
\draw [fill=black] (-1,1) circle (1.5pt);
\end{scriptsize}
\end{axis}
\end{tikzpicture}
\end{center}
\end{figure}

Another possible point of confusion is that the symbol $ \gamma_{k,m} $ is used in \cite{Sinha:13} to indicate a class in $ H^{m(2^k-1)}(\Sigma_{m2^k}; \Ftwo) $. The class we define is the image of this cohomology class of the symmetric group in $ H^{m(2^k-1)}(\WB{m2^k};\Ftwo) $ via the map induced by the projection $ \pi \colon \WB{m2^k} \to \Sigma_{m2^k} $, as we will prove later (Proposition \ref{prop:BS}).

We can interpret all the cohomology classes that we defined above geometrically.
\begin{proposition}\label{prop:geometry generators}
The following statements are true:
\begin{enumerate}
\item Consider the proper submanifold $ \Gamma_{k,m} $ of $ \frac{Y^{(\infty)}_{\WB{m2^k}}}{\WB{m2^k}} $ consisting of $ 2^m $ points that can be partitioned into $ m $ sets of $ 2^k $ points, where all the points in the same subset share the first coordinate. Then $ \gamma_{k,m} $ is the Thom class of $ \Gamma_{k,m} $ in $ Y^{(\infty)}_{\WB{m2^k}} $, in the sense of \cite[Definition 4.6]{Sinha:12}.
\item Consider the vector bundle $ \eta \colon E(\WB{n}) \times_{\WB{n}} \mathbb{R}^n \to B(\WB{n}) $, where $ \WB{n} $ acts on $ \mathbb{R}^n $ via its irreducible reflection representation. Then $ \delta_n $ is the $ n $-dimensional Stiefel-Whitney class of $ \eta $ (the non-oriented version of the Euler class).
\end{enumerate}
\end{proposition}
\begin{proof}
The description of $ \gamma_{k,n} $ is a direct consequence of the conclusions of the geometric arguments of the previous section.

Regarding the second point, consider the vector bundle $ \eta \colon E(\eta) \to B(\eta) $ above, with zero section $ \sigma_0 \colon B(\eta) \to E(\eta) $, and let $ T(\eta) \in H^n(E(\eta),E(\eta) \setminus \sigma_0(B(\eta))) $ be its Thom class. Define
\[
	\tilde{X}_n = \left\{ \left( \underline{x}_1, \dots, \underline{x}_n \right) \in Y^{(\infty)}_{B_n}: \left( x_1 \right)_1 = \dots = \left( x_n \right)_1 = 0 \right\}.
\]
We observe that $ \tilde{X}_n $ is a proper submanifold of $ Y^{(\infty)}_{\WB{n}} $ and that the Thom class of the image $ X_n $ of $ \tilde{X}_n $ in $ \frac{Y^{(\infty)}}{\WB{n}} $ is $ \delta_n $.
We observe that the normal bundle of $ X_n $ in $ \frac{Y^{(\infty)}}{\WB{n}} $ is isomorphic to $ \eta|_{X_n} $.
Since restriction of vector bundles to subspaces preserve Thom classes, we deduce that, if we take $ \frac{Y^{(\infty)}}{\WB{n}} $ as a model for $ B(\WB{n}) $, then $ j^* (l^*)^{-1} k^*(T(\eta)) = \delta_n $, where:
\begin{itemize}
\item $ k \colon \left(E(\eta|_{X_n}),E(\eta|_{X_n}) \setminus \sigma_0(X_n)\right) \to \left( E(\eta), E(\eta) \setminus \sigma_0(B(\eta)) \right) $,
\item $ l \colon \left(E(\eta|_{X_n}),E(\eta|_{X_n}) \setminus \sigma_0(X_n)\right) \to \left( B(\WB{n}), B(\WB{n}) \setminus X_n \right) $ is a tubular neighborhood of $ X_n $ in $ B(\eta) $,
\item and $ j \colon \left(B(\WB{n}),\varnothing\right) \to \left( B(\WB{n}), B(\WB{n}) \setminus X_n \right) $.
\end{itemize}
Note that the induced map $ l^* $ in cohomology is invertible by excision.

Let $ \Phi \colon H^* \left( B \left( \eta \right); \Ftwo \right) \rightarrow H^* \left( E \left( \eta \right), E \left( \eta \right) \setminus \sigma_0 \left( B \left( \eta \right) \right); \Ftwo \right) $ be the Thom isomorphism. We recall that $ \Phi(\alpha) = p^*(\alpha) \cup T(\eta) $, where $ p \colon E(\eta) \to B(\eta) $ is the projection.
We know, for instance from Milnor and Stasheff's book \cite[page 91]{Milnor-Stasheff}, that the Thom class $ T \left( \eta \right) $ and $ w_n \left( \eta \right) $ are linked by the formula:
\[
	w_n \left( \eta \right) = \Phi^{-1} \left( \Sq^n \left( T \left( \eta \right) \right) \right) = \Phi^{-1} \left( T \left( \eta \right)^2 \right)
\]
Hence, to prove that $ w_n \left( \eta \right) = \delta_n $, it is sufficient to show that $ i^* \left( T \left( \eta \right) \right) = p^*j^*(l^*)^{-1}k^* \left( T \left( \eta \right) \right) $, where $ i $ is the obvious inclusion map $ i \colon \left( E \left( \eta \right), \varnothing \right) \rightarrow \left( E \left( \eta \right), E \left( \eta \right) \setminus \sigma_0 \left( B \left( \eta \right) \right) \right) $.

To prove this claim, we first observe that we can use a slightly different model for $ B(\eta) $. We recall that there is a tubular neighborhood $ \tilde{N} $ of $ \tilde{X}_n $ in $ Y^{(\infty)}_{\WB{n}} $ determined by an embedding of the total space of the normal bundle. Explicitly, we can define the embedding by the formula
\[
(\underline{x}_1,\dots,\underline{x}_n) \times (\lambda_1,\dots,\lambda_n) \in \tilde{X}_n \times \mathbb{R}^n \mapsto (\underline{x}_1 + \lambda_1 \bm{e}_1, \dots, \underline{x}_n + \lambda_n \bm{e}_1),
\]
where $ \bm{e}_1 $ is the first element of the canonical basis of $ \mathbb{R}^\infty $. Hence
\begin{align*}
\tilde{N} = \Big\{ (\underline{x}_1,\dots,\underline{x}_n): &\forall 1 \leq i < j \leq n: (\underline{x}_i - (x_{i})_{1} \bm{e}_1) \not= \pm (\underline{x}_j - (x_{j})_{1} \bm{e}_1), \\
&\forall 1 \leq i \leq n: (\underline{x}_i - (x_{i})_{1} \bm{e}_1) \not= 0 \Big\}.
\end{align*}
Note that the action of $ \WB{n} $ preserves $ \tilde{N} $, that $ \tilde{N} $ is provided with a stratification induced from that on $ Y^{(\infty)}_{\WB{n}} $ by restriction. Moreover, every stratum of $ N $ is obtained from a stratum of $ Y^{(\infty)}_{\WB{n}} $ by removing an infinite-codimensional affine subspace. It follows that $ \tilde{N} \to Y^{(\infty)}_{\WB{n}} $ is a homotopy equivalence. $ \tilde{N} $ is still contractible, and therefore we can use its quotient $ N = \frac{\tilde{N}}{\WB{n}} $ as an alternative model for $ B(\WB{n}) $. In this model, the inclusion $ l $ is an isomorphism. Thus we do not need to worry about excision maps, and this simplifies the argument.
The claim now follows by observing that $ i $ and $ kjp $ are homotopic. An explicit homotopy is the following:
\begin{align*}
&\left((\underline{x}_1,\dots,\underline{x}_n) , \underline{\lambda} , t \right) \in \tilde{N} \times_{\WB{n}} \mathbb{R}^n \times [0,1] \mapsto \\
&\left( (\underline{x}_1-(1-t)(x_{1})_{1} \bm{e}_1,\dots,\underline{x}_n-(1-t)(x_{n})_{1} \bm{e}_1), ((1-t)\underline{\lambda} + t ((x_{1})_{1},\dots,(x_{n})_{1})) \right)
\end{align*}
\end{proof}

We now turn our attention to $ \WD{n} $.
First, we give the following definition.
\begin{definition}
Let $ n \geq 1 $, $ m \geq 0 $. We define $ \delta_{n;m}^0 \in H^{n} \left( \WD{n+m}; \Ftwo \right) $ as the restriction of $ \delta_n \odot 1_m \in H^* \left( \WB{n+m}; \Ftwo \right) $ to the cohomology of $ \WD{n+m} $. We also let $ \delta_{0:m}^0 $ be the unique non-zero class in $ H^0 \left( \WD{m}; \Ftwo \right) $ for all $ m \geq 0 $.
\end{definition}

We will require some other generators that do not arise as restrictions of cohomology classes of $ \WB{n} $.

\begin{definition}
Given $ k,m \geq 1 $, we define two cochains in $ FN^{m2^k}_{\WD{n}} $:
\begin{itemize}
\item $ \gamma_{k,m}^+ = [ 0: \underbrace{\underbrace{1: \dots: 1}_{2^k-1 \mbox{ times}}: 0: \underbrace{1: \dots: 1}_{2^k-1 \mbox{ times}}: 0: \dots: \underbrace{1: \dots:1}_{2^k-1 \mbox{ times}}}_{m \mbox{ times}} ] $
\item $ \gamma_{k,m}^- = [ 1: 0 : \underbrace{\underbrace{1: \dots:1 }_{2^k-2 \mbox{ times}}: 0: \underbrace{1: \dots:1}_{2^k-1 \mbox{ times}}: 0: \dots: \underbrace{1: \dots: 1}_{2^k-1 \mbox{ times}}}_{m \mbox{ times}} ] $
\end{itemize}
\end{definition}

\begin{lemma}
$ \gamma_{k,m}^+ $ and $ \gamma_{k,m}^- $ are cocycles.
\end{lemma}
\begin{proof}
The cochain equivalence $ \varphi^* $ of Lemma \ref{lem:CHE} maps $ \gamma_{k,m}^{\pm} $ to
\[
[0:\underbrace{\underbrace{1:1:\dots:1}_{2^k -1 \mbox{ times }}: 0: \underbrace{1:1:\dots:1}_{2^{k}-1 \mbox{ times}}: 0: \dots: 0: \underbrace{1:1:\dots:1}_{2^k-1 \mbox{ times}}}_{m \mbox{ times}}]^{\pm}.
\]
The same proof used for Lemma \ref{lem:coboundary generators}, with the additional requirement of keeping track with orientations, shows that these cochains in $ {FN'}^{m2^k}_{\WD{n}} $ are cocycles. As $ \varphi^* $ is injective, $ \gamma_{k,m}^{\pm} $ must also be a cocycle.

An alternative proof can be obtained by directly using the De Concini and Salvetti description of the boundary in $ C_*^{\WD{n}} $ and dualizing.
\end{proof}
A consequence of the previous lemma is that $ \gamma_{k,m}^+ $ and $ \gamma_{k,m}^- $ represent cohomology classes, that, once again, we denote with the same symbols with a slight abuse of notation.

To adapt our notation to Giusti and Sinha's one for the alternating groups, we will refer to $ \gamma_{k,m}^+ $ (respectively $ \gamma_{k,m}^- $) for some $ k,m $ as positively (respectively negatively) charged generators, and to $ \delta_{n:m}^0 $ for some $ n,m $ as neutral generators.

\subsection{Relations} \label{sec:RelAdd}

This subsection is devoted to deriving algebraic relations between the generators defined above. We will mainly obtain the relations as a consequence of the results in Section \ref{sec:geometric description}.

We first focus on $ A_B $.
We can retrieve in $ A_B $ the same relations among the classes $ \gamma_{k,m} $ that appear in Giusti, Salvatore, and Sinha's Theorem \ref{teo:gruppi simmetrici}. 
\begin{proposition}\label{cor:gamma}
The following formulas hold in $ A_B $:
\[
	\Delta \left( \gamma_{k,m} \right) = \sum_{i+j=m} \gamma_{k,i} \otimes \gamma_{k,j}
\]
\[
	\gamma_{k,n} \odot \gamma_{k,m} = \left( \begin{array}{c} n+m \\ n \end{array} \right) \gamma_{k,n+m}
\]
\end{proposition}
\begin{proof}
We use the chain level formulas computed in Propositions \ref{prop:coprodotto geometrico} and \ref{prop:transfer geometrico B}.

To compute the coproduct, we represent $ \gamma_{k,m} $ by the symmetric planar level tree depicted in Figure \ref{fig:gamma tree}. Note that $ P_l(\gamma_{k,m}) $ is trivial for $ l \geq 2 $ and that the $ 0 $-pruning map gives $ P_0(\gamma_{k,m}) = \sum_{i+j=m} \gamma_{k,i} \otimes \gamma_{k,j} $. Therefore it is enough to prove that $ P_1(\gamma_{k,n}) = \gamma_{k,n} \otimes I $, where $ I $ is the trivial symmetric level tree, for all $ k \geq 0 $ and $ n > 0 $. Consider a $ 1 $-pruning $ (T',T'') $ of $ \gamma_{k,m} $. Every vertex $ v_i $ ($ 1 \leq i \leq m $) as depicted in figure \ref{fig:gamma tree} corresponds to a vertex $ u_i $ of height $ 1 $ in $ T' $. Let $ 2^k-n_i $ be the number of outgoing edges of $ u_i $ for some integer $ 0 \leq n_i < 2^k $.
We can obtain the pruning $ (T',T'') $ from $ \gamma_{k,m} $ by applying a sequence of elementary $ 1 $-prunings at each vertex $ v_i $ and their mirror vertices $ r(v_i) $ that prunes away $ a_i $ outgoing edges from $ v_i $ and $ n_i-a_i $ outgoing edges from $ r(v_i) $, for some $ 0 \leq a_i \leq n_i $.
Therefore, summing over all the possible shuffles of leftovers, whose number is $ \frac{(\sum_{i=1}^m n_i)!}{\prod_{i=1}^m a_i! \prod_{i=1}^m (n_i - a_i)!} $, we deduce that $ (T',S_1(T'')) $ appears in $ P_1(\gamma_{k,m}) $ with coefficient
\[
\sum_{0 \leq a_1 \leq n_1,\dots, 0 \leq a_m \leq n_m} \frac{(\sum_{i=1}^m n_i)!}{\prod_{i=1}^m a_i! \prod_{i=1}^m (n_i - a_i)!} = \frac{(\sum_{i=1}^m n_i)!}{\prod_{i=1}^m n_i!} \prod_{i=1}^m 2^{n_i}.
\]
This number is even unless $ n_i = 0 $ for all $ 1 \leq i \leq m $, yielding the trivial $ 1 $-pruning.

The transfer product formula follows directly from the application of the cochain-level map of Proposition \ref{prop:transfer geometrico B}, by observing that $ \gamma_{k,m} $ has $ m $ principal $ 0 $-blocks all equal to $ [1,\dots,1] $, where the entry $ 1 $ is repeated $ 2^k-1 $ times, and that it has no principal $ l $-block for $ l \geq 1 $. Thus $ \gamma_{k,n} \odot \gamma_{k,m} $ is given by a single basis element in $ {FN}^*_{\WB{2^k(n+m)}} $ (representing $ \gamma_{k,n+m} $) counted as many times as the number of $ (n,m) $-shuffles, that is the binomial coefficient appearing in the equation.
\end{proof}

We can obtain coproduct formulas for $ \delta_n $ via the same geometric description. The following is again a consequence of the formulas in Lemmas \ref{prop:coprodotto geometrico} and \ref{prop:transfer geometrico B}.
\begin{proposition}\label{cor:delta}
\[
	\Delta \left( \delta_n \right) = \sum_{k+l=n} \delta_k \otimes \delta_l
\]
\[
	\delta_n \odot \delta_m = \left( \begin{array}{c} n+m \\n \end{array} \right) \delta_{n+m}
\]
\end{proposition}
\begin{proof}
Since all the entries of $ \delta_n $ are equal to $ 1 $, the cochain-level coproduct map on $ \delta_n $ reduces to the $ 1 $-pruning map $ P_1 $ and provides the desired formula. We compute the transfer product as in Proposition \ref{cor:gamma}, by observing that $ \delta_n $ has no principal $ l $-blocks for $ l \not= 1 $, and has $ n $ principal $ 1 $-blocks all empty.
\end{proof}

These relations will suffice to describe $ A_B $ completely. Our main result, which we will prove in Section \ref{sec:proof}, is the following.
\begin{theorem} \label{teo:Bn}
The Hopf ring $ A_B $ is generated by classes $ \gamma_{k,n} $ (with $ k \geq 0, n > 0 $) and $ \delta_n $ (with $ n > 0 $) with the relations described in Propositions \ref{cor:gamma} and \ref{cor:delta}, together with the following additional relations:\\
{\centering \emph{the product $ \cdot $ of generators from different components is $ 0 $}} \\
{\centering \emph{$ \gamma_{0,n} $ is the $ \cdot $-product unit of the $ n^{th} $ component}}
\end{theorem}

We now turn our attention to $ A_D $.
A trick borrowed from \cite[page 9]{Sinha:17} can be used to simplify the presentation of this almost-Hopf ring. We recall that there is an involution $ \iota \colon A_D \to A_D $. We can define $ A'_D $ to be the bigraded $ \Ftwo $-vector space defined by $ ( A'_D )_{n,d} = H^d \left( \WD{n}; \Ftwo \right) $ if $ (n,d) \not= (0,0) $ and $ ( A'_D )_{0,0} = \Ftwo \left\{ 1^+, 1^- \right\} $. We can embed $ A_D $ as a vector space in $ A'_D $ by identifying the non-zero class in $ H^0 \left( \WD{0}; \Ftwo \right) $ with $ 1^+ + 1^- $.

\begin{lemma}\label{lem:involution on structural maps}
The following statements are true in $ A_D $:
\begin{enumerate}
\item $ \iota(x \odot y) = \iota(x) \odot y = x \odot \iota(y) $
\item $ (\iota \otimes \id) \Delta(x) = (\id \otimes \iota) \Delta(x) = \Delta \iota(x) $
\item $ \iota(x \cdot y) = \iota(x) \cdot \iota(y) $
\end{enumerate}
\end{lemma}
\begin{proof}
\begin{enumerate}
\item $ \odot $ is commutative, and the following diagram induces a pullback of finite coverings at the level of classifying spaces.
\begin{center}
\begin{tikzcd}
\WD{n} \times \WD{m} \ar{r} \ar{d}{c_{s_0} \times \id} & \WD{n+m} \ar{d}{c_{s_0}} \\
\WD{n} \times \WD{m} \ar{r} & \WD{n+m} \\
\end{tikzcd}
\end{center}
\item This follows from the cocommutativity of $ \Delta $ and the commutativity of the diagram above.
\item It is the cohomological consequence of the diagonal map being equivariant with respect to the conjugation $ c_{s_0} $.
\end{enumerate}
\end{proof}

\begin{proposition} \label{prop:HR extension}
Write the coproduct of every element $ x \in A_D $ in $ A_D $ as $ x \otimes 1_0 + \overline{\Delta}(x) + 1_0 \otimes x $, so $ \overline{\Delta} $ is the reduced coproduct.
By letting $ 1^- \cdot 1^+ = 0 $, $ 1^- \cdot 1^- = 1^- $, $ 1^+ \cdot 1^+ = 1^+ $, $ 1^- \odot 1^- = 1^+ $ and $ \Delta(1^{\pm}) = 1^+ \otimes 1^{\pm} + 1^- \otimes 1^{\mp} $, the almost-Hopf ring structure on $ A_D $ extends to an almost-Hopf ring structure on $ A'_D $ such that $ 1^- \odot x = \iota \left( x \right) $, $ 1^- \cdot x = 0 $ and $ \Delta(x) = 1^+ \otimes x + 1^- \otimes \iota(x) + \overline{\Delta}(x) + x \otimes 1^+ + \iota(x) \otimes 1^- $ for every $ x \in A'_D $ of positive degree.
\end{proposition}
\begin{proof}
sing the formulas in the statement of this proposition, we can extend $ \odot $ and $ \cdot $ uniquely to two commutative products on $ A'_D $ and $ \Delta $ to a unique cocommutative coproduct on $ A'_D $.
The coassociativity of $ \Delta $ follows from point 3 of Lemma \ref{lem:involution on structural maps}. The associativity of $ \cdot $ on $ A'_D $ is obvious. The bialgebra structure of $ A'_D $ with $ \cdot $ and $ \Delta $ follows from the bialgebra structure on $ A_D $ and point 2 of the previous lemma. Moreover, the fact that the transfer product with $ 1^- $ is associative follows from point 1.
Hopf ring distributivity with classes involving a transfer product with $ 1^- $ follows again from point 3 of the result referenced above.
\end{proof}

Instead of determining a presentation for $ A_D $, we calculate a presentation for $ A'_D $ because we can write it more concisely. For example, $ \gamma_{k,m}^- = 1^- \odot \gamma_{k,m}^+ $ in $ A'_D $, thus he formulas for $ \gamma_{k,m}^- $ arise as a direct consequence of the formulas for $ \gamma_{k,m}^+ $ and the almost-Hopf ring structure of $ A'_D $. The two approaches are equivalent.

\begin{proposition} \label{lem:cop tr D}
Let $ k,m \geq 1 $ and $ n \geq 0 $. The following coproduct formulas hold in $ A'_D $, where $ \gamma_{k,m-l}^- = 1^- \odot \gamma_{k,m-l}^+ $:
\begin{align*}
\Delta \left( \gamma_{k,m}^+ \right) &= \sum_{l=0}^m \gamma_{k,l}^+ \otimes \gamma_{k,m-l}^+ + \gamma_{k,l}^- \otimes \gamma_{k,m-l}^- \\
\Delta \left( \delta_{n:m}^0 \right) &= \sum_{i=0}^n \sum_{j=0}^m \delta_{i:j}^0 \otimes \delta_{n-i:m-j}^0 \\
\end{align*}
Moreover, the transfer product in $ A'_D $ satisfies the following formulas for every choice of indexes:
\begin{align*}
\gamma_{k,a}^+ \odot \gamma_{k,b}^+ &= 
\left( \begin{array}{c} a+b \\a \end{array} \right) \gamma_{k,a+b}^+ \\
b \odot b' &= 0 \mbox{ if $ b $ and $ b' $ are cup products of neutral generators (i.e. $ \delta^0_{n:m} $)} \\
\delta_{n:m}^0 \odot 1^- &= \delta_{n:m}^0 \\
\end{align*}
\end{proposition}
\begin{proof}
Note that $ \gamma_{k,m}^- = 1_- \odot \gamma_{k,m}^+ = \iota(\gamma_{k,m}^+) $ as a direct consequence of Lemma \ref{lem:involuzione geometrica} and the definition of the cochain representatives of these classes.
The coproduct formulas for $ \gamma_{k,m}^+ $ 
follow from Lemma \ref{lem:CHE} and Lemma \ref{prop:coprodotto geometrico D}. More precisely, we observe that mapping $ \gamma_{k,m}^+ $ into $ {FN'}^*_{\WD{m2^k}} \otimes \Ftwo $ via $ \varphi^* $ yields a cohomology class represented by the same symmetric planar level tree of Figure \ref{fig:gamma tree}, with positive orientation. The same proof of Proposition \ref{cor:gamma} holds in this case by keeping track of orientations.

The coproduct formula for $ \delta_{n:m}^0 $ is a consequence of Proposition \ref{cor:delta}, the Hopf ring properties of $ A_B $, and the fact that the restriction map $ \rho \colon A_B \rightarrow A_D $ preserves coproducts.

Regarding transfer product, we prove the first identity using Proposition \ref{prop:transfer geometrico D} precisely in the same way as the second part of Proposition \ref{cor:gamma}.

Let $ \rho \colon A_B \rightarrow A_D $ be the restriction map. For every $ x \in H^* \left( \WB{n}; \Ftwo \right) $ and $ y \in H^* \left( \WB{m}; \Ftwo \right) $, we can prove that $ \rho \left( x \right) \odot \rho \left( y \right) = 0 $ in $ H^* \left( \WD{n+m}; \Ftwo \right) $ with the same argument used in \cite[Proposition 3.14]{Sinha:17}.
Essentially, it is sufficient to observe that both the restriction $ H^* \left(\WB{n} \times \WB{m}; \Ftwo\right) \to H^* \left( \WD{n} \times \WD{m} ; \Ftwo \right) $ and the transfer $ H^* \left( \WD{n} \times \WD{m}; \Ftwo \right) \to H^* \left( \WD{n+m}; \Ftwo \right) $ factor through the cohomology of the subgroup $ G = \WD{n+m} \cap \left( \WB{n} \times \WB{m} \right) $, and that the composition
\[
H^* \left( G; \Ftwo \right) \stackrel{res}{\to} H^* \left( \WD{n} \times \WD{m}; \Ftwo \right) \stackrel{tr}{\to} H^* \left( G; \Ftwo \right)
\]
is $ 0 $ for mod $ 2 $ coefficients because $ \WD{n} \times \WD{m} $ has even index in $ G $.
In particular, non-trivial transfer products of blocks obtained by cup-multiplying neutral generators must be $ 0 $.
The last relation also follows from the invariance of $ \delta_{n:m}^0 $ with respect to the involution $ \iota $.
\end{proof}

After these coproduct and transfer product formulas, we will also need some cup product relations. Since the Fox-Neuwirth--type cell complex does not behave well with cup products, we found that it is simpler to obtain these formulas via restriction to elementary abelian subgroups.
This approach is fruitful because of a detection theorem for these subgroups. We postpone the proof of the following proposition to Section \ref{sec:restrizione}, where we will explain this in detail.

\begin{proposition} \label{lem:relations cup}
Let $ \gamma_{k,m}^- = 1^- \odot \gamma_{k,m}^+ $ as an element of $ A'_D $. Then the following formulas hold in $ A_D $:
\begin{enumerate}
\item $\forall n,m,k \geq 1, h \geq 2: \gamma_{k,n}^+ \cdot \gamma_{h,m}^- = 0 $
\item $ \forall m \geq 1: \gamma_{1,m}^+ \gamma_{1,m}^- = ( \gamma_{1,m-1}^+ )^2 \odot \delta_{2:0}^0 $
\item the $ \cdot $ product of generators belonging to different components is $ 0 $ and $ \delta_{0:m}^0 $ is the $ \cdot $-product unit of the $ m^{th} $ component
\item $ \forall m \geq 0: \delta_{1:m}^0 = 0 $
\item $ \forall k > 0, m,n \geq 0: \delta_{n:m}^0 \cdot \gamma_{k,\frac{n+m}{2^k}}^+ = \delta_{n:0}^0 \cdot \gamma_{k,\frac{n}{2^k}}^+ \odot \gamma_{k,\frac{m}{2^k}}^+ $, where we understand that $ \gamma_{k,r}^+ = 0 $ if $ r $ is not an integer
\end{enumerate}
\end{proposition}

The last relation we require involves the behavior of the coproduct with the transfer product. We need a preliminary remark.
Let $ b \in A'_D $ be an element obtained as a cup-product of positively and neutrally charged generators (i,e, $ \gamma_{k,m}^+ $ or $ \delta^0_{n:m} $), with at least one positively charged generator. Note that, by Propositions \ref{lem:cop tr D} and \ref{lem:relations cup}, $ \Delta(b) $ can be written as a sum $ \sum_i b_i' \otimes b_i'' $ where $ b_i' $ and $ b_i'' $ are elements obtained as iterated transfer products of elements of the same form, or the images of such elements via the involution $ \iota = 1^- \odot \_ $. We let $ \Delta'(b) $ be that sum restricted only to addends $ b_i' \otimes b_i'' $ in which the involution is not performed to obtain $ b_i' $. Being $ \Delta $ $ (\iota \otimes \iota) $-invariant, this intuitively amounts to keeping half of the addends of the coproduct in $ A'_D $.
\begin{proposition}[cf. \cite{Sinha:17}, Theorem 3.21] \label{lem:transfer mezzo}
Let $ \tau \colon \alpha \otimes \beta \in A'_D \otimes A'_D \mapsto \beta \otimes \alpha \in A'_D \otimes A'_D $ be the map that exchanges the two factors.
For all $ b \in A'_D $ be the cup-product of positively and neutrally charged generators, with at least a positively charged generator appearing, and for all $ x \in A'_D $, we have that
\[
    \Delta \left( b \odot x \right) = \left( \odot \otimes \odot \right) \circ \left( \id \otimes \tau \otimes \id \right) \left( \Delta' \left( b \right) \otimes \Delta \left( x \right) \right),
\]
where $ \Delta' $ is the expression described above.
\end{proposition}
The proof of the analog of this proposition is done in \cite{Sinha:17} by a careful examination of certain spectral sequences. It can be done this way also for $ A_D $. Still, we decide to argue here using detection by elementary abelian subgroups that for finite Coxeter groups comes for free and leads to a shorter proof.
Therefore, we postpone the proof of this proposition until the next section.

We can now state our presentation theorem for $ A'_D $.
\begin{theorem} \label{teo:Dn}
$ A'_D $ is generated, as an almost-Hopf ring, by the classes $ \delta_{n:m}^0 $ ($n \geq 2$, $m \geq 0$), $ \gamma_{k,m}^+ $ ($k,m \geq 1$) and $ 1^- $ defined above, under the relations described in Propositions \ref{lem:cop tr D}, \ref{lem:relations cup} and \ref{lem:transfer mezzo} and the relations $ 1^- \odot 1^- = 1^+ $, $ 1^- \cdot 1^- = 1^- $, $ 1^+ \cdot 1^- = 0 $ and $ \Delta(1^-) = 1^+ \otimes 1^- + 1^- \otimes 1^+ $ coming from Proposition \ref{prop:HR extension}.
\end{theorem}

\subsection{Additive bases}

We describe here additive bases for $ A_B $ and $ A_D $. In this subsection, we assume that the statements of Theorems \ref{teo:Bn} and \ref{teo:Dn} are true. They do not rely logically on the existence of such bases in $ A_B $ and $ A_D $. Thus this choice does not invalidate their proof.

We begin with $ A_B $.
\begin{definition}[cf. \cite{Sinha:12}]
A \emph{gathered block} in $ A_B $ is an element of the form
\[
    b = \delta_m^{t_0} \prod_{k=1}^n \gamma_{k,\frac{m}{2^k}}^{t_k}
\]
where $ m $ is a positive integer, $ 2^n $ divides $ m $, and $ n $ is the maximal index such that $ \gamma_{n,\frac{m}{2^{n}}} $ appears in $ b $ with a non-zero exponent. 
The \emph{profile} of $ b $ is the $ \left( n+1 \right) $-tuple $ \left( t_0, \dots, t_n \right) $.
We also allow $ n = 0 $: in this case, $ b = \delta_m^{t_0} $ for some $ t_0 \geq 0 $.

A \emph{Hopf monomial} is a transfer product of gathered blocks $ x = b_1 \odot \dots \odot b_r $. We denote with $ \mathcal{M}_B $ the set of Hopf monomials whose constituent gathered blocks have pairwise different profiles.
\end{definition}

Note that, given a possible profile $ (t_0, \dots, t_n) $, for all $ l \geq 1 $, there is a unique gathered block $ b $ in the $ (l2^n)^{th} $ component having that profile. As a notational convention, we denote it $ b_{l,\mathcal{t}} $.

We can describe elements of $ \mathcal{M}_B $ graphically.
We represent $ \gamma_{k,n} $ as a rectangle of width $ n2^k $ and height $ 1 - 2^{-k} $ and $ \delta_n $ as a rectangle of width $ n $ and height $ 1 $. The width of a box is the number of the component to which the class belongs. Its area is its cohomological dimension.
We understand the cup product of two generators as stacking the corresponding boxes on top of the other. In contrast, their transfer product corresponds graphically to placing them next to each other horizontally. The profile of a gathered block is described by the height of the rectangles of the corresponding column.
Thus, we can represent every gathered block as a column made of boxes with the same width. Hence, an element of $ \mathcal{M}_B $ is a diagram consisting of columns placed next to each other, such that there are not two columns that consist of rectangles of the same height.
Following the notation in Giusti--Salvatore--Sinha \cite{Sinha:12}, we call these objects \emph{$ B $-skyline diagrams} or, more concisely, \emph{skyline diagrams} where it is clear that we are considering the Hopf ring $ A_B $.

As in \cite{Sinha:12}, the coproduct and the two products in $ \mathcal{M}_B $ have a graphical description, derived from our relations:
\begin{itemize}
\item We divide rectangles corresponding to $ \delta_n $ or $ \gamma_{k,n} $ in $ n $ equal parts via vertical dashed lines. The coproduct is then given by dividing along all vertical lines (dashed or not) of full height and then partitioning the new columns into two to make two new skyline diagrams.
\item The transfer product of two skyline diagrams is given by placing them next to each other and merging every two columns with constituent boxes of the same heights, with a coefficient of $ 0 $ if the widths of these columns share a $ 1 $ in their binary expansion.
\item To compute the cup product of two diagrams, we consider all possible ways to split each into columns, along vertical lines (dashed or not) of full height. We then match columns of each in all possible ways up to automorphism and stack the resulting matched columns to build a new diagram.
\end{itemize}
We depict some examples of calculations with skyline diagrams in Figure \ref{fig:skyline}.

\begin{figure}[htb]
\begin{center}
\begin{tikzpicture}[scale=0.8]
\draw (-0.5,0.25) node {$ \Delta $};
\draw (0,0) rectangle (2,0.5);
\draw (0,0.5) rectangle (2,0.75);
\draw[dashed] (0.5,0) -- (0.5,0.5);
\draw[dashed] (1,0) -- (1,0.75);
\draw[dashed] (1.5,0) -- (1.5,0.5);
\draw (2.5,0.25) node {$ = $};
\draw (3,0) rectangle (5,0.5);
\draw (3,0.5) rectangle (5,0.75);
\draw[dashed] (3.5,0) -- (3.5,0.5);
\draw[dashed] (4,0) -- (4,0.75);
\draw[dashed] (4.5,0) -- (4.5,0.5);
\draw (5.5,0.25) node {$ \otimes $};
\draw (6,0.25) node {$ 1 $};
\draw (6.5,0.25) node {$ + $};
\draw (7,0) rectangle (8,0.5);
\draw (7,0.5) rectangle (8,0.75);
\draw[dashed] (7.5,0) -- (7.5,0.5);
\draw (8.5,0.25) node {$ \otimes $};
\draw (9,0) rectangle (10,0.5);
\draw (9,0.5) rectangle (10,0.75);
\draw[dashed] (9.5,0) -- (9.5,0.5);
\draw (10.5,0.25) node {$ + $};
\draw (11,0.25) node {$ 1 $};
\draw (11.5,0.25) node {$ \otimes $};
\draw (12,0) rectangle (14,0.5);
\draw (12,0.5) rectangle (14,0.75);
\draw[dashed] (12.5,0) -- (12.5,0.5);
\draw[dashed] (13,0) -- (13,0.75);
\draw[dashed] (13.5,0) -- (13.5,0.5);
\draw (1,-0.5) node {$ \delta_4 \gamma_{1,2} $};
\draw (4,-0.5) node {$ \delta_4 \gamma_{1,2} $};
\draw (7.5,-0.5) node {$ \delta_2 \gamma_{1,1} $};
\draw (9.5,-0.5) node {$ \delta_2 \gamma_{1,1} $};
\draw (13,-0.5) node {$ \delta_4 \gamma_{1,2} $};

\draw (0,-3) rectangle (1,-2.5);
\draw (0,-2.5) rectangle (1,-2.25);
\draw[dashed] (0.5,-3) -- (0.5,-2.5);
\draw (1,-3) rectangle (2,-2.5);
\draw[dashed] (1.5,-3) -- (1.5,-2.5);
\draw (2.5,-2.75) node {$ \odot $};
\draw (3,-3) rectangle (5,-2.5);
\draw (3,-2.5) rectangle (5,-2.25);
\draw[dashed] (3.5,-3) -- (3.5,-2.5);
\draw[dashed] (4,-3) -- (4,-2.25);
\draw[dashed] (4.5,-3) -- (4.5,-2.5);
\draw (5.5,-2.75) node {$ = $};
\draw (6,-3) rectangle (9,-2.5);
\draw (6,-2.5) rectangle (9,-2.25);
\draw[dashed] (6.5,-3) -- (6.5,-2.5);
\draw[dashed] (7,-3) -- (7,-2.25);
\draw[dashed] (7.5,-3) -- (7.5,-2.5);
\draw[dashed] (8,-3) -- (8,-2.25);
\draw[dashed] (8.5,-3) -- (8.5,-2.5);
\draw (9,-3) rectangle (10,-2.5);
\draw[dashed] (9.5,-3) -- (9.5,-2.5);
\draw (1,-3.5) node {$ \delta_2 \gamma_{1,1} \odot \delta_2 $};
\draw (4,-3.5) node {$ \delta_4 \gamma_{1,2} $};
\draw (8,-3.5) node {$ \delta_6 \gamma_{1,3} $};

\draw (0,-6) rectangle (1,-5.5);
\draw[dashed] (0.5,-6) -- (0.5,-5.5);
\draw (1,-6) -- (2,-6);
\draw (2.5,-5.75) node {$ \cdot $};
\draw (3,-6) rectangle (3.5,-5.5);
\draw (3.5,-6) rectangle (4.5,-5.75);
\draw (4.5,-6) -- (5,-6);
\draw (5.5,-5.75) node {$ = $};
\draw (6,-6) rectangle (6.5,-5.5);
\draw (6,-5.5) rectangle (6.5,-5);
\draw (6.5,-6) rectangle (7.5,-5.75);
\draw (7.5,-6) rectangle (8,-5.5);
\draw (8.5,-5.75) node {$ + $};
\draw (9,-6) rectangle (9.5,-5.5);
\draw (9.5,-6) rectangle (10.5,-5.5);
\draw (9.5,-5.5) rectangle (10.5,-5.25);
\draw[dashed] (10,-6) -- (10,-5.5);
\draw (10.5,-6) -- (11,-6);
\draw (1,-6.5) node {$ \delta_2 \odot 1_2 $};
\draw (4,-6.5) node {$ \delta_1 \odot \gamma_{1,1} \odot 1_1 $};
\draw (7,-6.5) node {$ \delta_1^2 \odot \gamma_{1,1} \odot \delta_1 $};
\draw (10,-6.5) node {$ \delta_1 \odot \delta_2 \gamma_{1,1} \odot 1_1 $};
\end{tikzpicture}
\caption{Computations via skyline diagrams}\label{fig:skyline}
\end{center}
\end{figure}

\begin{proposition}[cf. \cite{Sinha:12}, Proposition 6.4] \label{prop:basis B}
$ \mathcal{M}_B $ is an additive basis for $ A_B $ and $ \Delta $, $ \odot $ and $ \cdot $ of basis elements are computed graphically via the algorithmic procedures described above.
\end{proposition}
\begin{proof}
We first prove the correctness of the graphical interpretation of the structural morphisms.
The coproduct of a gathered block $ b_{m,\underline{t}} $ with profile $ (t_0,\dots,t_n) $ is of the form
\[
\Delta(b_{m,\underline{t}}) = \sum_{i+j=m, i} b_{i,\underline{t}} \otimes b_{j,\underline{t}}.
\]
We prove this formula by induction on the number of cup-product generators constituting $ b_{i,\underline{t}} $: for single generators $ \delta_m $ or $ \gamma_{k,\frac{m}{2^k}} $ the formula appears in the set of relations for $ A_B $, and the induction step is a consequence of the bialgebra structure formed by $ \cdot $ and $ \Delta $. Thus, the graphical procedure for the calculation of the coproduct is correct on single-column skyline diagrams. As a general skyline diagram represents the transfer product of its columns, the general algorithm is justified because $ \Delta $ and $ \odot $ form a bialgebra.

Regarding $ \odot $, the transfer product of two Hopf monomials corresponds to the horizontal juxtaposition of the corresponding skyline diagrams. Thus, we only need to justify the merging of columns. In formulas, this reads as follows. Fix a profile $ \underline{t} = (t_0,\dots, t_n) $, with $ t_k \geq 0 $ for $ 0 \leq k < n $ and $ t_n > 0 $. Then
\[
b_{i,\underline{t}} \odot b_{j,\underline{t}} = \left( \begin{array}{c} i+j \\ i \end{array} \right) b_{i+j, \underline{t}}.
\]
Again, we prove this by induction on $ r = t_0 + \dots + t_n $. For $ r = 1 $, gathered blocks with profile $ \underline{t} $ are single generators, and the formula above is exactly our transfer product relation among them. For $ r > 1 $, the induction step is proved by combining the coproduct formula for $ \gamma_{l,(i+j)2^{n-l}} $, Hopf ring distributivity, and the fact that cup products of elements in different components is $ 0 $ to deduce that $ \gamma_{l,(i+j)2^{n-l}} (b_{i,\underline{t}} \odot b_{j,\underline{t}}) = (b_{i,\underline{t}} \cdot \gamma_{l,i2^{n-l}}) \odot ( b_{j,\underline{t}} \cdot \gamma_{l,j2^{n-l}}) $, or the analogous formula with $ \delta_{i+j} $ in place of $ \gamma_{n,(i+j)2^{n-l}} $ if $ n=0 $.

The $ \cdot $-product algorithm above graphically encodes Hopf ring distributivity.

Finally, we prove that $ \mathcal{M}_B $ is an additive basis for $ A_B $. We consider the bigraded vector space $ V $ over $ \Ftwo $ with skyline diagrams or, equivalently, $ \mathcal{M}_B $ as a basis. Define linear maps $ \odot, \cdot \colon V \otimes V \to V $ and $ \Delta \colon V \to V \otimes V $ by computing their values on basis elements via the algorithm above.
Note that these maps define a Hopf ring structure on $ V $.
There is a map $ V \to A_B $ that realizes every Hopf monomial as the corresponding element of $ A_B $. Since the procedures to compute the structural morphisms on $ \mathcal{M}_B $ are deduced from the Hopf ring structure of $ A_B $ and the relations of Theorem \ref{teo:Bn}, this map is a morphism of Hopf rings.
We also note that $ V $ is generated as a Hopf ring by single rectangles, corresponding to $ \gamma_{k,n} $ and $ \delta_n $, and that the relations of Theorem \ref{teo:Bn} are satisfied in $ V $. Since $ A_B $ is presented by such generators and relations, it follows that the map $ V \to A_B $ is an isomorphism.
\end{proof}

We now construct an additive basis for $ A_D $, assuming Theorem \ref{teo:Dn}.
The first step is to identify the subalgebra of $ A_D $ under the cup product generated by neutral generators. Let $ \tilde{\mathcal{B}}^0 $ be the set of Hopf monomials $ x \in A_B $ of the form $ x = \delta_{k_1}^{a_1} \odot \dots \odot \delta_{k_r}^{a_r} $, ordered with $ a_1 > \dots > a_r $ and $ k_1 \geq 2 $. These correspond to skyline diagrams in which only boxes of height $ 1 $ appear and in which the highest column has width strictly bigger than $ 1 $.

\begin{lemma} \label{lemma:0}
Every element of $ \tilde{\mathcal{B}}^0 \cap H^* \left( \WB{n}; \Ftwo \right) $ lies in the cup product subalgebra generated by $ \delta_n, \delta_{n-1} \odot 1_1, \dots, \delta_1 \odot 1_{n-1} $. Moreover, the images in $ A_D $ of elements of $ \tilde{\mathcal{B}}^0 $ are a vector space basis for the sub-almost-Hopf ring generated by elements of the form $ \delta_{n:m}^0 $ for $ n,m \geq 0 $.
\end{lemma}
\begin{proof}
Let $ \overline{\mathcal{B}}^0 $ be the set of Hopf monomials $ x \in A_B $ of the form $ x = \delta_{k_1}^{a_1} \odot \dots \odot \delta_{k_r}^{a_r} $ ordered with $ a_1 > \dots > a_r $, without the condition $ k_1 \geq 2 $.
We can define an injective function $ \varepsilon_n \colon \overline{\mathcal{B}}^0 \cap H^* \left( \WB{n}; \Ftwo \right) \rightarrow \mathbb{N}^n $ given by
\[
\varepsilon_n ( \delta_{k_1}^{a_1} \odot \dots \odot \delta_{k_r}^{a_r} ) = ( \underbrace{a_1, \dots, a_1}_{k_1 \mbox{\tiny{ times }}}, a_2, \dots, a_{r-1}, \underbrace{a_r, \dots, a_r}_{k_r \mbox{\tiny{ times }}} )
\]
By identifying $ \mathcal{\overline{B}}^0 \cap H^* \left( \WB{n}; \Ftwo \right) $ with a subset of $ \mathbb{N}^n $ this way, the lexicographic ordering on $ \mathbb{N}^n $ induces a total order on $ \overline{\mathcal{B}}^0 $.
We observe that $ \prod_{i=1}^n \left( \delta_{i} \odot 1_{n-{i}} \right)^{a_i} $ is a linear combination of elements of $ \overline{\mathcal{B}}^0 $. In this linear combination, the biggest non-zero Hopf monomial in the previously defined total order corresponds to $ \left( \sum_{i=1}^n a_i, \sum_{i=2}^n a_i \dots, a_{n-1} + a_n, a_n \right) $. Moreover, this belongs to $ \tilde{\mathcal{B}}^0 $ if and only if $ a_1 = 0 $, i.e. if and only if $ \delta_1 \odot 1_{n-1} $ does not appear as a factor.
Since these are all different, $ \delta_n, \delta_{n-1} \odot 1_1, \dots, \delta_1 \odot 1_{n-1} $ generate, under the cup product, a polynomial subalgebra with basis $ \overline{\mathcal{B}}^0 \cap H^*(\WB{n}; \Ftwo) $. By Proposition \ref{lem:relations cup}, the kernel of the restriction map to $ H^*(\WD{n}; \Ftwo) $ on this subalgebra is the ideal generated by $ \delta_1 \odot 1_{n-1} $. Consequently, the images of elements of $ \tilde{\mathcal{B}}^0 $ in $ A'_D $ are a basis for the cup product subalgebra generated by the elements $ \delta_{n:m}^0 $. Since the transfer products of these elements are trivial and this subalgebra is closed under coproduct by Proposition \ref{lem:cop tr D}, this is a sub-almost-Hopf ring.
\end{proof}

\begin{definition} \label{def:gathered block D}
We call a \emph{neutral gathered block} in $ A_D $ an element $ b \in A'_D $ obtained as the image in $ A'_D $ of an element of the set $ \tilde{B}^0 $ considered in the previous lemma. A \emph{positively charged gathered block}, or simply \emph{positive gathered block}, is an element of the form form $ b = (\delta_{2^nm:0}^0)^{t_0} \prod_{i=1}^n (\gamma_{k,m2^{n-k}}^+)^{t_k} $, for some $ n,m \geq 1 $, $ t_k \geq 0 $ for $ 0 \leq k < n $ and $ t_n > 0 $.
The \emph{profile} of $ b $ is $ (t_0,\dots,t_n) $.
A \emph{negatively charged gathered block}, or simply \emph{negative gathered block}, is an element of the form form $ b = (\delta_{2^nm:0}^0)^{t_0} \prod_{i=1}^n (\gamma_{k,m2^{n-k}}^+)^{t_k} $, for some $ n,m \geq 1 $, $ t_k \geq 0 $ for $ 0 \leq k < n $ and $ t_n > 0 $. The \emph{profile} of $ b $ is $ (t_0,\dots,t_n) $.
A \emph{Hopf monomial} is a transfer product of gathered blocks.
\end{definition}

Note that, given a possible profile $ \underline{t} = (t_0, \dots, t_n) $, for all $ l \geq 1 $, there is a unique positively (respectively negatively) charged gathered block in the $ (l2^n)^{th} $ component having that profile. As a notational convention, we denote it $ b^+_{l,\underline{t}} $ (respectively $ b^-_{l,\underline{t}} $).
Moreover, we stress that we require that a positively charged generator and a negatively charged one do not appear in the same gathered block. This is not a restriction since, due to Proposition \ref{lem:relations cup}, a cup product of two such generators is $ 0 $, or we can write it as a transfer product of gathered blocks. Therefore Hopf monomials generate $ A'_D $ as an $ \Ftwo $-vector space.

We also define a filtration of $ A'_D $ that we will use to extract an additive basis from this set of (linear) generators.
\begin{definition}\label{def:filtration D}
Define the \emph{weight} of a neutral gathered block $ b $ as $ \w(b) = 0 $. Define the weight of a positively or negatively charged gathered block $ b^{\pm}_{l,\underline{t}} $, with profile $ \underline{t} = (t_0,\dots, t_n) $, $ n \geq 1 $, as $ \w(b_{l,\underline{t}}^{\pm}) = l2^{n-1}t_1 $. Define the weight of a Hopf monomial $ x = b_1 \odot \dots \odot b_r $ as the sum $ \w(x) = \w(b_1) + \dots + \w(b_r) $ of the weights of its constituent gathered blocks.
Define the \emph{weight filtration} as the increasing filtration $ F(A'_D) = \{ F_n(A'_D) \}_{n=0}^\infty $ of $ A'_D $ such that $ F_n(A'_D) $ is the linear span of Hopf monomials in $ A'_D $ of weight at most $ n $.
\end{definition}

We first compute formulas for the coproduct and transfer product of gathered blocks in $ A'_D $. These are essentially the charged versions of the corresponding identities in $ A_B $, except for gathered blocks involving the generators $ \gamma_{1,n}^{\pm} $, for which this is true only in the graded space $ \gr_F(A'_D) $ associated with the weight filtration. Complete formulas in $ A'_D $ are complicated and can be retrieved recursively on the filtration $ F $.

\begin{lemma}\label{lem:transfer coproduct basis D}
Let $ n \geq 1 $. Let $ \underline{t} = (t_0,\dots,t_n) $ with $ t_k \geq 0 $ for all $ 0 \leq k < n $ and $ t_n > 0 $. In any almost-Hopf ring satisfying the relations of Theorem \ref{teo:Dn} the following statements are true for all $ i,j > 0 $:
\begin{enumerate}
\item $ 1^- \odot b_{i,\underline{t}}^+ = b_{i,\underline{t}}^- $ and $ 1^- \odot b_{i,\underline{t}}^- = b_{i,\underline{t}}^+ $
\item if $ n > 1 $, the coproduct satisfies $ \Delta(b^+_{m,\underline{t}}) = \sum_{i+j=n} \left( b^+_{i,\underline{t}} \otimes b^+_{j,\underline{t}} + b^-_{i,\underline{t}} \otimes b^-_{j,\underline{t}} \right) $ and $ \Delta(b^-_{m,\underline{t}}) = \sum_{i+j=n} \left( b^+_{i,\underline{t}} \otimes b^-_{j,\underline{t}} + b^-_{i,\underline{t}} \otimes b^+_{j,\underline{t}} \right) $
\item if $ n > 1 $, $ \odot $ satisfies $ b_{i,\underline{t}}^+ \odot b_{j,\underline{t}}^+ = b_{i,\underline{t}}^- \odot b_{j,\underline{t}}^- = \left( \begin{array}{c} i+j \\ i \end{array} \right) b_{i+j,\underline{t}}^+ $ and $ b_{i,\underline{t}}^+ \odot b_{j,\underline{t}}^- = b_{i,\underline{t}}^- \odot b_{j,\underline{t}}^+ = \left( \begin{array}{c} i+j \\ i \end{array} \right) b_{i+j,\underline{t}}^- $
\item for all neutral gathered block $ b^0 $, $ b_{i,\underline{t}}^+ \odot b^0 = b_{i,\underline{t}}^- \odot b^0 $
\item for all profiles $ \underline{u} $, possibly different from $ \underline{t} $, $ b_{i,\underline{t}}^+ \odot b_{j,\underline{u}}^+ = b_{i,\underline{t}}^- \odot b_{j,\underline{u}}^- $, $ b_{i,\underline{t}}^+ \odot b_{j,\underline{u}}^- = b_{i,\underline{t}}^- \odot b_{j,\underline{u}}^+ $
\item if $ n > 1 $, for all profiles $ \underline{u} $, $ b_{i,\underline{t}}^+ \cdot b_{j,\underline{u}}^- = b_{i,\underline{t}}^- \cdot b_{j,\underline{u}}^+ = 0 $
\item for all neutral gathered block $ b' $, $ b_{i,\underline{t}}^- \cdot b' = 1^- \odot ( b_{i,\underline{t}}^+ \cdot b') $
\end{enumerate}
Moreover, 2, 3 and 6 are true in $ \gr_{F}(A'_D) $ even if $ n = 1 $.
\end{lemma}
\begin{proof}
\begin{enumerate}
\item Recall that, by definition, $ \gamma_{k,m}^- = 1^- \odot \gamma_{k,m}^+ $. Combining the link between transfer product and coproduct provided by Proposition \ref{lem:transfer mezzo} with the coproduct formula for $ 1^- $ and $ \gamma_{k,m}^+ $, we deduce that
\[
\Delta(\gamma^-_{k,m}) = \sum_{i+j=m} \left( \gamma_{k,i}^+ \otimes \gamma_{k,j}^- + \gamma_{k,i}^- \otimes \gamma_{k,j}^+ \right),
\]
with the convention that $ \gamma_{k,0}^{\pm} = 1^{\pm} $.
Then, we can prove that $ 1^- \odot b_{i,\underline{t}}^+ = b_{i,\underline{t}}^- $ by induction on the number of cup-product factors of the involved gathered block. If $ b_{i,\underline{t}}^+ $ is a single generator $ \gamma_{k,m}^+ $, the statement holds by definition. The induction step
\[
(b_{i,\underline{t}}^- \cdot \gamma_{k,\frac{i2^n}{2^k}}^-) = (1^- \odot b_{i,\underline{t}}^+) \cdot \gamma_{k,\frac{i2^n}{2^k}}^- = 1^- \odot (b_{i,\underline{t}}^+ \cdot \gamma_{k,\frac{i2^n}{2^k}}^+)
\]
is deduced from Hopf ring distributivity and the coproduct formula derived above for $ \gamma^-_{k,m} $, using that $ 1^- \cdot 1^- = 1^- $, $ 1^- \cdot 1^+ = 0 $ and that the cup product of elements in different components is zero.
The statement for negatively charged gathered blocks is obtained from its analog for positively charged ones by taking the transfer product of both members of the identity with $ 1^- $ and using the relation $ 1^- \odot 1^- = 1^+ $.
\item We begin with the case of positively charged gathered blocks $ b^+_{m,\underline{t}} $. We proceed, again, by induction on the number of $ \cdot $ generators appearing in the expression of $ b^+_{m,\underline{t}} $. If $ b^+_{m,\underline{t}} $ is a single generator, then the statement holds by the coproduct identities of Proposition \ref{lem:cop tr D}. The induction step
follows from the fact that $ \cdot $ and $ \Delta $ form a bialgebra, and relations 1,2,3 of Proposition \ref{lem:relations cup}. For instance, for $ k \geq 2 $, we explicitly have
\begin{align*}
&\Delta(b^+_{m,\underline{t}} \cdot \gamma_{k,\frac{m2^n}{2^k}}^+) = (\cdot \otimes \cdot)( \id \otimes \tau \otimes \id) \Bigg[ \sum_{i+j=m} \left( b^{+}_{i,\underline{t}} \otimes b^{+}_{j,\underline{t}} + b^-_{i,\underline{t}} \otimes b^-_{j,\underline{t}} \right) \\
&\quad \otimes \sum_{r+s=\frac{m2^n}{2^k}} \left( \gamma_{k,r}^+ \otimes \gamma_{k,s}^+ + \gamma_{k,r}^- \otimes \gamma_{k,s}^- \right) \Bigg] \\
&= \sum_{i+j=m} \Bigg( b^+_{i,\underline{t}} \cdot \gamma_{k,\frac{i2^n}{2^k}}^+ \otimes b^+_{j,\underline{t}} \cdot \gamma_{k,\frac{j2^n}{2^k}}^+ +
b^-_{i,\underline{t}} \cdot \gamma_{k,\frac{i2^n}{2^k}}^- \otimes b^-_{j,\underline{t}} \cdot \gamma_{k,\frac{j2^n}{2^k}}^- \\
&\quad + b^+_{i,\underline{t}} \cdot \gamma_{k,\frac{i2^n}{2^k}}^- \otimes b^+_{j,\underline{t}} \cdot \gamma_{k,\frac{j2^n}{2^k}}^- +
b^-_{i,\underline{t}} \cdot \gamma_{k,\frac{i2^n}{2^k}}^+ \otimes b^-_{j,\underline{t}} \cdot \gamma_{k,\frac{j2^n}{2^k}}^+ \Bigg) \\
&= \sum_{i+j=m} \left( b^+_{i,\underline{t}} \cdot \gamma_{k,\frac{i2^n}{2^k}}^+ \otimes b^+_{j,\underline{t}} \cdot \gamma_{k,\frac{j2^n}{2^k}}^+ +
b^-_{i,\underline{t}} \cdot \gamma_{k,\frac{i2^n}{2^k}}^- \otimes b^-_{j,\underline{t}} \cdot \gamma_{k,\frac{j2^n}{2^k}}^- \right). \\
\end{align*}

We only need to be careful when $ k_1 > 1 $ because $ \gamma_{1,l}^+ \gamma_{1,l}^- $ is not necessarily $ 0 $. Note that for $ k \geq 2 $ we have by Hopf ring distributivity $ \gamma_{k,r}^{\pm} \gamma_{1,2^{k-1}r}^+ \gamma_{1,2^{k-1}r}^- = \gamma_{k,r}^{\pm} \left( (\gamma_{1,2^{k-1}r-1}^+ )^2 \odot \delta_{2:0}^0 \right) = 0 $, because the coproduct of $ \gamma_{k,r}^{\pm} $  does not have an addend $ x' \otimes x'' $ with the component of $ x'' $ equal to $ 2 $. This observation guarantees that, if $ n > 1 $, the mixed-charge terms vanish.
Even if $ n = 1 $, we obtain the additional terms by applying relation 2 of Proposition \ref{lem:relations cup} to expressions of this form, and this procedure lowers weights. Thus, the desired formula holds in $ \gr_{F}(A'_D) $ in this case.

The formulas for negatively charged gathered blocks are, once again, obtained by applying the transfer product with $ 1^- $.
\item The formula is easily deduced from the coproduct formulas 2 by induction on the number of $ \cdot $-product generators appearing in $ b_{i,\underline{t}}^{\pm} $. In the case $ n = 1 $, we use the obvious fact that $ \odot $ preserves the weight filtration to deduce that the desired formula holds in the graded space.
\item This is a combination of 1 and the relations $ 1^- \odot \delta_{n:m}^0 = \delta_{n:m}^0 $.
\item This is a combination of 1 and the relations $ 1^{\pm} \odot 1^{\pm} = 1^+ $, $ 1^{\pm} \odot 1^{\mp} = 1^- $.
\item If $ n > 1 $, it follows directly from relation 1 of Proposition \ref{lem:relations cup}. If $ n = 1 $, assume that  $ b_{i,\underline{t}}^+ \in F_a $ and $ b_{j,\underline{u}}^- \in F_b $. Relation $ 2 $ of Proposition \ref{lem:relations cup} provides a way to write $ b_{i,\underline{t}}^+ \cdot b_{j,\underline{u}}^- $ as a product of the form $ \left((\gamma_{1,l-1}^+)^2 \odot \delta_{2:0}^0 \right) \cdot b_{i,\underline{t}'}^+ \cdot b_{j,\underline{u}'}^- $ for some $ l \geq 1 $, where $ b_{i,\underline{t}'}^+ \in F_{a-l} $ and $ b_{i,\underline{u}'} \in F_{b-l} $. By relation $ 5 $ of the same Proposition, these $ \cdot $ products preserve the weight filtration. Therefore the statement is true in $ \gr_F(A'_D) $.
\item We argue as we did for point 1, combining the formula given by point 1 with the relation $ 1^- \odot \delta_{n:m} = \delta_{n:m} $, which implies that neutral gathered blocks are invariant by the action of $ 1^- \odot \_ $.
\end{enumerate}
\end{proof}

Using this lemma, we can use Hopf monomials in the additive basis for $ A_B $ to construct basis elements of $ A_D $ by adding charges.
\begin{proposition} \label{prop:basis D}
Let $ \mathcal{M}_B $ be the Hopf monomial basis for $ A_B $ of Proposition \ref{prop:basis B}. Let $ \tilde{\mathcal{B}}^0 \subseteq \mathcal{M}_B $ be as in Lemma \ref{lemma:0}. Let $ \tilde{\mathcal{B}}^{c} $ be the subset of $ \mathcal{M}_B $ consisting of non-trivial Hopf monomials in which every constituent gathered block has profile $ \underline{t} = (t_0,\dots, t_n) $ with $ n \geq 1 $.
Let $ \mathcal{M}_D = \mathcal{B}^0 \sqcup \mathcal{B}^+ \sqcup \mathcal{B}^- \subseteq A'_D $, where:
\begin{itemize}
\item $ \mathcal{B}^+ = \{ x^+ = \bigodot_{i=1}^k b_{l_i, \underline{t}_i}^+ \}_{x = \bigodot_{i=1}^k b_{l_i,\underline{t}_i} \in \tilde{\mathcal{B}}^{c}} \cup \{1^+ \} $
\item $ \mathcal{B}^- = \{ x^- = \bigodot_{i=1}^{k-1} b_{l_i, \underline{t}_i}^+ \odot b_{l_k,\underline{t}_k}^- \}_{x = \bigodot_{i=1}^k b_{l_i,\underline{t}_i} \in \tilde{\mathcal{B}}^c} \cup \{ 1^- \} $
\item $ \mathcal{B}^0 = \{ x^0 = \rho(y) \odot z^+ \}_{x = y \odot z, y \in \tilde{\mathcal{B}}^0 \setminus \{ 1_0 \}, z \in \tilde{\mathcal{B}}^{c} \cup \{1_0 \}} $
\end{itemize}
$ \mathcal{M}_D $ is an additive basis for $ A'_D $.
\end{proposition}

Before providing a proof of this statement, we make a remark that clarifies the cumbersome identity of Proposition \ref{lem:transfer mezzo}.
\begin{remark}\label{rem:transfer coproduct on basis}
Proposition \ref{prop:basis D} provides a direct sum decomposition of $ A'_D $ as an $ \Ftwo $-vector space with three addends $ V^+ $, $ V^- $ and $ V^0 $, defined as the linear span of $ \mathcal{B}^+ $, $\mathcal{B}^- $ and $ \mathcal{B}^0 $ respectively. Note that the involution $ \iota = 1^- \odot \_ $ switches $ V^+ $ and $ V^- $ and fixes all elements of $ V^0 $ by Lemma \ref{lem:transfer coproduct basis D}. We can consider the linear projection $ \pi \colon V \to V^+ $ defined as the identity on $ V^+ $ and as $ 0 $ on $ V^- $ and $ V^0 $. With this notation, we can rewrite Proposition \ref{lem:transfer mezzo} as follows:
\[
\Delta(b \odot x) = (\odot \otimes \odot) (\pi \otimes \tau \otimes \id) (\Delta \otimes \Delta)(b \otimes x)
\]
for all $ x \in A'_D $ and $ b $ charged gathered block.

A further reduction can be performed. We can consider the free $ \mathbb{F}_4 $-module $ \tilde{V} $ with basis $ \mathcal{M}_D $ and define a $ \mathbb{Z} $-linear map $ \tilde{\Delta} \colon A'_D \otimes A'_D \to \tilde{V} $ as follows. Given $ x,y \in \mathcal{M}_D $, first compute the expansions of the coproducts $ \Delta(x) = \sum_i x_i' \otimes x_i'' $ and $ \Delta(y) = \sum_j y_j' \otimes y_j'' $ on the basis $ \mathcal{M}_D \otimes \mathcal{M}_D $ of $ A'_D \otimes A'_D $. Then, let
\[
\tilde{\Delta}(x \otimes y) = \sum_i \sum_j (x_i' \odot y_j') \otimes (x_i'' \odot y_j'')
\]
where, this time, the sum is computed in $ \tilde{V} $. Recall that both $ \odot $ and $ \Delta $ are $ \iota \otimes \iota $-invariant. Hence, each addend appears twice, except possibly the elements of the form $ (x_i' \odot y_j') \otimes (x_i'' \odot y_j'') $ where $ x_i' $, $ x_i'' $, $ y_j' $, and $ y_j'' $ are all fixed by $ \iota $. But this implies that these classes belong to $ \mathcal{B}^0 $, and thus their transfer product is zero. Consequently, such addends do not appear in the summation.
This implies that $ \tilde{\Delta}(x \otimes y) $ is killed by the multiplication by $ 2 $, and thus $ \tilde{\Delta} $ extends linearly to a map as desired. The image of $ \tilde{\Delta} $ is contained in the image of the embedding $ \xi \colon A'_D \hookrightarrow \tilde{V} $ that maps every $ x \in \mathcal{M}_D $ to $ 2x \in \tilde{V} $.
We can rephrase Proposition \ref{lem:transfer mezzo} by saying that $ (\Delta \odot) \colon x \otimes y \in A'_D \otimes A'_D \mapsto \Delta(x \odot y) \in A'_D $ is the unique linear map satisfying $ \xi \circ ( \Delta \odot) = \tilde{\Delta} $.
We immediately see that this statement is equivalent to the formulation above when $ x $ or $ y $ is a charged gathered block. If both $ x $ and $ y $ belong to $ \mathcal{B}^0 $, $ \tilde{\Delta}(x \otimes y) = 0 $ because the transfer product of two neutral gathered blocks is always zero. The general case follows by induction on the number of $ \odot $-factors in the Hopf monomials involved.
\end{remark}

\begin{proof}[Proof of Proposition \ref{prop:basis D}]
We can write every element in an almost-Hopf ring with generators $ \gamma_{k,n}^+ $, $ 1^- $, $ \delta_{n:m}^0 $ satisfying the relations of Theorem $ \ref{teo:Dn} $ as a linear combination of addends in $ \mathcal{M}_D $ due to Lemmas \ref{lemma:0} and \ref{lem:transfer coproduct basis D}, therefore $ \mathcal{M}_D $ is a set of linear generators for $ A'_D $.

The fact that Hopf monomials in $ \mathcal{M}_D $ are linearly independent is a byproduct of the proof of Theorem \ref{teo:Dn}. It is nevertheless possible to provide a fully independent proof that a basis for the almost-Hopf ring with the presentation of Theorem \ref{teo:Dn} has an additive basis given by $ \mathcal{M}_D $. We will not provide it here because we do not need it in any part of this paper, but the interested reader can find it in the appendix.
\end{proof}

\subsection{Comparison between $ A_{\Sigma} $, $ A_B $, and $ A_D $} \label{subsec:comparison}

In this subsection, we compute the action of the connecting homomorphisms on the elements of the additive bases determined in the previous subsection.

We first start with the link between $ A_{\Sigma} $ and $ A_B $. We recall that there are a natural injection $ j \colon \Sigma_n \to \WB{n} $ and a natural projection $ \pi \colon \WB{n} \to \Sigma_n $, providing linear maps linking $ A_B $ and $ A_{\Sigma} $. We begin by analyzing the relationship between $ A_{\Sigma} $ and $ A_B $.

\begin{proposition} \label{prop:SB}
Let $ j \colon \Sigma_n \rightarrow \WB{n} $ and $ \pi \colon \WB{n} \rightarrow \Sigma_n $ be the natural homomorphisms. The induced maps $ j^* \colon A_B \rightarrow A_\Sigma $ and $ \pi^* \colon A_\Sigma \rightarrow A_B $ are Hopf-ring homomorphisms.
\end{proposition}
\begin{proof}
It is obvious from the fact that the following diagrams are pullbacks of finite coverings, where $ p $ indicates covering maps.
\begin{center}
\begin{tikzcd}
\frac{ \Conf_{n+m} \left( \left( 0, +\infty \right)^\infty \right) }{ \Sigma_n \times \Sigma_m } \arrow{r}{j} \arrow{d}{p} & \frac{ Y^{(\infty)}_{\WB{n+m}} }{ \WB{n} \times \WB{m} } \arrow{d}{p} & \frac{ E \left( \Ftwo \right)^{n+m} \times E \left( \Sigma_{n+m} \right) }{ \left( \Ftwo \wr \Sigma_n \right) \times \left( \Ftwo \wr \Sigma_m \right)}\arrow{r}{\pi} \arrow{d}{p} & \frac{ E \left( \Sigma_{n+m} \right) }{ \Sigma_n \times \Sigma_m } \arrow{d}{p} \\
\frac{ \Conf_{n+m} \left( \left( 0, +\infty \right)^\infty \right) }{ \Sigma_{n+m} } \arrow{r}{j} & \frac{ Y^{(\infty)}_{\WB{n+m}} }{ \WB{n+m} } & \frac{ E \left( \Ftwo \right)^{n+m} \times E \left( \Sigma_{n+m} \right) }{ \Ftwo \wr \Sigma_{n+m} } \arrow{r}{\pi} & \frac{ E \left( \Sigma_{n+m} \right) }{ \Sigma_{n+m} }
\end{tikzcd}
\end{center}
\end{proof}

The following proposition is a direct consequence of Corollary \ref{cor:connecting hom cochain} and Proposition \ref{prop:SB}
\begin{proposition} \label{prop:restriction Bsigma}
With reference to the notation of Theorem \ref{teo:gruppi simmetrici}, $ j^*(\gamma_{k,n}) = \gamma_{k,n} $, and $ j^*(\delta_n) = 0 $.
More generally, given a $ B $-skyline diagram $ x \in \mathcal{M}_B $, $ j^*(x) $ is zero if $ x $ contains a rectangle of height $ 1 $. Otherwise, it is obtained by interpreting $ x $ as a skyline diagram in $ A_{\Sigma} $.
\end{proposition}

We can now use our algebraic description to compute the action of $ \pi^* $ on generators.
\begin{proposition} \label{prop:BS}
$ \pi^*(\gamma_{k,n}) = \gamma_{k,n} $. For a skyline diagram $ x \in A_{\Sigma} $, $ \pi^*(x) $ is obtained by interpreting $ x $ as a $ B $-skyline diagram without rectangles of height $ 1 $.
\end{proposition}
\begin{proof}
We proceed by induction on $ n $.
If $ n = 1 $, since $ \pi \circ j = \id $, $ \pi^* $ is injective. Hence $ \pi^*(\gamma_{k,1}) $ is a non-zero class in $ H^{2^k-1}(\WB{2^k}; \Ftwo) $. Thanks to Proposition \ref{prop:SB}, $ \pi^*(\gamma_{k,1}) $ is primitive. From our description of $ A_B $ in terms of skyline diagrams, formalized with the statement of Proposition \ref{prop:basis B}, we see that the only non-trivial primitive of $ A_B $ in the right component and cohomological degree is $ \gamma_{k,1} $.
For $ n > 1 $, Proposition \ref{prop:SB}  guarantees that $ \pi^* $ preserves coproducts. Hence we inductively have that $ \pi^*(\gamma_{k,n}) + \gamma_{k,n} $ is primitive. However, there are no non-zero primitive in that bidegree, thus $ \pi^*(\gamma_{k,n}) = \gamma_{k,n} $.
\end{proof}

We now turn to $ A_D $. There is a restriction map $ \rho \colon A_B \to A_D $ induced by the inclusions $ \WD{n} \hookrightarrow \WB{n} $. Moreover, we recall that we have natural injections $ i_+,i_- \colon \Sigma_n \to \WD{n} $ determining maps $ A_D \to A_{\Sigma} $ and an involution $ \iota \colon A_D \to A_D $ induced on $ H^*(\WD{n}; \Ftwo) $ by the conjugation with $ s_0 \in \WB{n} $. We analyzed these maps in Subsection \ref{subsec:connecting}.

First, we explain the relation between $ \gamma_{k,m}^+ $ and $ \gamma_{k,m}^- $ and the natural maps between $ \WD{n} $, $ \WB{n} $, and $ \Sigma_n $. 

\begin{proposition} \label{prop:gamma+-}
The following formulas hold for all $ n,k \geq 1 $ and $ m \geq 0 $:
\begin{align*}
i_+^* \left( \gamma_{k,n}^+ \right) &= \gamma_{k,n}, \quad i_+^* \left( \gamma_{k,n}^- \right) = 0 \\
i_-^* \left( \gamma_{k,n}^+ \right) &= 0, \quad i_-^* \left( \gamma_{k,n}^- \right) = \gamma_{k,n} \\
i_+^*(\delta_{n:m}^0) &= i_-^*(\delta_{n:m}^0) = 0
\end{align*}
More generally, with reference to Proposition \ref{prop:basis D}, $ i_+^* $ (respectively $ i_-^* $) is zero on all neutral or negatively charged (respectively neutral or positively charged) Hopf monomials. We obtain the value of positively (respectively negatively) charged Hopf monomials under $ i_+^* $ (respectively $ i_-^* $) by forgetting the charge to get a Hopf monomial in $ \mathcal{M}_B $ and then applying $ j^* $ as described in Proposition \ref{prop:restriction Bsigma}.
\end{proposition}
\begin{proof}
The formulas involving $ \gamma_{k,n}^\pm $ are a direct consequence of Corollary \ref{cor:connecting hom cochain} and the form of the chain representative of $ \gamma_{k,n} \in FN^*_{\Sigma_n} \otimes \Ftwo $ retrieved in \cite[Definition 4.9]{Sinha:13}. To deduce that $ i_+^*(\delta_{n:m}^0) = 0 $, we recall that $ \delta_{n:m}^0 = \rho(\delta_n \odot 1_m) $ and that the composition $ \Sigma_n \stackrel{i_+}{\hookrightarrow} \WD{n} \hookrightarrow \WB{n} $ is equal to $ j $. By Proposition \ref{prop:restriction Bsigma} $ j^*(\delta_n \odot 1_m) = 0 $, therefore $ i_+^*(\delta_{n:m}^0 ) = 0 $. The same is also true for $ i_-^*(\delta_{n:m}^0) $ because $ i_- $ is obtained by composing $ i_+ $ with the conjugation with an element of $ \WB{n} $, whose action is trivial on elements coming from $ A_B $.
\end{proof}

Since we identify the involution $ \iota $ with the transfer product with $ 1^- $, the following proposition is essentially a restatement of the description of the previous subsection.
\begin{proposition}
If $ x^0 $ is a neutral Hopf monomial in $ \mathcal{M}_D $, then $ \iota(x^0) = x^0 $. If $ x^{\pm} $ is a charged Hopf monomial in $ \mathcal{M}_D $, we get $ \iota(x^{\pm}) $ by inverting the charge.
\end{proposition}

To complete the description of the homomorphisms connecting our structures, we need to compute the restriction $ \rho \colon A_B \to A_D $ and transfer $ \tr \colon A_D \to A_B $ maps. To do this, we need to establish preliminary identities.
\begin{lemma} \label{lem:identities transfer restriction}
For all $ x,x' \in A_D $ and for all $ y \in A_B $, the following identities are satisfied:
\begin{enumerate}
\item $ \displaystyle \rho(\tr(x) \odot y) = x \odot \rho(y) $
\item $ \tr(x) \cdot \tr(x') = \tr(x \cdot x' + \iota(x) \cdot x') $
\item $ \tr(x \cdot \rho(y)) = \tr(x) \cdot y $
\item $ \tr(x \odot x') = \tr(x) \odot \tr(x') $
\end{enumerate}
\end{lemma}
\begin{proof}
The first statement follows from the fact that this commutative diagram induces a pullback of covering spaces at the level of classifying spaces:
\begin{center}
\begin{tikzcd}
\WD{n} \times \WD{m} \arrow{r} \arrow{d} & \WD{n} \times \WB{m} \arrow{d} \\
\WD{n+m} \arrow{r} & \WB{n+m}
\end{tikzcd}
\end{center}

As regarding the second statement, since the conjugation by $ s_0 $ is an endomorphism of the covering $ B(\WD{n} \times \WD{n}) \to B(\WD{2n} \cap (\WB{n} \times \WB{n})) $, then $ \tr_{\WD{n} \times \WD{n}}^{\WD{2n} \cap (\WB{n} \times \WB{n})} c_{s_0}^* = c_{s_0}^* \tr_{\WD{n} \times \WD{n}}^{\WD{2n} \cap (\WB{n} \times \WB{n})} $. Moreover, the classifying space functor applied to the following square produces a diagram homotopy equivalent to a pullback of covering, where $ d $ and $ d' $ are diagonal maps:
\begin{center}
\begin{tikzcd}
\WD{n} \arrow{r}{d} \arrow{d}{j} & \WD{n} \times \WD{n} \arrow{d} \\
\WB{n} \arrow{r}{d'} & \WD{2n} \cap ( \WB{n} \times \WB{n} )
\end{tikzcd}
\end{center}
Hence $ \tr d^*= d'^* \tr^{\WD{2n} \cap( \WB{n} \times \WB{n})}_{\WD{n} \times \WD{n}} $.
These facts imply that, denoting with $ d $ the diagonal subgroups,
\begin{align*}
&\tr(x) \cdot \tr(x') = \rho_{d(\WB{n}}^{\WB{n} \times \WB{n}} \tr_{\WD{n} \times \WD{n}}^{\WB{n} \times \WB{n}}(x \otimes x') \\
&= \rho_{d(\WB{n})}^{\WD{2n} \cap ( \WB{n} \times \WB{n})} \rho_{\WD{2n} \cap ( \WB{n} \times \WB{n})}^{\WB{n} \times \WB{n}} \tr_{\WD{2n} \cap (\WB{n} \times \WB{n})}^{\WB{n} \times \WB{n}} \tr_{\WD{n} \times \WD{n}}^{\WD{2n} \cap (\WB{n} \times \WB{n})}(x \otimes x') \\
&= d'^* (\id + c_{s_0}^*) \tr_{\WD{n} \times \WD{n}}^{\WD{2n} \cap (\WB{n} \times \WB{n})} (x \otimes x') \\
&= d'^* \tr_{\WD{n} \times \WD{n}}^{\WD{2n} \cap (\WB{n} \times \WB{n})} (\id + c_{s_0}^*)(x \otimes x') \\
&= \tr d^* (\id + c_{s_0}^*) (x \otimes x') \\
&=  \tr(x \cdot x' + \iota(x) \cdot x')
\end{align*}

Similarly, the last two statements follow from the diagrams below, where the vertical maps of the first one are the diagonal morphisms:
\begin{center}
\begin{tikzcd}
\WD{n} \arrow{r} \arrow{d} & \WB{n} \arrow{d} & \WD{n} \times \WD{m} \arrow{r} \arrow{d} & \WD{n+m} \arrow{d} \\
\WD{n} \times \WB{n} \arrow{r} & \WB{n} \times \WB{n} & \WB{n} \times \WB{m} \arrow{r} & \WB{n+m} \\
\end{tikzcd}
\end{center}
\end{proof}

\begin{proposition} \label{prop:transfer D}
The transfer map $ \tr \colon A_D \to A_B $ is such that $ \tr(\gamma_{k,n}^\pm) = \gamma_{k,n} $ and $ \tr(\delta_{n:m}^0) = 0 $.
More generally, if $ b_{l,\underline{t}}^{\pm} $ is a charged gathered block with profile $ \underline{t} = (t_0,t_1,\dots,t_n) $, then $ \tr(b_{l,\underline{t}}^{\pm}) = b_{l,\underline{t}} $ if $ n \geq 2 $, while if $ n = 1 $ the transfer of gathered blocks is computed inductively by the formula
\[
b_{l,\underline{t}} = \sum_{a=0}^{\left\lfloor \frac{t_1}{2} \right\rfloor} \tr \left(b_{l,(t_0,t_1-a)}^\pm (\gamma_{1,l}^{\mp})^a \right).
\]
The transfer of every neutral gathered block is $ 0 $, and we realize the transfer of a Hopf monomial as the transfer product of the transfer of its constituent gathered blocks.
\end{proposition}
\begin{proof}
The statement for generators is a direct consequence of their definition at the cochain level. The general claim for Hopf monomials in $ A_D $ follows directly from Lemma \ref{lem:identities transfer restriction}.
\end{proof}

\begin{proposition}
$ \rho(\gamma_{k,n}) = \gamma_{k,n}^+ + \gamma_{k,n}^- $ for all $ n,k \geq 1 $. Moreover, $ \rho(\delta_m) = \delta_{m:0}^0 $ for $ n \geq 2 $ and $ \rho(\delta_1) = 0 $.
More generally, for every Hopf monomial $ x \in \mathcal{M}_B $, $ \rho(x) $ can be computed as follows. If $ x = b_{l,\underline{t}} $ is a gathered block with profile $ \underline{t} = (t_0, \dots, t_n) $, we have that
\[
\rho(x) = \left\{ \begin{array}{ll}
x^0 & \mbox{if } n = 0\\
\sum_{a=0}^{t_1} \left( \begin{array}{c} t_1 \\ a \end{array} \right) b_{l,(t_0,a)}^+ (\gamma_{1,l}^-)^{t_1-a} & \mbox{if } n = 1 \\
b_{l,\underline{t}}^+ + b_{l,\underline{t}}^- & \mbox{if } n \geq 2
\end{array} \right. .
\]
The restriction of a Hopf monomial $ x $ with a constituent gathered block in $ \tilde{\mathcal{B}}^0 $ is $ x^0 $. We calculate the restriction of a Hopf monomial $ x \in \tilde{\mathcal{B}}^c $ as follows. First, replace every constituent gathered block in $ x $ with the sum of the positively or neutrally charged elements of its restriction. Then, write the resulting linear combination as a sum of Hopf monomials in $ A_D $. Finally, add to that the negatively charged counterpart of every positively charged Hopf monomials appearing in the sum.
\end{proposition}
\begin{proof}
Using the cochain-level representative of $ \gamma_{k,n} $ introduced in Definition \ref{def:cochain generators B}, we immediately see that its restriction is represented in $ {FN'}^*_{\WD{n2^k}} $ by the sum of two elements obtained by providing this cochain with the two possible orientations. These elements correspond to cochain representatives of $ \gamma_{k,n}^+ $ and $ \gamma_{k,n}^- $ via the cochain equivalence $ \varphi $ of Lemma \ref{lem:CHE}. The formulas for $ \delta_m $ are a consequence of the generators' definition in $ A_D $ and relation 4 of Proposition \ref{lem:relations cup}.
\end{proof}

We conclude this section with a short description of the Gysin sequence of the double cover $ \WD{n} \to \WB{n} $. In \cite{Sinha:17}, Giusti and Sinha adopt the analysis of a similar Gysin exact sequence as the starting point to compute the cohomology of the alternating groups as an almost-Hopf ring. While we retrieve that as a byproduct of our algebraic description, we stress that Giusti and Sinha's approach could be used in our framework as an alternative method to deduce relations in $ A'_D $. Indeed, a direct consequence of the following proposition is that $ \mathcal{M}_D = \mathcal{B}^0 \sqcup \mathcal{B}^+ \sqcup \mathcal{B}^- $ is the polarized basis arising from a Gysin decomposition in the sense of \cite{Sinha:17}.

\begin{proposition}[cf. \cite{Sinha:17}, Section 3]
The restriction $ \rho \colon A_B \to A'_D $ and the transfer $ \tr \colon A'_D \to A_D $ fit into the Gysin sequence
\begin{center}
\begin{tikzcd}
\dots \arrow{r}{\partial_{k-1}} & H^k(\WB{n}; \Ftwo) \ar{r}{\rho_k} & H^k(\WD{n}; \Ftwo) \ar{r}{\tr_k} & { } \\
\ar{r} & H^k(\WB{n}; \Ftwo) \ar{r}{\partial_k} & H^{k+1}(\WB{n}; \Ftwo) \ar{r}{\rho_{k+1}} & \dots, \\ 
\end{tikzcd}
\end{center}
where $ \partial $ is the multiplication with $ \delta_1 \odot 1_{n-1} $. It can be described on skyline diagrams by the operation of replacing each column corresponding to $ \delta_k^m $ with the diagram corresponding to $ \delta_1^{m+1} \odot \delta_{k-1}^m $.
\end{proposition}
\begin{proof}
By a general fact, the connecting homomorphism $ \partial $ is the multiplication with the Euler class $ e $ of the double covering. In the case $ n = 1 $, this covering is isomorphic to the universal double covering $ S^{\infty} \to \mathbb{P}^\infty(\mathbb{R}) $, and its Euler class is $ \delta_1 $. For bigger $ n $, the Euler class is $ \delta_1 \odot 1_{n-1} $ because it is the only class in the right degree that restricts to $ \delta_1 $.

$ \tr \circ \rho = 0 $ because we are working modulo $ 2 $. Therefore the transfer of a neutral gathered block is $ 0 $. If $ b = b_{l,\underline{t}}^\pm $ is a charged gathered block, then the restriction of $ \tr(b) $ must be $ b + \iota(b) $, and the multiplication with $ \delta_1 \odot 1_{n-1} $ must kill $ \tr(b) $. These two conditions force $ \tr(b) = b_{l,\underline{t}} $. Since $ \tr $ preserves the transfer product $ \odot $, the formula for a general Hopf monomial follows.
\end{proof}

\section{Restriction to elementary abelian subgroups} \label{sec:restrizione}

We recall here some theorems from Swenson's thesis \cite{Swenson}, which constitute the formal framework in which we will calculate the cohomology of $ \WB{n} $ and $ \WD{n} $. We will then exploit these theorems to determine the restriction of our generators in $ A_B $ and $ A_D $ to elementary abelian $ 2 $-subgroups. This yields the restriction of all the cohomology of the groups $ \WB{n} $ and $ \WD{n} $ to maximal elementary abelian subgroups, because the structural morphisms of our almost-Hopf rings behave in a predictable way: cup products and coproducts are preserved by such restriction, while the relation with transfer product is determined via double cosets formulas, as stated in Adem and Milgram's book \cite[Section II.6]{Adem-Milgram}.

\subsection{Quillen's theorem for finite reflection groups}
The relevance of these restriction maps is encompassed by a result by Quillen \cite{Quillen:71,Quillen:72}, which we state here.
Let $ G $ be a finite group and $ \mathcal{F} $ a family of subgroups. Let $ \theta_g \colon H^* \left( K; \Fp \right) \rightarrow H^* \left( gKg^{-1}; \Fp \right) $ be the conjugation homomorphism. Define
\[
	\mathcal{F}^n = \left\{ \left\{ f_K \right\}_{K \in \mathcal{F}}, f_K \in H^n \left( K; \Fp \right): \forall K,K': g^{-1}Kg \subseteq K' \Rightarrow f_K = \theta_g^* \left( f_{K'} \right)|_K \right\}.
\]
Alternatively, we can consider $ \mathcal{F} $ as a category in which:
\[
	\Hom \left( K, K' \right) = \left\{ g: g^{-1}Kg \subseteq K' \right\}
\]
Then $ \mathcal{F}^n $ is the inverse limit of the functor $ H^n $ from $ \mathcal{F} $ into the category of $ \Fp $-vector spaces. In other words, $ \mathcal{F}^* $ consists of collections of cohomology classes of groups in $ \mathcal{F} $ that are stable under restrictions and conjugation by elements of $ G $.
Observe that $ \mathcal{F}^* = \bigoplus_n \mathcal{F}^n $ has a natural ring structure.

\begin{theorem}{\rm\cite[Theorem 6.2, page 564]{Quillen:71}}$\mspace{9mu}$
Let $ G $ be a finite group. Let $ \mathcal{F}^* $ be as before. The map $ q_G \colon H^* \left( G; \Fp \right) \rightarrow \mathcal{F}^* $ given by $ q_G \left( f \right) = \left\{ f|_K \right\}_K $ is a well-defined ring homomorphism. Moreover, if $ \mathcal{F} $ is the family of elementary abelian $ p $-subgroups, then the kernel and cokernel of $ q_G $ are nilpotent.
\end{theorem}
Hence elementary abelian $ p $-subgroups give much information on the $ \Fp $-cohomology of a group.
In the case of a finite reflection group, an even stronger property holds.
\begin{theorem} {\rm\cite[Theorem 11, page 2]{Swenson}}$\mspace{9mu}$ \label{teo:Quillen iso}
If $ G $ is a finite reflection group and $ \mathcal{F} $ is the family of elementary abelian $ p $-subgroups of $ G $, then $ q_G $ is an isomorphism.
\end{theorem}

\subsection{Restriction from $ A_B $} \label{subsec:res B}
For the reasons explained in the previous subsection, Swenson has calculated the elementary abelian $ 2 $-subgroups of $ \WB{n} $.
Before stating his result, we need to recall the structure of elementary abelian $ 2 $-groups of the symmetric group $ \Sigma_n $ on $ n $ objects. The relevant calculations are reviewed in \cite{Adem-Milgram}. $ \Sigma_n $ admits a transitive elementary abelian $ 2 $-subgroup if and only if $ n = 2^k $. In this case, all these subgroups are conjugated in $ \Sigma_n $ to the image $ V_k $ of the homomorphism $ \rho_k \colon \Ftwo^k \hookrightarrow \Sigma_{2^k} $ given by the regular action of $ \Ftwo^k $ on itself.
More generally, a maximal elementary abelian $ 2 $-subgroup of $ \Sigma_n $ is conjugated to a direct product
\[
	V_{k_1} \times \dots \times V_{k_r} \hookrightarrow \Sigma_{2^{k_1}} \times \dots \times \Sigma_{2^{k_r}} \hookrightarrow \Sigma_{2^{k_1}+ \dots + 2^{k_r}}
\]
Hence, conjugacy classes of maximal elementary abelian $ 2 $-subgroups in $ \Sigma_n $ are parametrized by partitions $ \pi $ of $ n $ such that every element of $ \pi $ is an integral power of $ 2 $ and the multiplicity of $ 1 = 2^0 $ in $ \pi $ is at most $ 1 $.

To simplify further notation, we borrow from Swenson's thesis the following definition.
\begin{definition}[from \cite{Swenson}]
Let $ n \in \mathbb{N} $. We say that a partition $ \pi $ of $ n $ is \emph{admissible} if it consists only of parts that are integral powers of $ 2 $.
\end{definition}

The main results about elementary abelian $ 2 $-subgroups in $ \WB{n} $ is the following:
\begin{proposition} {\rm\cite[page 22]{Swenson}}$\mspace{9mu}$ \label{prop:subgroups B}
Let $ A_1, A_2 \leq \WB{n} $ be maximal elementary abelian  $ 2 $-subgroups. Then:
\begin{itemize}
\item $ \tilde{A}_i = A_i \cap \Sigma_n \leq \Sigma_n $ is conjugated to a subgroup of the form $ V_{k_1} \times \dots \times V_{k_r} $, with $ k_i \geq 0 \mbox{ } \forall i $.
\item $ A_1 $ and $ A_2 $ are conjugated in $ \WB{n} $ if and only if $ \tilde{A}_1 $ and $ \tilde{A}_2 $ are conjugated in $ \Sigma_n $.
\end{itemize}
In particular, conjugacy classes of maximal elementary abelian $ 2 $-subgroups in $ \WB{n} $ are parametrized by admissible partitions $ \pi $.
Moreover, if we denote by $ A_\pi $ the subgroup corresponding to a partition $ \pi $, we have that $ A_{(2^k)} = V_k \times C_k $, where $ C_k \cong \Ftwo $ is the center of $ \WB{2^k} $ and, more generally, if $ m_i $ is the multiplicity of $ 2^i $ in a partition $ \pi $, then $ A_{\pi} $ is isomorphic to the direct product $ \prod_i{A_{(2^i)}^{m_i}} $.
Let $ d_{2^i - 1}, \dots, d_{2^i-2^{i-1}} $ be the Dickson invariants in $ H^* \left( V_i; \Ftwo \right) \hookrightarrow H^* \left( A_{(2^i)}; \Ftwo \right) $ and define
\[
	f_{2^i} = \prod_{y \in H^1 \left( V_i; \Ftwo \right)} \left( x + y \right)
\]
where $ x \in H^1 \left( A_{2^i}; \Ftwo \right) $ is the linear dual to the non-trivial element in the $ C_i $-factor of $ A_{2^i} $.
There is a natural isomorphism:
\[
	\left[ H^* \left(A_\pi; \Ftwo \right) \right]^{N_{\WB{n}} \left( A_{\pi}\right)} \cong \bigotimes_i \left( \Ftwo \left[ f_{2^i}, d_{2^i-1}, \dots, d_{2^i-2^{i-1}} \right]^{\otimes^{m_i}} \right)^{\Sigma_{m_i}}
\]
\end{proposition}

We can calculate the restriction of our generating classes $ \gamma_{k,n} $ and $ \delta_n $ to these abelian subgroups.
The calculation for $ \gamma_{k,n} $ has been essentially carried out by Giusti, Salvatore, and Sinha \cite{Sinha:12}. We state here the result.
\begin{proposition}{\rm\cite[Corollary 7.6, page 189]{Sinha:12}}$\mspace{9mu}$ \label{prop:restrizione gamma}
Let $ l,n \geq 1 $. Let $ \pi $ be a partition of $ n2^l $ consisting of powers of $ 2 $, $ \pi = \left( 2^{k_1}, \dots, 2^{k_r} \right) $. Then:
\[
	\gamma_{l,n}|_{A_\pi} = \left\{ \begin{array}{ll} \otimes_{i=1}^r d_{2^{k_i}-2^{k_i-l}} & \mbox{if } k_i \geq l \mbox{ }\forall 1 \leq i \leq r \\
	0 & \mbox{otherwise} \end{array} \right.
\]
\end{proposition}

\begin{proposition} \label{prop:restrizione delta}
Let $ n \geq 0 $. Let $ \pi = \left( 2^{k_1}, \dots, 2^{k_r} \right) $ be an admissible partition. The restriction of $ \delta_n $ to the cohomology of the maximal elementary abelian $ 2 $-subgroup $ A_{\pi} $ is equal to $ \otimes_{i=1}^r f_{2^{k_i}} $. Moreover, $ \delta_n $ is the unique class in $ H^n \left( \WB{n}; \Ftwo \right) $ that has this property for every $ \pi $.
\end{proposition}
\begin{proof}
We observe that the restrictions of a cohomology class to $ A_\pi $ for all the partitions $ \pi $ of $ n $ determine its restriction to every elementary abelian $ 2 $-subgroup (not necessarily maximal). Hence, by Theorem \ref{teo:Quillen iso}, a class that satisfies the condition in the statement for every $ \pi $ is necessarily unique.

Let $ U_n = \mathbb{R}^n $ be the reflection representation of $ \WB{n} $. Recall that, if $ n=2^k $ and $ \pi = \left( 2^k \right) $, then $ A_\pi = V_k \times C_k $, where $ C_k = \langle t \rangle $ is a cyclic group of order $ 2 $, the center of $ \WB{n} $, and $ V_k = \langle v_1, \dots, v_k \rangle \leq \Sigma_{2^k} $ is the subgroup defined above. $ H^*(A_{\pi}; \Ftwo) $ is polynomial on degree $ 1 $ elements $ x,y_1,\dots,y_k $, the linear duals to $ t,v_1,\dots,v_k $ respectively. Given $ a \in A_{\pi} \setminus \{1\} $, let $ \varepsilon_a $, $ \sgn_a $, and $ \mathbb{R}\langle a \rangle $ be the $ 1 $-dimensional trivial representation, the signum representation, and the regular representation of $ \langle a \rangle \cong \Ftwo $, respectively.
We first observe that, since $ t $ acts on $ U_n $ as the multiplication by $ -1 $, $ U_n|_{A_\pi} \cong \sgn_t \otimes U_n|_{V_k} $. Moreover, the inclusion of $ V_k $ in $ \Sigma_{2^k} $ is given by the regular representation, hence:
\[
	U_n|_{V_k} \cong \bigotimes_{i=1}^k \mathbb{R}\langle v_i \rangle \cong \bigoplus_{S \subseteq \left\{1,\dots,k \right\}} \bigotimes_{i=1}^k U_{S,i}
\]
where $ U_{S,i} $ is equal to $ \sgn_{v_i} $ if $ i \in S $, to $ \varepsilon_{v_i} $ if $ i \notin S $.
Thus, with the notation used before in this document, the Stiefel-Whitney class of $ U_n|_{A_\pi} $ is:
\[
	\prod_{ S \subseteq \left\{ 1, \dots, n \right\} } \left( 1 + x + \sum_{i \in S} y_i \right)
\]
Its $ n $-dimensional part is exactly $ f_{2^k} $. Hence, the thesis for $ \pi = \left( 2^k \right) $ follows from the naturality of the characteristic classes and Proposition \ref{prop:geometry generators}

In the case of a general admissible partition $ \pi = \left( 2^{k_1}, \dots, 2^{k_r} \right) $, the proposition follows from the fact that $ A_{\pi} \cong \prod_{i=1}^r A_{(2^{k_i})} $ and $ U_n|_{A_{\pi}} \cong \oplus_{i=1}^r U_{2^{k_i}}|_{A_{(2^{k_i} )}} $.
\end{proof}

To complete the calculation of the restriction morphisms from $ A_B $ to maximal elementary abelian $ 2 $-subgroups, we need to describe how such maps behave with the structural morphisms of $ A_B $. Restrictions preserve cup products, and, regarding the coproduct, there is nothing to say because every maximal elementary abelian subgroup of $ \WB{n} \times \WB{m} $ is itself a maximal elementary abelian subgroup of $ \WB{n+m} $. The only non-trivial behavior occurs with the transfer product. We describe it in the following proposition.

\begin{proposition}\label{lem:transfer abelian subgroups}
Let $ x,y \in A_B $ be in positive components $ n $ and $ m $ respectively. Let $ \pi =(2^{k_1},\dots, 2^{k_r}) $ be an admissible partition of $ n+m $. For all $ I \subseteq \{1,\dots,r\} $, write $ I = \{i_1,\dots,i_s\} $ with $ i_1 < \dots < i_s $ and let $ \pi_I = (2^{k_{i_1}},\dots,2^{k_{i_s}}) $. Then
\[
(x \odot y)|_{A_{\pi}} = \sum_{I,J} \tau_{I,J} ( x|_{A_{\pi_I}} \otimes y|_{A_{\pi_J}}),
\]
where the sum runs over all partition $ \{1,\dots,r\} = I \sqcup J $ of $ \{1,\dots,r\} $ into two subsets such that $ \sum_{i\in I}2^{k_i} = n $ (and, consequently, $ \sum_{j\in J} 2^{k_j} = m $), and $ \tau_{I,J} \colon H^*(A_{\pi_I}; \Ftwo) \otimes  H^*(A_{\pi_J}; \Ftwo) \to H^*(A_{\pi}; \Ftwo) $ is the obvious permutation of tensor factors.
\end{proposition}
\begin{proof}
We begin by assuming that $ r = 1 $, thus $ \pi = (2^k) $ for some $ k $ and $ n+m=2^k $. Then, since $ A_{\pi} $ acts transitively on $ \{1,\dots,2^k\} $, no conjugate of $ A_{\pi} $ in $ \WB{n+m} $ is contained in $ \WB{n} \times \WB{m} $. Being $ A_{\pi} $ abelian, the classically known property stated in \cite[Proposition 5.6, page 69]{Adem-Milgram} implies that the transfer map $ H^*(A_{\pi} \cap \sigma (\WB{n} \times \WB{m}) \sigma^{-1}; \Ftwo) \to H^*(A_{\pi}; \Ftwo) $ is identically zero. Eilenberg's double coset formula then guarantees that the composition of the restriction with the transfer product $ H^*(\WB{n}; \Ftwo) \otimes H^*(\WB{m}; \Ftwo) \stackrel{\odot}{\rightarrow} H^*(\WB{n+m}; \Ftwo) \rightarrow H^*(A_{\pi}; \Ftwo) $ is zero. Thus $ (x \odot y)|_{A_{(2^k)}} = 0 $.

In the general case, the restriction of $ x \odot y $ to this subgroup factors through the $ r $-fold coproduct. By the calculations above, addends in this coproduct for which a factor is a non-trivial transfer product restrict to $ 0 $.
Since $ \odot $ and $ \Delta $ form a bialgebra structure on $ A_B $, the other addends have the desired form.
\end{proof}

\subsection{Restriction from $ A_D $ and proof of relations}

We can adapt the argument to calculate the restriction to elementary abelian subgroups of generators also in the $ D_n $ case.
First, we state the analog of Proposition \ref{prop:subgroups B}. Recall that a partition $ \pi $ of $ n $ is admissible if and only if it consists of parts that are powers of $ 2 $.

\begin{theorem}{\rm\cite[Theorem 5.4.3, page 40]{Swenson}}$\mspace{9mu}$ \label{thm:subgroups D}
Let $ \pi $ be an admissible partition of $ n $. Let $ m_1 $ and $ m_2 $ be the multiplicities of $ 1 $ and $ 2 $ in $ \pi $. We write $ \pi = \left(1\right)^{m_1} \cup \left( 2 \right)^{m_2} \cup \pi' $.
Let $ A_{\pi} \leq \WB{n} $ the maximal elementary abelian $ 2 $-subgroup corresponding to $ \pi $ and let $ \widehat{A}_{\pi} = A_{\pi} \cap \WD{n} $.
Then $ \widehat{A}_{\pi} $ is maximal as an elementary abelian subgroup of $ \WD{n} $ if and only if $ m_1 \not= 2 $. Moreover:
\begin{itemize}
\item If $ m_1 > 0 $, then $ \widehat{A}_{\pi} = \ker\left( \sum \colon \Ftwo^{m_1} \rightarrow \Ftwo \right) \times A_{(2)^{m_2} \cup \pi'} $. If $ e_1, \dots, e_{m_1} $ are the elementary symmetric functions in
$ H^* \left( \Ftwo^{m_1}; \Ftwo \right) = H^* \left( A_{(1)^{m_1}}; \Ftwo \right) $, we define $ \bar{e}_i = e_i + e_1 e_{i-1} $ if $ 2 \leq i < m $ and $ \bar{e}_m = e_1 e_{m-1} $.
There is an isomorphism between the invariant subalgebra $ \left[ H^* \left( \widehat{A}_{\pi}; \Ftwo \right) \right]^{N_{\WD{n}} \left( \widehat{A}_{\pi} \right)} $ and 
\[
	 \Ftwo \left[ \bar{e}_2, \dots, \bar{e}_m \right] \otimes \left[ H^* \left( A_{(2)^{m_2} \cup \pi'}; \Ftwo \right) \right]^{N_{\WB{n-m_1}}(A_{(2)^{m_2} \cup \pi'})}.
\]
Moreover, the cohomological restriction from $ A_{(1)^m} $ to $ \widehat{A}_{(1)^m} $ is given by $ e_1 \mapsto 0 $ and $ e_i \mapsto \overline{e}_i $ if $ 2 \leq i \leq m $.
\item If $ m_1 = 0 $ and $ m_2 > 0 $, then $ \widehat{A}_{\pi} = A_{\pi} $. Identifying $ H^* \left( A_{(2)}; \Ftwo \right)^{\otimes^{m_2}} $ with $ \otimes_{i=1}^{m_2} \Ftwo \left[ x_i,y_i \right] $, we can define:
\[
		h_{m_2} = \sum_{\substack{ S \subseteq \left\{ 1, \dots, m_2 \right\}\\ |S| = 2l }} \prod_{i \in S} \left( x_i + y_i \right) \prod_{j \notin S} x_j
\]
Then $ [ H^* ( \widehat{A}_\pi; \Ftwo ) ]^{N_{\WD{n}} ( \widehat{A}_\pi )} $ is the free $ \left[ H^* \left( A_\pi; \Ftwo \right) \right]^{N_{\WB{n}} \left( A_\pi \right)} $-module with basis $ \left\{ 1, h_{m_2} \otimes 1_{H^*(A_{\pi'}; \Ftwo)} \right\} $.
\item If $ m_1 = m_2 = 0 $, then $ \widehat{A}_{\pi} = A_{\pi} $ and $ N_{\WD{n}} \left( A_{\pi} \right) = N_{\WB{n}} \left( A_{\pi} \right) $, hence:
\[
[ H^* ( \widehat{A}_{\pi}; \Ftwo ) ]^{N_{\WD{n}} ( \widehat{A}_{\pi} )} =
\left[ H^* \left( A_{\pi}; \Ftwo \right) \right]^{N_{\WB{n}} \left( A_{\pi} \right)}
\]
\end{itemize}
Moreover, if $ m_1 \not= 0 $ or $ m_2 \not= 0 $, then $ A_\pi $ is $ \WB{n} $-conjugate to $ A' $ if and only if $ \widehat{A}_{\pi} $ is $ \WD{n} $-conjugate to $ A' \cap \WD{n} $.
Conversely, if $ m_1 = m_2 = 0 $, then the $ \WB{n} $-conjugacy class of $ A_\pi $ contains exactly two $ \WD{n} $-conjugacy classes of elementary abelian $ 2 $-subgroups.
\end{theorem}

We now determine the restriction of our generators to the elementary abelian subgroups.

\begin{proposition} \label{prop:restriction eas}
Let $ n= 2^km $, for some $ k,m \geq 1 $. Let $ \pi $ be an admissible partition of $ n $. Let $ m_1 $ and $ m_2 $ be the multiplicities of $ 1 $ and $ 2 $ in $ \pi $. Then:
\begin{enumerate}
\item for every $ k \geq 1 $, if $ m_1 = m_2 = 0 $, then $ \gamma_{k,m}^+|_{A_\pi} = \gamma_{k,m}|_{A_{\pi}} $, $ \gamma_{k,m}^+|_{A_\pi^{s_0}} = 0 $, $ \gamma_{k,m}^-|_{A_\pi} = 0 $, $ \gamma_{k,m}^-|_{A_\pi^{s_0}} = \gamma_{k,m}|_{A_\pi^{s_0}} $
\item for every $ k \geq 2 $, if $ m_1 \not= 0 $ or $ m_2 \not= 0 $, or for $ k = 1 $ if $ m_1 \not= 0 $, then $ \gamma_{k,m}^{\pm}|_{\widehat{A}_{\pi}} = 0 $.
\item if $ m_1 = 0 $ but $ m_2 \not= 0 $, that is $ \pi = \left( 2 \right)^{m_2} \sqcup \pi' $, then the restriction of $ \gamma_{1,m}^+ $ (respectively $ \gamma_{1,m}^- $) to $ \widehat{A}_{\pi} = A_{(2)^{m_2}} \times A_{\pi'} $ is $ h_{m_2} \otimes \gamma_{1,m-m_2}|_{A_{\pi'}} $ (respectively $ ( d_1^{\otimes^{m_2}} + h_{m_2} ) \otimes \gamma_{1,m-m_2}|_{A_{\pi'}} $)
\item if $ \pi = (1)^{m_1} \sqcup \pi' $, then the restriction of $ \delta_{k:m}^0 $ to $ \widehat{A}_{\pi} = \widehat{A}_{(1)^{m_1}} \times A_{\pi'} $ is $ 1 \otimes (\delta_k \odot 1_{\WB{m-m_1}})|_{A_{\pi'}} + \sum_{i=2}^k \overline{e}_i \otimes ( \delta_{k-i} \odot 1_{\WB{m-m_1+i}} )|_{A_{\pi'}} $, with the convention that $ 1_{\WB{r}} = 0 $ when $ r < 0 $.
\end{enumerate}
\end{proposition}
\begin{proof}
If $ \pi $ has more than $ 1 $ element and is different from $ (1,1,\dots,1) $, then the restriction to $ \widehat{A}_{\pi} $ or $ \widehat{A}_{\pi}^{s_0} $ factors through the coproduct. Thus, by applying the coproduct formulas of Proposition \ref{lem:cop tr D}, we can inductively reduce to these two cases.

We begin by assuming that $ \pi = (2^n) $ has only one element, and we prove the first statement.
If $ k \geq 1 $ and $ n \geq 2 $, the restriction of $ \gamma_{k,2^l}^\pm $ to $ A_{\pi} $ ($ n=k+l $) must be $ N_{\WD{2^n}} \left( A_{\pi} \right) $-invariant.
Hence, for degree reasons, it can be $ 0 $ or $ d_{2^n - 2^l} $. Since $ i_+^* ( \gamma_{k,2^l}^+ ) = \gamma_{k,2^l} $ (respectively $ i_+^* ( \gamma_{k,2^l}^- ) = 0 $) by Proposition \ref{prop:gamma+-}, its restriction to $ A_{\pi} \cap \Sigma_{2^n} $ must be the Dickson invariant of degree $ 2^n - 2^l $ (respectively $ 0 $).
This forces $ \gamma_{k,2^l}^+|_{A_{\pi}} = d_{2^n - 2^l} = \gamma_{k,2^l}|_{A_{\pi}} $ (respectively $ \gamma_{k,2^l}^-|_{A_{\pi}} = 0 $).
By essentially the same argument, considering $ i_- $ instead of $ i_+ $, we determine the restrictions to $ A_{\pi}^{s_0} $, proving the first point.

Claim number 2 is immediate from the fact that, if $ k \geq 1 $, there are no non-zero elements in $ H^*(\widehat{A}_{(1)^{2^kn}}; \Ftwo)^{N_{\WD{2^kn}}(\widehat{A}_{(1)^{2^kn}})} $ in the same degree of $ \gamma_{k,n}^\pm $, and that if $ k \geq 2 $ the coproduct of $ \gamma_{k,n}^\pm $ has no element in component 2.

To prove 3 when $ \pi = (2) $, we notice that $ A_{(2)} = \WD{2} $ and $ \gamma_{1,1}^+ $ can be identified with $ h_1 $, while $ \gamma_{1,1,}^- $ with $ d_1 + h_1 $.

By the coproduct formula for $ \delta_{k:m}^0 $, we have that  $ \delta_{k:m}^0|_{A_{(1)}^{k+m}} = e_k $. Thus, the last statement for $ \pi = (1,\dots,1) $ follows directly by combining Proposition \ref{prop:restrizione delta} and Theorem \ref{thm:subgroups D}.
\end{proof}

As in $ A_B $, the behavior of the restriction to maximal elementary abelian $ 2 $-subgroups with the cup product and coproduct is straightforward. We describe the relation between such restriction maps and the transfer product in the following proposition, which is the counterpart of Proposition \ref{lem:transfer abelian subgroups}.
\begin{proposition} \label{prop:transfer subgroups D}
Let $ x,y \in A_D $ be elements in positive component $ n $ and $ m $ respectively. Let $ \pi = (2^{k_1},\dots,2^{k_r}) $ be an admissible partition of $ n+m $. For all $ I \subseteq \{1,\dots,r\} $, write $ I = \{i_1,\dots,i_s\} $ with $ i_1 < \dots < i_s $ and let $ \pi_I = (2^{k_{i_1}},\dots,2^{k_{i_s}}) $. Then
\[
(x \odot y)|_{\widehat{A}_{\pi}} = \sum_{I,J} \tau_{I,J} (x|_{\widehat{A}_{\pi_I}} \otimes y|_{\widehat{A}_{\pi_J}} + \iota(x)|_{\widehat{A}_{\pi_I}} \otimes \iota(y)|_{\widehat{A}_{\pi_J}}),
\]
where the sum runs over all partition $ \{1,\dots,r\} = I \sqcup J $ of $ \{1,\dots,r\} $ into two subsets such that $ \sum_{i \in I} 2^{k_i} = n $ (and, consequently, $ \sum_{j \in J} 2^{k_j} = m $) and at least one between $ I $ and $ J $ does not contains any $ l \in \{1,\dots,r\} $ such that $ k_l = 0 $, and where $ \tau_{I,J} \colon H^*(\widehat{A}_{\pi_I}; \Ftwo) \otimes H^*(\widehat{A}_{\pi_J}; \Ftwo) \to H^*(\widehat{A}_{\pi}; \Ftwo) $ is the obvious permutation of tensor factors.
Moreover, if $ 1 \notin \pi $ and $ 2 \notin \pi $,
\[
(x \odot y)|_{\widehat{A}_{\pi}^{s_0}} = c_{s_0}^* \sum_{I,J} \tau_{I,J} \left( x|_{\widehat{A}_{\pi_I}} \otimes \iota(y)|_{\widehat{A}_{\pi_J}} + \iota(x)|_{\widehat{A}_{\pi_I}} \otimes y|_{\widehat{A}_{\pi_J}} \right),
\]
where $ I,J,\tau_{I,J} $ are as above, and $ c_{s_0}^* \colon H^*(\widehat{A}_{\pi}; \Ftwo) \to H^*(\widehat{A}_{\pi}^{s_0}; \Ftwo) $ is induced by the conjugation with $ s_0 $.
\end{proposition}
\begin{proof}
We cannot repeat the proof of Proposition \ref{lem:transfer abelian subgroups} because in $ A'_D $ the transfer product and the coproduct do not form a bialgebra. Therefore, we argue by considering Eilenberg's double coset formula associated with the two subgroups $ \WD{n} \times \WD{m} $ and $ \widehat{A}_\pi $ of $ \WD{n+m} $.
We preliminarily fix some notation.
Let $ P_{\pi} $ be the partition of the set $ \{1,\dots, n+m\} $ given by
\[
P_{\pi} = \{\{1,\dots,2^{k_1}\},\{2^{k_1}+1,\dots,2^{k_1}+2^{k_2}\},\dots,\{\sum_{l=1}^{r-1}2^{k_l}+1,\dots, n+m\}\}.
\]
Moreover, let $ P_{j,\pi} = \{\sum_{l=1}^{j-1} 2^{k_l}+1,\dots, \sum_{l=1}^j 2^{k_l}\} $.

Assume that $ 1 \notin \pi $. A set of representatives for $ \WD{n+m}/(\WD{n} \times \WD{m}) $ is the set $ \Sh(n,m) \cdot \{1,t\} $, where $ \Sh(n.m) \subseteq \Sigma_{n+m} \hookrightarrow \WD{n+m} $ is the set of $ (n,m) $-shuffles, and $ t = s_0 \times s_0 \in \WB{n} \times \WB{m} $.
Note that $ \widehat{A}_{\pi} \subseteq (\sigma t^{\varepsilon})(\WD{n} \times \WD{m})(\sigma t^{\varepsilon})^{-1} $ if and only if $ \sigma(\{1,\dots,n\}) $ is a union of parts of $ P_{\pi} $. Being $ \widehat{A}_{\pi} $ abelian, these provide the only non-zero terms in the summation of the double coset formula.
Moreover, by inspecting the image of  $ \{1,\dots,n\} \subseteq \{1,\dots, n+m \} $ under the signed permutation action of $ \WD{n+m} \subseteq \WB{n+m} $, we see that if $ \sigma t^{\varepsilon} $, $ \sigma't^{\varepsilon'} $ are two coset representatives satisfying this condition, then $ \widehat{A}_{\pi}\sigma t^{\varepsilon} (\WD{n} \times \WD{m}) = \widehat{A}_{\pi} \sigma' t^{\varepsilon'} (\WD{n} \times \WD{m}) $ if and only if $ \sigma = \sigma' $ and $ \varepsilon = \varepsilon' $. 

Consequently, the double coset formula allows us to rewrite $ \rho^{\WD{n+m}}_{\widehat{A}_{\pi}}(x \odot y) $ as the sum
\[
 \sum_{\substack{I \subseteq \{1,\dots,r\}\\ \sum_{i \in I} 2^{k_i} = n}} \left( c_{\sigma_I}^* \rho^{\WD{n} \times \WD{m}}_{\widehat{A}_{\pi_I} \times \widehat{A}_{\pi_J}}(x \otimes y) \otimes + c_{\sigma_I}^* (c_{s_0}^* \otimes c_{s_0}^*) \rho^{\WD{n} \times \WD{m}}_{\widehat{A}_{\pi_I}^{s_0} \otimes \widehat{A}_{\pi_J}^{s_0}} \right),
\]
where $ \sigma_I $ is the unique $ (n,m) $-shuffle satisfying $ \sigma(\{1,\dots,n\}) = \bigcup_{i \in I} P_{i,\pi} $ and $  J = \{1,\dots,r\} \setminus I $.
The statement follows by observing that $ c^*_{\sigma_I} = \tau_{I,J} $ and that $ c_{s_0}^* \rho^{\WD{l}}_{\widehat{A}_{\pi'}^{s_0}} = \rho^{\WD{l}}_{\widehat{A}_{\pi'}} \iota $ for all $ l \geq 1 $ and $ \pi' $ admissible partition of $ l $.

The case of $ \widehat{A}_{\pi}^{s_0} $ where $ 1,2 \notin \pi $ is done similarly. If $ 1 \in \pi $, the same argument holds, but if $ \exists i \in I $ and $ \exists j \in J $ such that $ k_i = k_j = 0 $, then, interpreting the elements of $ \WD{n+m} $ as signed permutations, $ (p_i,-p_i)(p_j,-p_j) $ belongs to $ \widehat{A}_{\pi} $ but not to $ (\sigma_I t^{\varepsilon})(\WD{n} \times \WD{m})(\sigma_I t^{\varepsilon})^{-1} $, where $ P_{i,\pi} = \{p_i\} $ and $ P_{j,\pi} = \{p_j\} $. Thus, we need to restrict the summation only to partitions $ \{1,\dots,r\} = I \sqcup J $ in which all the occurrences of $ 1 $ in $ \pi $ belong to the same part.
\end{proof}

This result provides a way to detect the charge of a Hopf monomial via restriction to maximal elementary abelian $ 2 $-subgroups. We first fix preliminary notation.
\begin{definition}\label{def:subgroup decomposition}
With the notation of Theorem \ref{thm:subgroups D}, write $ H^*(\widehat{A}_{(2)};\Ftwo) = \Ftwo[x,y] $ and let $ z = x+y $. Let $ H_{A_{(2)}}^+ $ (respectively $ H_{A_{(2)}}^- $) be the vector subspace generated by elements of the form $ x^az^b $ where $ a > b $ (respectively $ b > a $).
If $ \pi $ is an admissible partition of $ n $, write $ \pi = (1)^{m_1} \cup (2)^{m_2} \cup \pi' $ where $ 1\notin \pi' $ and $ 2 \notin \pi' $.
For $ S \subseteq \{1,\dots,m_2\} $, we define $ H_{i,S} = H_{A_{(2)}}^+ $ if $ i \notin S $ and $ H_{A_{(2)}}^- $ if $ i \in S $.
Then we define:
\[
H_{\widehat{A}_\pi}^+ = \left\{ \begin{array}{ll}
0 & \mbox{if } m_1 > 0 \\
\bigoplus_{S \subset \{1,\dots,m_2\}, |S|=2k} \bigotimes_{i=1}^{m_2} H_{i,S} \otimes H^*(A_{\pi'}; \Ftwo) & \mbox{if } m_1 = 0
\end{array} \right. \text{ and}
\]
\[
H_{\widehat{A}_\pi}^- = \left\{ \begin{array}{ll}
0 & \mbox{if } m_1 > 0 \\
\bigoplus_{S \subset \{1,\dots,m_2\}, |S|=2k+1} \bigotimes_{i=1}^{m_2} H_{i,S} \otimes H^*(A_{\pi'}; \Ftwo) & \mbox{if } m_1 = 0
\end{array} \right.
\]
Moreover, if $ m_1 = m_2 = 0 $, we define $ H_{A_{\pi}^{s_0}}^+ = 0 $ and $ H_{A_{\pi}^{s_0}}^- = H^*(A_{\pi}^{s_0}; \Ftwo) $.
\end{definition}

\begin{proposition}\label{prop:restriction charge}
Referring to Definition \ref{def:subgroup decomposition}, we have that for every maximal elementary abelian $ 2 $-subgroup $ A = \widehat{A}_{\pi} $ or $ A = A_{\pi}^{s_0} $ of $ \WD{n} $, the restriction of a positively (respectively negatively) charged Hopf monomial in $ \mathcal{M}_D \cap H^*(\WD{n}; \Ftwo) $ to the cohomology of $ A_{(\pi)} $ belongs to $ H_A^+ $ (respectively $ H_A^- $).
\end{proposition}
\begin{proof}
Every positively charged gathered block $ b $ restricts to an element of $ H_A^+ $. Non-trivial computations arise only if $ b = (\delta_{2m:0}^0)^r (\gamma_{1,m}^+)^s $ with $ r \geq 0 $ and $ s > 0 $ and $ A = A_{(2)^m} $. In this case, with the notation of Theorem \ref{thm:subgroups D}, we observe that
\[
h_{m}^{2^k} = \sum_{\substack{ S \subseteq \left\{ 1, \dots, m \right\}\\ |S| = 2l }} \prod_{i \in S} z_i^{2^k} \prod_{j \notin S} x_j^{2^k},
\]
where $ z_i = x_i + y_i $. Thus $ h_m^{2^k} \in H_{(2)^m}^+ $. Let $ 2^k $ be the biggest power of $ 2 $ smaller than $ s $. $ h_m^{s-2^k} $ is a sum of pure tensors of the form $ w_1 \otimes \dots \otimes w_m $, where $ w_i $ is a monomial in $ x_i $ and $ z_i $ with total degree smaller than $ 2^k $. Therefore, $ h_m^s = h_m^{2^k} h_m^{s-2^k} $ still belongs to $ H^+_{(2)^m} $. The restriction of $ b $ to $ A_{(2)^m} $ is equal to $ \bigotimes_{i=1}^m (x_iz_i)^r h_m^s $, which belongs to $ H^+_{(2)^m} $ because multiplication by $ \bigotimes_{i=1}^m x_iz_i $ preserves $ H^+_{(2)^m} $.

We see the corresponding statement for negatively charged gathered blocks by noting that conjugation with $ s_0 $ exchanges $ H^+_A $ and $ H^-_{A^{s_0}} $.

In general, a positively (respectively negatively) charged Hopf monomial $ x $ is a transfer product of gathered blocks, all positively charged (respectively, all positively charged except one). Consequently, Proposition \ref{prop:transfer subgroups D} yields the statement for $ x $.
\end{proof}

We can finally complete our relations for $ A'_D $ by providing the proofs of the two leftover propositions of Subsection \ref{sec:RelAdd}.

\begin{proof}[Proof of Proposition \ref{lem:transfer mezzo}]
Let $ b $ be a positively charged gathered block in $ A_D $ and $ x \in A_D $. From Lemma \ref{lem:involution on structural maps} and the definition of $ \Delta' $ we deduce that $ \Delta(b) = \Delta'(b) + (\iota \otimes \iota) \Delta'(b) $, and that  $ \Delta(\iota(b)) = (\id \otimes \iota + \iota \otimes \id)\Delta'(b) $. During this proof, we assume, by convention, that $ x|_{A_{\pi}} = 0 $ when $ x \in H^*(\WD{n}; \Ftwo) $ and $ \pi $ is not an admissible partition of $ n $.
Let $ \pi = (2^{k_1},\dots,2^{k_r}) $ and $ \pi' = (2^{h_1},\dots,2^{h_s}) $ be admissible partitions of some integers.
From Proposition \ref{prop:transfer subgroups D}, we deduce that:
\begin{align*}
&\left[(\odot \otimes \odot)(\id \otimes \tau \otimes \id)(\Delta' \otimes \Delta)(b \otimes x)\right]|_{\widehat{A}_{\pi} \times \widehat{A}_{\pi'}} \\
&= \sum_{\substack{I \sqcup J = \{1,\dots,r\}\\ I' \sqcup J' = \{1,\dots,s\}}} \tau_{I,I',J,J'} \Big( \Delta'(b) \otimes \Delta(x) + (\id \otimes \iota)\Delta'(b) \otimes (\id \otimes \iota)\Delta(x) \\
&+ (\iota \otimes \id)\Delta'(b) \otimes (\iota \otimes \id)\Delta(x) + (\iota \otimes \iota)\Delta'(b) \otimes (\iota \otimes \iota)\Delta(x) \Big)|_{\widehat{A}_{\pi_I} \times \widehat{A}_{\pi'_{I'}} \times \widehat{A}_{\pi_J} \times \widehat{A}_{\pi'_{J'}}} \\
&= \sum_{\substack{I \sqcup J = \{1,\dots,r\}\\ I' \sqcup J' = \{1,\dots,s\}}} \tau_{I,I',J,J'} \Big( (\id + \iota \otimes \iota)\Delta'(b) \otimes \Delta(x) \\
&+ (\id + \iota \otimes \iota)\Delta'(\iota(b)) \otimes \Delta(\iota(x)) \Big)|_{\widehat{A}_{\pi_I} \times \widehat{A}_{\pi'_{I'}} \times \widehat{A}_{\pi_J} \times \widehat{A}_{\pi'_{J'}}}\\
&= \sum_{\overline{I} \sqcup \overline{J} = \{1,\dots,r+s\}} (b \otimes x + \iota(b) \otimes \iota(x))|_{\widehat{A}_{(\pi \sqcup \pi')_{\overline{I}}} \times \widehat{A}_{(\pi \sqcup \pi')_{\overline{J}}}} \\
&= \left[\Delta(b \odot x)\right]|_{\widehat{A}_{\pi} \times \widehat{A}_{\pi'}}
\end{align*}
In this equalities we used the identities of Lemma \ref{lem:involution on structural maps} to perform the substitutions $ (\iota \otimes \iota)\Delta(x) = \Delta(x) $ and $ (\id \otimes \iota) \Delta(x) = (\iota \otimes \id)\Delta(x) = \Delta(\iota(x)) $. $ \pi \sqcup \pi' $ is assumed to be $ (2^{k_1},\dots,2^{k_r},2^{h_1},\dots,2^{h_s}) $. $ I = \overline{I} \cap \{1,\dots,r\} $ and $ J = \overline{J} \cap \{1,\dots,r\} $, while $ I' $ and $ J' $ are $ \overline{I} \cap \{r+1,\dots,r+s\} $ and $ \overline{J} \cap \{r+1,\dots,r+s\} $ suitably shifted. The sum should be over all $ I,J,I',J' $ such that at least one between $ I $ and $ J $ does not contain an $ l $ such that $ k_l = 0 $ and at least one between $ I' $ and $ J' $ does not contain an $ l $ such that $ h_l = 0 $. However, since the restriction of positively charged gathered blocks is zero on elementary abelian $ 2 $-subgroups corresponding to admissible partitions containing $ 1 $, we can restrict the sum only to the terms for which $ \forall i \in I: k_i \not= 0 $ and $ \forall i' \in I': h_i \not= 0 $. This condition is equivalent to $ \overline{I} $ not containing $ 1 $, and we can, once again, restrict the last sum only to these terms and get the last equality.
\end{proof}

\begin{proof}[Proof of Proposition \ref{lem:relations cup}]
Using Proposition \ref{prop:restriction eas}, the newly proved Lemma \ref{lem:transfer mezzo}, Proposition \ref{prop:transfer subgroups D}, and the fact that cup products commute with restrictions, we check that the desired identity hold when restricted to maximal elementary abelian $ 2 $-subgroups. Then Theorem \ref{teo:Quillen iso} yields the relations in $ A_D $.
\end{proof}

\section{Proof of the main theorems} \label{sec:proof}

We devote this section to the proofs of the presentation theorems for $ A_B $ and $ A_D $. They will be proved by comparing restrictions to elementary abelian $ 2 $-subgroups and exploiting Theorem \ref{teo:Quillen iso}. We will separate two technical lemmas from the proofs for the sake of clarity of exposition.

We first provide a proof for our structure theorem for $ A_B $.

\begin{lemma} \label{lem:intersection abelian subgroups B}
Let $ k > 0 $. The kernel of the restriction map
\[
H^*(A_{(2^k)}; \Ftwo)^{N_{\WB{2^k}}(A_{(2^k)})} \to H^*(A_{(2^k)} \cap A_{(2^{k-1},2^{k-1})}; \Ftwo)
\]
is the ideal generated by $ d_{2^k-1} $.
\end{lemma}
\begin{proof}
From Swenson's description of $ A_{\pi} $, stated as in Proposition \ref{prop:subgroups B}, we can identify $ A_{(2^k)} $ with the image of the diagonal embedding $ \id \times d \colon \Sigma_2 \times V_{k-1} \to \Sigma_2 \wr V_{k-1} \rightarrow \WB{2^k} $. Its intersection with the product $ A_{(2^{k-1},2^{k-1})} = V_{k-1} \times V_{k-1} $ is identified with the subgroup $ V_{k-1} \subseteq \Sigma_2 \times V_{k-1} $, embedded diagonally in $ \WB{2^k} $.

The restriction to this subgroup maps $ f_{2^k} $ to $ (f_{2^{k-1}})^2 $, $ d_{2^k - 2^l} $ to $ (d_{2^{k-1} - 2^{l-1}})^2 $ if $ l > 0 $, and $ d_{2^k - 1} $ to $ 0 $. This is known, but we sketch a proof for completeness. If we chose bases $ \{x,y_1,\dots,y_k\} $ of $ H^1(A_{(2^k)}; \Ftwo) $ and $ \{x,y_1,\dots,y_{k-1}\} $ of $ H^1(A_{(2^{k-1})}; \Ftwo) $ as in Subsection \ref{subsec:res B}, we have that the restriction is given by $ x \mapsto x $, $ y_i \mapsto y_i $ if $ 1 \leq i < k $ and $ y_k \mapsto 0 $. The polynomial $ F_k(t) = \prod_{v \in H^1(V_k; \Ftwo)} (t+v) $ in $ H^*(V_k; \Ftwo)[t] $ restricts to $ (F_{k-1}(t))^2 $.
Since $ f_{2^k} = F_k(x) $, we deduce the formula for $ f_{2^k} $. The identities for $ d_{2^k-2^l} $ are obtained from this by using the classical identity $ F_k(t) = \sum_{i=0}^k t^{2^i} d_{2^k-2^i} $.
\end{proof}

\begin{proof}[Proof of Theorem \ref{teo:Bn}]
Let $ A'_B $ be the Hopf ring generated by $ \gamma_{k,m} $ and $ \delta_m $ with the desired relations. Since the relations mentioned above hold in $ A_B $, there exists an obvious morphism $ \varphi \colon A'_B \rightarrow A_B $.

We need to fix a total ordering $ \leq $ on the set $ \mathcal{P}_n $ of admissible partitions of $ n $ such that, for all $ \pi,\pi' \in \mathcal{P}_n $,  $ \pi' > \pi $ if $ \pi' $ is a refinement of $ \pi $. In other words, $ \leq $ extends the partial ordering given by refinement.
Let $ b $ be a non-trivial gathered block in $ A_B $. There exist unique non-negative integers $ n,m $ such that $ b = \prod_{i=1}^n \gamma_{i,2^{n-i}m}^{a_i} \delta_{2^nm}^{a_0} $ with $ a_n \not= 0 $. We consider the partition of $ 2^nm $ $ \pi_b = \left( 2^n, \dots, 2^n \right) $.
Given $ x = b_1 \odot \dots \odot b_r \in \mathcal{M}_B $, let $ \pi_x = \sqcup_{i=1}^r \pi_{b_i} $.
As a consequence of Proposition \ref{prop:restrizione gamma}, Propositions \ref{prop:restrizione delta} and \ref{lem:transfer abelian subgroups}, $ x|_{A_{\pi}} \not= 0 $ implies that $ \pi_x > \pi $.
Explicitly, if $ b = \prod_{i=1}^n \gamma_{i,2^{n-i}m}^{a_i} \delta_{2^nm}^{a_0} $, we have:
\[
	b|_{A_{\pi_b}} = \left( f_{2^n}^{a_0} \prod_{i=1}^n d_{2^n-2^{n-i}}^{a_i} \right)^{\otimes^m}
\]
For any $ x = b_1 \odot \dots \odot b_r \in \mathcal{M}_B $, $ x|_{A_{\pi_n}} $ is the symmetrization of $ \bigotimes_{i=1}^r b_i|_{A_{\pi_{b_i}}} $.
Given a partition $ \pi $, let $ \mathcal{M}_{\pi} $ be the set of elements $ x \in \mathcal{M}$ such that $ \pi_x = \pi $.

We first prove that $ \varphi $ is injective. We proceed by contradiction, and we assume that there exists a non-trivial sum $ \sum_i x_i $ of elements of $ \mathcal{M}_B $ that is $ 0 $ when restricted to every elementary abelian $ 2 $-subgroup.
Let $ \pi $ be maximal among the set of partitions $ \left\{ \pi_{x_i} \right\}_i $.
Since, by the explicit calculation above, the restrictions of the elements of $ \mathcal{M}_{\pi} $ to $ A_{\pi} $ are linearly independent, this gives a contradiction.

To prove surjectivity, it is sufficient, by Theorem \ref{teo:Quillen iso}, to prove that an element $ \alpha $ of the Quillen group $ \mathcal{F}^*_{\WB{n}} $ can be written as the image via $ q_{\WB{n}} $ of a linear combination of elements of $ \mathcal{M}_B $.
Note that such $ \alpha $ is determined by its values $ \alpha_{\pi} $ on the maximal abelian $ 2 $-subgroups $ A_{\pi} $.
Let $ \overline{\pi}_{\alpha} = \max \left\{ \pi \in \mathcal{P}_n: \alpha_{\pi} \not= 0 \right\} $ with respect to the chosen linear ordering. We write $ \overline{\pi}_{\alpha} = \left( 2^{k_1}, \dots, 2^{k_r} \right) $.
We proceed by induction on $ \overline{\pi}_{\alpha} $.
$ \alpha_{\overline{\pi}_{\alpha}} $ must be invariant with respect to the action of the normalizer $ N_{\WB{n}}(A_{\overline{\pi}_{\alpha}}) $. By Swenson's description of these invariant subalgebras stated in Proposition \ref{prop:subgroups B}, it is a sum of elements $ \sum_i c_{i,1} \otimes \dots \otimes c_{i,r} $, with $ c_{i,j} = \prod_{l=1}^{k_j-1} d_{2^{k_j}-2^{k_j-l}}^{a_{i,j,l}} f_{2^{k_j}}^{a_{i,j,0}} \in H^* \left( A_{(2^{k_j})}; \Ftwo \right) $. We must have $ a_{i,j,k_j} \not= 0 $ for all $ i $ and $ j $. Otherwise, we can define a partition $ \pi' $ obtained from $ \overline{\pi}_{\alpha} $ by substituting $ 2^{k_j} $ with two parts both equal to $ 2^{k_j-1} $ and observe that, by Lemma \ref{lem:intersection abelian subgroups B}, we must have $ \alpha_{\overline{\pi}_\alpha}|_{A_{\overline{\pi}_{\alpha}} \cap A_{\pi'}} \not= 0 $. Thus $ \alpha_{\pi'} \not= 0 $ and this would contradict the maximality of $ \overline{\pi}_{\alpha} $.
By our calculations above, since $ \alpha_{\overline{\pi}_{\alpha}} $ must be invariant by permutations of tensor factors, this condition guarantees the existence of an element $ x $ in the linear span of $ \mathcal{M}_{\overline{\pi}_{\alpha}} $ such that $ x|_{A_{\overline{\pi}_\alpha}} = \alpha_{\overline{\pi}_{\alpha}} $.
This reduces the statement to $ \alpha' = \alpha + q_{\WB{n}}(x) $ for which, by construction, $ \overline{\pi}_{\alpha'} < \overline{\pi}_{\alpha} $, and completes the induction argument.
\end{proof}

We now focus on the presentation of $ A_D $.

\begin{lemma} \label{lem:intersection abelian subgroups D}
Let $ \mathcal{M}_2 \subseteq \mathcal{M}_D $ be the set of Hopf monomials in $ A'_D $ whose constituent gathered blocks are all of the form $ (\delta_{2k:0}^0)^r (\gamma_{1,k}^{\pm})^s $ with $ r \geq 0 $ and $ s > 0 $, or of the form $ (\delta_{2:0}^0)^a $ with $ a \geq 0 $.
Then, for all $ m \geq 0 $, $ \mathcal{M}_2 \cap H^*(\WD{2m}; \Ftwo) $ restricts to a linearly independent set in $ H^*(\widehat{A}_{(2)^m}; \Ftwo) $. Moreover, the image of $ \mathcal{M}_2 $ in the cohomology of $ \widehat{A}_{(2)^m} $ generates the kernel of the restriction $ \rho_{2,1} \colon H^*(\widehat{A}_{(2)^m}; \Ftwo)^{N_{\WD{2m}}(\widehat{A}_{(2)^m})} \to H^*(\widehat{A}_{(1)^4 \cup (2)^{m-2}} \cap \widehat{A}_{(2,2)}; \Ftwo) $.
\end{lemma}
\begin{proof}
Note that, due to Theorem \ref{thm:subgroups D} and Proposition \ref{prop:restriction eas}, the Hopf monomials in $ \mathcal{M}_2 \cap H^*(\WD{2}; \Ftwo) $ restrict to linearly independent elements in $ H^*(\widehat{A}_{(2)}; \Ftwo) $.
Therefore, to prove the linear independence claim for $ m > 1 $, it is enough to check that the restrictions of the elements of $ \mathcal{M}_2 \cap H^*(\WD{2m}; \Ftwo) $ to $ H^*(\WD{2}^m; \Ftwo) $ (which is a component $ \Delta_{(2)^m} $ of the coproduct) are linearly independent.
Let $ \mathcal{F} $ be the weight filtration on $ A'_D $ provided by Definition \ref{def:filtration D}. It is enough to prove that this set is linearly independent when working in the associated graded spaces $ \gr_{\mathcal{F}}(A'_D) $ and $ \gr_{\mathcal{F}}(H^*(\WD{2}^m; \Ftwo)) $.
In this setting, the image of a gathered block $ b_{l,\underline{t}}^+ \in \mathcal{M}_2 $ (respectively $ b_{l,\underline{t}}^- \in \mathcal{M}_2 $) under $ \gr_{\mathcal{F}}(\Delta_{(2)^m}) $ is $ \sum_{\varepsilon_1,\dots,\varepsilon_l} \bigotimes_{i=1}^m b_{1,\underline{t}}^{\varepsilon_i} $, where the sum is over all $ l $-tuples $ (\varepsilon_1,\dots,\varepsilon_l) $ with $ \varepsilon_i \in \{+,-\} $ and the cardinality of the set $ \{i: 1 \leq i \leq l, \varepsilon_i = - \} $ is even (respectively odd).
Combining this with Proposition \ref{lem:transfer mezzo}, we check the claim directly.

By Propositions \ref{prop:restriction eas} and \ref{prop:transfer subgroups D}, every element of $ \mathcal{M}_2 $ restricts to $ 0 $ on $ \widehat{A}_{\pi} $ whenever $ 1 \in \pi $. Therefore, it is contained in the kernel of $ \rho_{2,1} $.
We now prove the opposite inclusion.
With the notation of Theorem \ref{thm:subgroups D}, we write $ H^*(A_{(2)^m}; \Ftwo)^{N_{\WD{2m}}(A_{(2)^m})} = (\Ftwo[f_2,d_1]^{\otimes m})^{\Sigma_m} \{1,h_m\} $. We note that $  A_{(2)^m} \cap \widehat{A}_{(1)^4 \cup (2)^{m-2}} = A_{(2)^m} \cap A_{(1)^4 \cup (2)^{m-2}} $. Moreover, $ h_m = \gamma_{1,m}^+|_{A_{(2)^m}} $ is $ 0 $ when restricted to $ A_{(2)^m} \cap \widehat{A}_{(1)^4 \cup (2)^{m-2}} $. Therefore, Lemma \ref{lem:intersection abelian subgroups B} implies that $ \ker(\rho_{2,1}) $ is the ideal generated by $ h_m $ and $ d_1^{\otimes m} = \rho_{2,1}(\gamma_{1,m}^+ + \gamma_{1,m}^-)|_{A_{(2)^m}} $.
Finally, the generators belong to the image of $ \mathcal{M}_2 $, the linear subspace generated by $ \mathcal{M}_2 $ is a $ \cdot $-subalgebra by our formulas in $ A'_D $, and restriction maps preserve cup products. These remarks complete the proof.
\end{proof}

\begin{proof}[Proof of Theorem \ref{teo:Dn}]
Let $ A''_D $ be the almost Hopf ring generated by elements of the form $ \delta_{n:m}^0 $, $ \gamma_{k,m}^\pm $, and $ 1^- $ with the desired relations.
Let $ \varphi \colon A''_D \to A'_D $ be the obvious morphism. We also consider the $ \Ftwo $-vector space $ A'''_D $ with basis $ \mathcal{M}_D $. By our relations for $ A'_D $, $ \mathcal{M}_D $ generates $ A''_D $. Thus, there is a surjective linear map $ \varphi' \colon A'''_D \to A''_D $. To prove that $ \varphi' $ is an isomorphism it is enough to prove that $ \varphi'' = \varphi' \varphi $ is so. Since in component $ 0 $ this is obvious, we can consider only positive components and replace $ A'_D $ with $ A_D $. For technical reasons, we consider the set $ \mathcal{M}'_D $, which differ from $ \mathcal{M}_D $ by replacing neutral gathered blocks with elements of the form $ \rho\left(\prod_{i=2}^n (\delta_i \odot 1_{n-i})^{k_i} \right) $ for $ k_2,\dots, k_n \geq 0 $. As shown in Lemma \ref{lemma:0}, this corresponds, at the level of $ A'''_D $, to perform a change of basis. Hence, it does not affect the argument.

We adapt the argument used in the proof of Theorem \ref{teo:Bn}. We define $ \pi_x $ for $ x \in \mathcal{M}'_D $ as we did for $ A_B $, with the only difference that gathered blocks of the form $ b = (\delta_{2:0}^0)^m $ have $ \pi_b = (2) $, because $ (1,1) $ does not define a maximal elementary abelian subgroup in $ \WD{2} $. It is still true that $ x|_{\widehat{A}_{\pi}} = 0 $ unless $ \pi_x $ is a refinement of $ \pi $. We extend refinement of admissible partitions to a total ordering $ \leq $, and we use the same argument by induction on $ \leq $ adopted for $ A_B $.
Our choice of the new basis $ \mathcal{M}'_D $ makes evident that for all admissible partitions $ \pi $ the set $ \mathcal{M}'_\pi = \{ x \in \mathcal{M}'_D: \pi_x = \pi \} $ restricts to a linearly independent set in the cohomology of $ \widehat{A}_{\pi} $ when $ 1 \in \pi $, and Lemma \ref{lem:intersection abelian subgroups D} guarantees that this is true if $ 1 \notin \pi $ and $ 2 \in \pi $. Hence, the injectivity part works verbatim.
We need to adapt the surjectivity argument for admissible partitions $ \pi $ such that $ 1 \notin \pi $ and $ 2 \in \pi $ (in all other cases, nothing changes). In these cases, we use Lemma \ref{lem:intersection abelian subgroups D} instead of Lemma \ref{lem:intersection abelian subgroups B} to carry on the proof.
\end{proof}

\section{Steenrod algebra action}

This section is devoted to the calculation of the Steenrod algebra action on $ A_B $ and $ A_D $.
We first observe that, since the coproducts and transfer products are induced by (stable) maps, they satisfy a Cartan formula with respect to Steenrod squares. In other words, $ A_B $ and $ A_D $ are almost-Hopf rings over the Steenrod algebra.
Thus it is sufficient to determine the action of the Steenrod squares on the generators $ \delta_{2^n} $, $ \gamma_{k,2^n} $, $ \delta_{n:m}^0 $, and $ \gamma_{k,2^n}^\pm $.

\begin{definition}[from \cite{Sinha:12}]
We define the following notions:
\begin{itemize}
\item the \emph{height} ($\height$) of a gathered block in $ A_B $ or $ A_D $ is the number of generators that are cup-multiplied to obtain it, and the height of a Hopf monomial $ x = b_1 \odot \dots \odot b_r $ is $ \max_{i=1}^r \height \left( b_i \right) $
\item the \emph{effective scale} ($\effsc$) of a gathered block in the cohomology of $ \WB{n} $ (respectively $ \WD{n} $) is the least $ l $ such that $ n/2^l $ is an integer and its restriction to $ \WB{2^l}^{n/2^l} $ (respectively $ \WD{2^l}^{n/2^l} $) is non-zero, and the effective scale of a Hopf monomial $ x = b_1 \odot \dots \odot b_r $ as $ \min_{i=1}^r \effsc \left( b_i \right) $
\item a \emph{full-width} monomial is a Hopf monomial in $ A_B $ (respectively $ A_D $) of which no constituent block is of the form $ 1_{\WB{n}} $ (respectively $ 1_{\WD{n}} $)
\end{itemize}
\end{definition}

\begin{theorem}{\rm cf \cite[Theorem 8.3, page 191]{Sinha:12}}$\mspace{9mu}$ \label{teo:Steenrod B}
Let $ k,n \geq 1 $ and $ i \geq 0 $. Then, in $ A_B $, the following formulas hold:
\begin{itemize}
\item $ \Sq^i \left( \gamma_{k,2^n} \right) $ is the sum of all the full-width monomials $ x \in \mathcal{M}_B $ of degree $ 2^{n+k} - 2^n + i $ with $ \height \left( x \right) \leq 2 $ and $ \effsc \left( x \right) \geq k $ in which generators of the form $ \delta_k $ do not appear
\item $ \Sq^i \left( \delta_{2^n} \right) $ is the sum of all the full-width monomials $ x \in \mathcal{M}_B $ of degree $ 2^n + i $ with $ \height \left( x \right) \leq 2 $ and $ \effsc \left( x \right) \geq 1 $ such that a generator of the form $ \delta_k $ appears in every constituent gathered block of $ x $
\end{itemize}
\end{theorem}

\begin{proof}
The calculation for $ \Sq^i \left( \gamma_{k,2^n} \right) $ is an obvious consequence of \cite[Theorem 8.3, page 191]{Sinha:12}.
As regarding $ \Sq^i \left( \delta_{2^n} \right) $, since $ \delta_{2^n} $ is the top-dimensional Stiefel-Whitney class of the reflection representation $ U_{2^n} $ by Proposition \ref{prop:geometry generators}, by Wu's formula $ \Sq^i \left( \delta_{2^n} \right) = w_i \left( U_{2^n} \right) \delta_{2^n} $.
Defining, by convention, $ \gamma_{k,0} = 1 $, let:
\[
	u_i = \sum_{ \substack{j_0, \dots,j_n \geq 0, \sum_{r=1}^{n-1} 2^r j_r + j_n + j_0 = 2^n \\ \sum_{r=1}^{n-1} \left( 2^r -1 \right) j_r + j_n = i}} \bigodot_{r=1}^{n-1} \gamma_{r,j_r} \odot \delta_{r_n} \odot 1_{\WB{j_0}}
\]
We computed the restriction of $ w_i \left( U_{2^n} \right) $ to the maximal elementary abelian subgroups $ A_\pi $ in the proof of Proposition \ref{prop:restrizione delta}. It coincides with the restriction of $ u_i $ by our previous calculations based on Proposition \ref{prop:restrizione gamma}. Thus:
\[
	\Sq^i \left( \delta_{2^n} \right) = w_i \left( U_{2^n} \right) \delta_{2^n} = u_i \delta_{2^n}
\]
and this class is exactly the sum of all the desired Hopf monomials $ x $.
\end{proof}

As regarding the calculation of the Steenrod squares on the generators of $ A_D $, We observe that the calculation for $ \Sq^i \left( \delta_{n:m} \right) $ is implicit in Theorem \ref{teo:Steenrod B} since $ \delta_{n:m} = \rho \left( \delta_n \odot 1_m \right) $ and $ \rho $ commute with Steenrod operations.
Thus we only need to consider generators of the form $ \gamma_{k,n}^\pm $.

\begin{theorem} \label{teo:Steenrod D}
Let $ k,n \geq 1 $ and $ i \geq 0 $. Then, in $ A_D $, $ \Sq^i ( \gamma_{k,n}^+ ) $ (respectively $ \Sq^i ( \gamma_{k,n}^- ) $) is the sum of all the full-width monomials $ x \in \mathcal{B}^+ $ (respectively $ x \in \mathcal{B}^- $) of degree $ 2^{n+k} - 2^n + i $ with $ \height \left( x \right) \leq 2 $ and $ \effsc \left( x \right) \geq l $ in which generators of the form $ \delta_{n:m} $ do not appear.
\end{theorem}

\begin{proof}
We recall that Definition \ref{def:subgroup decomposition} provides, for all maximal elementary abelian $ 2 $-subgroup $ A \subseteq \WD{2n} $, subspaces $ H^+_A $ and $ H^-_A $ of the cohomology of $ A $.
A direct calculation shows that $ \Sq^i \left( h_n \right) \in H^+_{(2)^n} $. Since restrictions preserve the Steenrod squares, $ \Sq^i(\gamma_{1,n}^+) $ is mapped to an element of $ H^+_A $ for all maximal elementary abelian $ 2 $-subgroup $ A \subseteq \WD{2n} $ and all choices of $ i $ and $ n $. Similarly, the restriction of $ \Sq^i(\gamma_{1,n}^-) $ to every such subgroup $ A $ lies in $ H^-_A $.
Let $ x^+ $ (respectively $ x^- $) be the sum of all the positively (respectively negatively) charged Hopf monomials considered in the statement.
By Proposition \ref{prop:restriction charge}, the restriction of $ x^+ $ (respectively $ x^- $) belongs to $ H^+_A $ (respectively $ H^-_A $).

Moreover $ \Sq^i ( \gamma_{k,n}^+ ) + \Sq^i ( \gamma_{k,n}^- ) = \rho \left( \Sq^i ( \gamma_{k,n} ) \right) $.
Consequently, Theorem \ref{teo:Steenrod B} implies that $ \Sq^i ( \gamma_{k,n}^+ ) + \Sq^i ( \gamma_{k,n}^- ) = x^+ + x^- $.

Since $ H^+_A \cap H^-_A = 0 $ for all $ A $, the two facts observed above guarantee that $ \Sq^i (\gamma_{k,n}^+ ) = x^+ $ and $ \Sq^i (\gamma_{k,n}^-) = x^- $.
\end{proof}

\appendix
\section{Independent proof of Proposition \ref{prop:basis D}}

This appendix contains an independent proof of the fact that $ \mathcal{M}_D $ is an additive basis for the almost-Hopf ring over $ \Ftwo $ presented as in Theorem \ref{teo:Dn}.

We consider the vector space $ V $ over $ \Ftwo $ with linear basis $ \mathcal{M}_D $. There is an obvious linear epimorphism from $ V $ to this almost-Hopf ring. To prove that this is an isomorphism, it is enough to verify that our basic relations extend to an almost-Hopf ring structure on $ V $.

The weight filtration of Definition \ref{def:filtration D} extends to a filtration on $ V $ preserved by this linear map.
We define linear maps $ \odot, \cdot \colon V \otimes V \to V $ and $ \Delta \colon V \to V $ providing a Hopf ring structure satisfying the relations of Theorem \ref{teo:Dn}, and making this linear map a morphisms of almost-Hopf rings. This condition implies that $ V \to A'_D $ is an isomorphism because $ A'_D $ is presented by such generators and relations.

We construct such operations recursively on weight. In weight $ 0 $, we define the structural morphisms as follows. For all $ x \in \mathcal{M}_B \subseteq A_B $, if $ x $ contains a gathered block in $ \tilde{\mathcal{B}}^0 $, then we define $ \Delta(x^0) = \sum_i x_i^{'0} \otimes x_i^{''0} $, where $ \Delta(x) = \sum_i x_i' \otimes x_i'' $ in $ A_B $, with the convention that $ (1_0)^0 = 1^+ + 1^- $. If $ x \in \tilde{\mathcal{B}}^c $, then we define $ \Delta(x^\pm) = \sum_i ( x_i^{'+} \otimes x_i^{''\pm} + x_i^{'-} \otimes x_i^{''\mp}) $, where $ \Delta(x) = \sum_i x_i' \otimes x_i'' $ in $ A_B $.
If $ x,y \in \mathcal{M}_B $ we can write in $ A_B $, using the Hopf monomial basis, $ x \odot y = \lambda z $ for some $ \lambda \in \Ftwo $ and $ z \in \mathcal{M}_B $. If at least a gathered block appearing in $ x $ belongs to $ \tilde{\mathcal{B}}^0 $, the same is true for $ z $ (or the coefficient $ \lambda $ is $ 0 $), and we define $ x^0 \odot y^{\pm}  = \lambda z^0 $, or $ x^0 \odot y^0 = 0 $ depending on whether $ y \in \tilde{\mathcal{B}}^c $ or not. Otherwise, define $ x^{\pm} \odot y^{\pm} $ as $ \lambda z^+ $ or $ \lambda z^- $, where we choose the sign according to the multiplication rule.
Define the $ \cdot $ product of Hopf monomials with the same charge, or of a charged and a neutral Hopf monomials, via the same rule described in $ A_B $ by giving the resulting elements of $ \mathcal{M}_D $ the same charges of the factor. Let the $ \cdot $ product of Hopf monomials with opposite charges be $ 0 $. The corresponding morphisms in $ A_B $ are so, and the relations of Theorem \ref{teo:Dn} not involving $ \gamma_{1,n}^\pm $ are satisfied. Hence, the maps defined this way are (co)associative and (co)commutative. The almost-Hopf ring axioms and the identity of Proposition \ref{lem:transfer mezzo} on the weight $ 0 $ part follow from the Hopf ring structure of $ A_B $ and the definition of the structural maps.

Then, let $ n \geq 0 $ and assume that $ \odot $ and $ \cdot $ are defined on the subspace of $ V \otimes V $ generated by elements $ x \otimes y \in \mathcal{M}_D \otimes \mathcal{M}_D $ such that $ w(x) + w(y) \leq n $ and $ \Delta $ is defined on the subspace of $ V $ generated by elements $ x \in \mathcal{M}_D $ of weight $ w(x) \leq n $. To perform the recursion step, we fix the lexicographic ordering on profiles of charged gathered blocks (indeed, any total order on profiles works). We also let $ \pi \colon V \to V $ be the map that sends every $ x^0 \in \mathcal{B}^0 $ to $ 0 $, every $ x^- \in \mathcal{B}^- $ to $ 0 $, and every $ x^+ \in \mathcal{B}^+ $ to itself.
We define the coproduct $ \Delta(x) $ for $ x \in \mathcal{M}_D $ such that $ w(x) = n+1 $ as follows:
\begin{itemize}
\item If $ x = b^{\pm} $ is a gathered block (necessarily charged) with profile $ (t_0,\dots,t_k) $, $ k > 1 $, we define $ \Delta(x) $ via the equation 2 of Lemma \ref{lem:transfer coproduct basis D}.
\item If $ x = b^\pm $ is a gathered block with profile $ \underline{t} = (c,d) $, then $ x = b_{k,(c,d)}^\pm = (\delta_{2k:0}^0)^c (\gamma_{1,k}^\pm)^d $ for some $ c \geq 0 $ and $ d > 0 $. We can assume that $ x $ is positively charged (we deduce the negatively charged case by applying $ 1^- \odot \_ $), and we define
\begin{align*}
\Delta(x) &= \sum_{l=0}^k \left( b_{l,(c,d)}^+ \otimes b_{k-l,(c,d)}^\pm + b_{l,(c,d)}^- \otimes b_{k-l,(c,d)}^\mp \right) \\
&+ \sum_{l=1}^{k-1}\sum_{a=1}^{d-1} \left( \begin{array}{c} d \\ a \end{array} \right) \Big( b_{l,(c,a)}^+ (\gamma_{1,l}^-)^{d-a} \otimes b_{k-l,(c,a)}^{\pm} (\gamma_{1,k-l}^\mp)^{d-a} \Big).
\end{align*}
The addends of the second summation, since identity 2 of Lemma \ref{lem:transfer coproduct basis D} are computed recursively.
\item Suppose $ x $ is not a gathered block. In that case, let $ b^+ $ being the smallest gathered block in the lexicographic order among those constituting $ x $ and in which a generator of the form $ \gamma_{1,k}^\pm $ appears with a non-zero exponent. We write $ x = b^+ \odot x' $, and we define $ \Delta(x) = (\odot \otimes \odot) (\pi \otimes \tau \otimes \id) (\Delta \otimes \Delta) (b^{\pm} \otimes x') $ using Proposition \ref{lem:transfer mezzo}, where $ \Delta(x') $ is defined because $ w(x') \leq n $ and $ \Delta(b^\pm) $ is defined either by recursion or as in the previous point.
\end{itemize}

$ x \odot y $ for $ w(x) + w(y) = n+1 $ is defined as follows:
\begin{itemize}
\item If $ x $ and $ y $ are both neutral Hopf monomials, we let $ x \odot y = 0 $.
\item If $ x $ is a neutral gathered block and $ y = b_{i,\underline{t}}^{\pm} $, we let $ x \odot y = x \odot b_{i,\underline{t}}^+ $, coherently with identity 4 of the Lemma above.
\item If $ x $ and $ y $ are charged gathered blocks that have the same profile $ \underline{t} = (t_0,t_1,\dots,t_n) $ with $ n>1 $, we define $ x \odot y $ by imposing identity 3 of Lemma \ref{lem:transfer coproduct basis D}.
\item If $ x $ and $ y $ are charged gathered blocks with the same profile $ \underline{t} = (t_0,t_1) $, it is enough, due to equation 5 of Lemma \ref{lem:transfer coproduct basis D}, to consider the case in which $ x $ and $ y $ are positively charged. Then $ x = b_{k,\underline{t}}^+ = (\delta_{2k:0}^0)^{t_0} (\gamma_{1,k}^+)^{t_1} $ and $ y = b_{l,\underline{t}}^+ = (\delta_{2l:0}^0)^{t_0} (\gamma_{1,l}^+)^{t_1} $, and we define
\begin{align*}
&x \odot y = \left( \begin{array}{c} k+l \\ k \end{array} \right) b_{k+l,\underline{t}}^+ \\
&+ \gamma_{1,k+l}^+ \cdot \left(b_{k,(t_0,t_1-1)}^+ \odot b_{l,(t_0,t_1-1)}^+ - \left( \begin{array}{c} k+l \\ k \end{array} \right) b_{k+l,(t_0,t_1-1)}^+ \right) \\
&+ \left(b_{k,(t_0,t_1-2)}^+ \cdot (\delta_{2:0}^0 \odot (\gamma_{1,k-1}^+)^2) \right) \odot \left( b_{l,(t_0,t_1-2)}^+ \cdot (\delta_{2:0}^0 \odot (\gamma_{1,l-1}^+)^2) \right),
\end{align*}
with the convention that $ b_{a,(u_0,u_1)}^+ = 0 $ if $ u_1 = 0 $.
In the second row, the term $ b_{k,(t_0,t_1-1)}^+ \odot b_{l,(t_0,t_1-1)}^+ - \left( \begin{array}{c} k+l \\ k \end{array} \right) b_{k+l,(t_0,t_1-1)}^+ $ is written recursively as a linear combination in $ \mathcal{M}_D $. Since identity 3 of Lemma \ref{lem:transfer coproduct basis D} holds modulo lower weight, we can compute the resulting cup product with $ \gamma_{1,k+l}^+ $ in $ F_n $, where it is already defined. Similarly, we calculate the expression of the last addend recursively.
\item In all the other cases, we proceed similarly to $ A_B $. We merge charged gathered blocks in $ x \odot y $ having the same profile using the inductive formulas of the previous points and use the linearity of $ \odot $ to write the result as a linear combination of elements of $ \mathcal{M}_D $.
\end{itemize}

$ x \cdot y $ for $ w(x) + w(y) = n+1 $ can be computed as follows:
\begin{itemize}
\item If both $ x $ and $ y $ are gathered blocks with the same charge (positive or negative), then $ x \cdot y $ is naturally a gathered block.
\item If $ x = b_{k,\underline{t}}^+ $ and $ y = b_{h,\underline{u}}^- $ are gathered blocks with opposite charges, of profile $ \underline{t} = (t_0,\dots,t_a) $ and $ \underline{u} = (u_0,\dots,u_b) $ respectively, then we proceed as follows. We let $ x \cdot y = 0 $ if $ a > 1 $ or $ b > 1 $. If $ a = 1 $ and $ b = 1 $, then we impose relation 2 of Proposition \ref{lem:relations cup}, and we let
\[
x \cdot y = b_{k,(t_0,t_1-1)}^+ b_{k,(u_0,u_1-1)}^- \left( (\gamma_{k-1,1}^+)^2 \odot \delta_{2:0}^0 \right),
\]
where we calculate the product recursively.
\item If $ x $ and $ y $ are both gathered blocks, of which $ x $ neutral and $ y $ charged, define $ x \cdot y $ by writing $ y = y' \cdot \gamma_{k,1}^\pm $, and letting $ x \cdot y = (x \cdot y') \cdot (\gamma_{k,1}^\pm) $, where $ x \cdot y' $ is computed recursively. Its $ \cdot $ product with $ \gamma_{k,1}^\pm $ is determined either by one of the previous points or by imposing relation $ 5 $ of Proposition \ref{lem:relations cup} (for products with neutral gathered blocks).
\item If either $ x $ or $ y $ is not a single gathered block, then $ x \cdot y $ is reduced to lower weight calculations or the previous points by imposing Hopf ring distributivity.
\end{itemize}

These maps satisfy by construction all the relations of Theorem \ref{teo:Dn}, possibly except Proposition \ref{lem:transfer mezzo}. Note that our definitions are compatible with all these relations. For instance, $ \Delta(\gamma_{1,n}^+) \Delta(\gamma_{1,n}^-) = \Delta( (\gamma_{1,n-1}^+)^2 \odot \delta_{2:0}^0) $.

By induction on $ n $, it is possible to deduce the validity of the almost-Hopf ring axioms and the equation of Proposition \ref{lem:transfer mezzo} on the part of weight $ \leq n $. For instance:
\begin{itemize}
\item $ \odot $ is commutative. Moreover, it is associative on elements of $ \mathcal{M}_D $ that do not possess charged gathered blocks with the same profile.
\item The statement of Proposition \ref{lem:transfer mezzo} can be proved as follows: consider a charged gathered block $ b $ and an element $ x \in \mathcal{M}_D $. Write $ x = b' \odot x' $ where $ b' $ is the smallest (with respect to the lexicographic order on profiles) gathered block in $ x $ containing a generator $ \gamma_{1,k}^\pm $. If $ x $ does not include a gathered block with the same profile as $ b $, then the profile of $ b $ is either smaller or bigger than the profile of $ b' $. In the first case, the identity holds by definition, in the second case, from the almost-Hopf ring axioms on the part of lower weight and the recursive definition of the coproduct in $ V $, we deduce that both $ \Delta(b \odot x) $ and $ (\odot \otimes \odot) (\pi \otimes \tau \otimes \id) (\Delta \otimes \Delta)(b \otimes x) $ are equal to
\[
(\odot^{(3)} \otimes \odot^{(3)}) \tau_{(2453)} \Delta^{\otimes 3}(b \otimes b' \otimes x),
\]
where $ \tau $ indicates permutation of tensor factors and $ \odot^{(3)} \colon V^{\otimes 3} \to V $  denotes the iterated transfer product (which is commutative and associative on the relevant classes by induction hypothesis and by the previous point).

If $ x $ contains a gathered block with the same profile as $ b $, then, by induction on the number of transfer factors appearing, we reduce the computation to the case in which $ x $ is itself a gathered block. In this particular case, non-trivial calculations only arise when the profile of $ b $ is of the form $ \underline{t}= (t_0,t_1) $, and we can further assume that both $ x $ and $ b $ are positively charged.
A direct calculation shows that
\[
\Delta (\gamma_{1,k}^+) \cdot \Delta( b_{k,(t_0,t_1-1)}^+) = \Delta(b_{k,\underline{t}}^+),
\]
where we shortened $ (\cdot \otimes \cdot) (\id \otimes \tau \otimes \id) \Delta^{\otimes 2} $ with $ \Delta \cdot \Delta $.
Hence, from the almost-Hopf ring axioms applied to lower weight elements of $ V $ and from the definition of $ \Delta $ and $ \odot $ in $ V $, we deduce that
\begin{align*}
\Delta(b_{k,\underline{t}} \odot b_{l,\underline{t}}) &= \Delta(\gamma_{1,k+l}^+) \cdot \Delta(b_{k,(t_0,t_1-1)}^+ \odot b_{l,(t_0,t_1-1)}^+) \\
&+ \Delta ( b_{k,(t_0,t_1-1)}^+ \gamma_{1,k}^- \odot b_{l,(t_0,t_1-1)}^+ \gamma_{1,l}^- ).
\end{align*}
Moreover, we note that, by induction, Remark \ref{rem:transfer coproduct on basis} holds in $ V $ for elements of total weight at most $ n $, using a map $ \tilde{\Delta} \colon \bigoplus_{a+b=n} F_a \otimes F_b \to \tilde{V} $ defined as explained there ($ F_* $ denotes the weight filtration on $ V $).
Therefore, it is enough to check that
\begin{align*}
\tilde{\Delta}(b_{k,\underline{t}} \otimes b_{l,\underline{t}}) &= \Delta(\gamma_{1,k+l}^+) \cdot \tilde{\Delta}(b_{k,(t_0,t_1-1)}^+ \odot b_{l,(t_0,t_1-1)}^+) \\
&+ \tilde{\Delta} ( b_{k,(t_0,t_1-1)}^+ \gamma_{1,k}^- \otimes b_{l,(t_0,t_1-1)}^+ \gamma_{1,l}^- ).
\end{align*}
This is, once again, verified with a direct computation using the inductive identity
\[
\Delta(b_{k,(t_0,t_1-1)}^+ \gamma_{1,k}^-) = \Delta(b_{k,(t_0,t_1-2)}^+) \cdot \Delta(\gamma_{1,k}^+ \gamma_{1,k}^-) = \Delta(b_{k,(t_0,t_1-1)}^+) \cdot \Delta(\gamma_{1,k}^-)
\]
arising from the fact that $ \gamma_{1,k}^+ \gamma_{1,k}^- $ has weight less than $ 2k $.
\item The coproduct $ \Delta $ is associative and commutative. This claim follows from a similar argument by induction on the number of $ \odot $-factors on a Hopf monomial $ x \in \mathcal{M}_D $.
\item $ \cdot $ is commutative on gathered blocks by definition. In general, $ \cdot $ is defined by imposing Hopf ring distributivity and $ \Delta $ is cocommutative. Therefore, $ \cdot $ is commutative on all Hopf monomials. Moreover, the associativity property $ (x \cdot y) \cdot z = x \cdot ( y \cdot z) $ holds when $ x $, $ y $, and $ z $ are charged gathered blocks or powers of $ \delta_{2k:0}^0 $. When they have all the same charge, the statement is obvious. In the other cases, both the triples cup products are $ 0 $ unless only generators of the form $ \delta_{2k:0}^0 $ and $ \gamma_{1,k}^\pm $ appear as factors. The definition of the mixed-charged cup products, in this case, immediately allows reducing the identity to lower weight calculations. Similarly, a direct computation shows that $ \Delta(x) \cdot \Delta(y) = \Delta(x \cdot y) $ when $ x $ and $ y $ are charged gathered blocks or powers of $ \delta_{2k:0}^0 $.
\item By definition, the Hopf ring distributivity identity $ x \cdot (y \odot z) = \sum_i x_i'y \odot x_i''z $, with $ \Delta(x) = \sum_i x_i' \odot x_i'' $, holds if $ y $ and $ z $ do not have a common profile among their constituent gathered blocks and do not both belong to $ \mathcal{B}^0 $. Therefore, it is enough to check that this statement is true only when $ y $ and $ z $ are gathered blocks, either both neutral or both charged with the same profile (in this case, we can assume that they are both positively charged up to applying $ 1^- \odot \_ $).
Furthermore, suppose $ x $ is a transfer product of at least two gathered blocks. In that case, an application of Proposition \ref{lem:transfer mezzo} (whose statement in $ V $ is proved above), rephrased as in Remark \ref{rem:transfer coproduct on basis}, reduces the desired equality to lower weight calculations.

Hence, we can assume without loss of generality that $ x $, $ y $, and $ z $ are gathered blocks, with $ y $ and $ z $ either both neutral or having the same profile.
If  $ y $ and $ z $ are neutral, $ x \cdot (y \odot z) = 0 $. $ \sum_i x_i'y \odot x_i''z $ is zero either because the transfer product of neutral gathered blocks is zero (if $ x $ is neutral) or because both $ \odot $ and $ \Delta $ are $ (1^- \odot \_) \otimes (1^- \odot \_) $-invariant, and cup product with a neutral gathered block preserves the action of the involution $ 1^- \odot \_ $ (and thus each addend appears twice if $ x $ is charged).
Otherwise, due to our definition of $ \cdot $ on $ V $, non-trivial calculations arise only if $ y = b_{k,\underline{t}}^+ $ and $ z = b_{l,\underline{t}}^+ $ for a profile of the form $ \underline{t} = (t_0,t_1) $.
We first note that the identity holds when $ x = \gamma_{1,k+l}^+ $ by construction (this is a rephrasing of the recursive definition of $ b_{k,(t_0,t_1+1)}^+ \odot b_{l,(t_0,t_1+1)}^+ $). Similarly, from the definition of $ b_{k,\underline{t}}^+ \odot b_{l,\underline{t}}^+ $ we deduce that
\begin{align*}
&\delta_{2(k+l):0}^0 (b_{k,\underline{t}}^+ \odot b_{l,\underline{t}}^+) = \delta_{2(k+l):0}^0 \left[\gamma_{1,k+l}^+ (b_{k,(t_0,t_1-1)}^+ \odot b_{l,(t_0,t_1-1)}^+) \right] \\
&+ \delta_{2(k+l):0}^0 \left( b_{k,(t_0,t_1-2)}^+ (\delta_{2:0}^0 \odot (\gamma_{1,k-1}^+)^2) \odot b_{l,(t_0,t_1-2)}^+ (\delta_{2:0}^0 \odot (\gamma_{1,l-1}^+)^2) \right).
\end{align*}
In the cases where $ x $ is a positively charged gathered block of profile $ \underline{u} = (u_0,u_1) $ or a power of $ \delta_{2(k+l):0}^0 $, these facts reduce, by a second induction argument on $ t_1 $, the desired identity to lower weight calculations. Note that the implementation of the induction step here requires the properties of $ \cdot $ observed in the previous point. The base case of this induction is $ \underline{t} = (0,1) $, i.e. $ y = \gamma_{1,k}^+ $ and $ z = \gamma_{1,l}^+ $. The statement in this particular case reduces to the previous point and the particular cases $ x = \gamma_{1,k+l}^+ $ and $ x = \delta_{2(k+l):0}^0 $, by splitting $ x $, in all other cases, as the cup product of two smaller factors and argue inductively.

If $ x = \gamma_{1,k+l}^- $, then the definition of $ b_{k,\underline{t}}^+ \odot b_{l,\underline{t}}^+ $, combined with the relation $ \gamma_{1,k+l}^- \gamma_{1,k+l}^+ = \delta_{2:0}^0 \odot (\gamma_{1,k+l-1}^+)^2 $, reduces once again the calculation to elements of lower weight in $ V $.
The statement also holds trivially if $ x = \gamma_{r,\frac{2(k+l)}{2^r}}^\pm $ for some power of $ 2 $ dividing $ 2(k+l) $, with $ r \geq 2 $. These facts imply, by a second induction argument on the number of cup product factors in $ x $, the desired identity if $ x $ is a charged gathered block or a power of $ \delta_{2(k+l):0}^0 $.
Finally, if $ x $ is a neutral gathered block different from a power of $ \delta_{2(k+l):0}^0 $, writing down the definition of $ b_{k,\underline{t}}^+ \odot b_{l,\underline{t}}^+ $ reduced the equality to previous cases and Hopf ring distributivity for lower weight elements.
\item The associativity of $ \cdot $ holds, as already observed, for charged gathered blocks. We reduce the associativity for all gathered blocks, possibly neutral, to lower weight calculations. We achieve this through Hopf ring distributivity, the definition of the cup product between a charged and a neutral gathered block, and relation 5 of Proposition \ref{lem:relations cup}. We deduce the general case from Hopf ring distributivity.
\item The compatibility between $ \cdot $ and $ \Delta $ making $ (V,\cdot,\Delta) $ a bialgebra can be checked on gathered blocks thanks to Hopf ring distributivity. We already noticed that it holds for charged gathered blocks. The statement for all gathered blocks, possibly neutral, is obtained inductively by writing down the definition and using the fact that relation 5 of Proposition \ref{lem:relations cup} holds.
\end{itemize}

\clearpage

\bibliographystyle{plain}
\bibliography{bibliografia}

\begin{thebibliography}{10}

\bibitem{Adem-Milgram}
Alejandro Adem and R.~James Milgram.
\newblock {\em Cohomology of finite groups}, volume 309 of {\em Grundlehren der
  Mathematischen Wissenschaften [Fundamental Principles of Mathematical
  Sciences]}.
\newblock Springer-Verlag, Berlin, second edition, 2004.

\bibitem{Bjorner-Brenti}
Anders Bj\"{o}rner and Francesco Brenti.
\newblock {\em Combinatorics of {C}oxeter groups}, volume 231 of {\em Graduate
  Texts in Mathematics}.
\newblock Springer, New York, 2005.

\bibitem{Davis}
Michael~W. Davis.
\newblock {\em The geometry and topology of {C}oxeter groups}, volume~32 of
  {\em London Mathematical Society Monographs Series}.
\newblock Princeton University Press, Princeton, NJ, 2008.

\bibitem{Salvetti:00}
C.~De~Concini and M.~Salvetti.
\newblock Cohomology of {C}oxeter groups and {A}rtin groups.
\newblock {\em Math. Res. Lett.}, 7(2-3):213--232, 2000.

\bibitem{Feshbach:02}
Mark Feshbach.
\newblock The mod 2 cohomology rings of the symmetric groups and invariants.
\newblock {\em Topology}, 41(1):57--84, 2002.

\bibitem{Friedmann-Medina-Sinha}
Greg Friedman, Anibal Medina-Mardones, and Dev Sinha.
\newblock Geometric cohomology.
\newblock {\em in preparation}.

\bibitem{Sinha:12}
Chad Giusti, Paolo Salvatore, and Dev Sinha.
\newblock The mod-2 cohomology rings of symmetric groups.
\newblock {\em J. Topol.}, 5(1):169--198, 2012.

\bibitem{Sinha:13}
Chad Giusti and Dev Sinha.
\newblock Fox-{N}euwirth cell structures and the cohomology of symmetric
  groups.
\newblock In {\em Configuration spaces}, volume~14 of {\em CRM Series}, pages
  273--298. Ed. Norm., Pisa, 2012.

\bibitem{Sinha:17}
Chad Giusti and Dev Sinha.
\newblock Mod-two cohomology rings of alternating groups.
\newblock {\em J. Reine Angew. Math.}, 772:1--51, 2021.

\bibitem{Guerra:17}
Lorenzo Guerra.
\newblock Hopf ring structure on the mod $p$ cohomology of symmetric groups.
\newblock {\em Algebr. Geom. Topol.}, 17(2):957--982, 2017.

\bibitem{Hepworth:16}
Richard Hepworth.
\newblock Homological stability for families of {C}oxeter groups.
\newblock {\em Algebr. Geom. Topol.}, 16(5):2779--2811, 2016.

\bibitem{Milnor-Stasheff}
John~W. Milnor and James~D. Stasheff.
\newblock {\em Characteristic classes}.
\newblock Princeton University Press, Princeton, N. J.; University of Tokyo
  Press, Tokyo, 1974.
\newblock Annals of Mathematics Studies, No. 76.

\bibitem{Papi}
Paolo Papi.
\newblock Inversion tables and minimal left coset representatives for {W}eyl
  groups of classical type.
\newblock {\em J. Pure Appl. Algebra}, 161(1-2):219--234, 2001.

\bibitem{Quillen:72}
D.~Quillen and B.~B. Venkov.
\newblock Cohomology of finite groups and elementary abelian subgroups.
\newblock {\em Topology}, 11:317--318, 1972.

\bibitem{Quillen:71}
Daniel Quillen.
\newblock The spectrum of an equivariant cohomology ring. {I}, {II}.
\newblock {\em Ann. of Math. (2)}, 94:549--572; ibid. (2) 94 (1971), 573--602,
  1971.

\bibitem{Strickland-Turner}
Neil~P. Strickland and Paul~R. Turner.
\newblock Rational {M}orava {$E$}-theory and {$DS^0$}.
\newblock {\em Topology}, 36(1):137--151, 1997.

\bibitem{Swenson}
James~Andrew Swenson.
\newblock {\em The mod-2 cohomology of finite {C}oxeter groups}.
\newblock ProQuest LLC, Ann Arbor, MI, 2006.
\newblock Thesis (Ph.D.)--University of Minnesota.

\bibitem{Hung:87}
Nguyen~H\ {u}'u Viet~Hu'ng.
\newblock The modulo {$2$} cohomology algebras of symmetric groups.
\newblock {\em Japan. J. Math. (N.S.)}, 13(1):169--208, 1987.

\bibitem{Vassiliev:92}
V.~A. Vassiliev.
\newblock {\em Complements of discriminants of smooth maps: topology and
  applications}, volume~98 of {\em Translations of Mathematical Monographs}.
\newblock American Mathematical Society, Providence, RI, 1992.
\newblock Translated from the Russian by B. Goldfarb.

\end{thebibliography}

\end{document}